\documentclass[11pt,reqno]{amsart} 

\usepackage{amssymb,amsfonts,amsmath,epsfig,yhmath}

 \allowdisplaybreaks
\numberwithin{equation}{section}

\font\script=rsfs10 at 11pt

\def\R{\mathbb R}

\def\B{\mathcal B}
\def\A{\mathcal A}
\def\D{\mathcal D}
\def\I{\mathcal I}
\def\DoubleS{\mathcal S}
\def\T{{\mbox{\script T}\,\,}}
\def\U{{\mbox{\script U}\,\,}}
\def\DoubleSS{{\mbox{\script S}\,\,}}
\def\BB{{\mbox{\script B}\,\,}}
\def\N{\mathbb N}
\def\W1p{{\mathrm W}^{1,p}}
\def\eps{\varepsilon}

\def\angle#1#2#3{#1\widehat{#2}#3}
\def\arc#1{\wideparen{#1}}
\def\u#1{\hbox{\boldmath $#1$}}

\def\proofof#1{\begin{proof}[Proof of~#1]}
\def\bal{\begin{aligned}}
\def\eal{\end{aligned}}
\def\lsegm#1{\ell\big(\u{#1}\big)}
\def\lcurve#1{\ell\big({\arc{\u{#1}}\big)}}
\def\uangle#1#2#3{\angle{\u{#1}}{\u{#2}}{\u{#3}}}
\def\CB{\widehat{\u{\B}}}
\def\tt{{\tilde\tau}}
\def\ee{\mathrm{e}}
\def\bu{$\u{A}_{10}\equiv \u{A}$}
\def\cu{$\u{B}_{10-12}\equiv \u{B}$}
\def\du{$\u{A}_{j-1}\equiv \u{A}_j$}
\def\eu{$\u{B}_{j-1}'$}

\def\case#1#2{\par\medskip \noindent{\underline{\it Case~#1.}}\emph{ #2}\par}
\def\part#1#2{\par\medskip \noindent{\underline{\it Part~#1.}}\emph{ #2}\par}
\def\bigstep#1#2{\bigskip\subsection{Step #1: #2}~\par\medskip}

\newcounter{mt}
\def\maintheorem#1#2#3{\par \medskip \noindent {\bf Theorem \mref{#1}}~(#2).~{\it #3}\par}
\def\mref#1{\Alph{#1}}
\def\maintheoremdeclaration#1{\stepcounter{mt}\newcounter{#1}\setcounter{#1}{\arabic{mt}}}

\maintheoremdeclaration{main}
\maintheoremdeclaration{maingen}

\newtheorem{theorem}{Theorem}[section]
\newtheorem{lemma}[theorem]{Lemma}
\newtheorem{definition}[theorem]{Definition}
\newtheorem{corollary}[theorem]{Corollary}
\newtheorem{remark}[theorem]{Remark}

\begin{document}

\title{A planar bi-Lipschitz extension Theorem}

\author{Sara Daneri}\address{Dipartimento di Matematica, via Ferrata 1, 27100 Pavia (Italy)}\email{sara.daneri@gmail.com}
\author{Aldo Pratelli}\address{Dipartimento di Matematica, via Ferrata 1, 27100 Pavia (Italy)}\email{aldo.pratelli@unipv.it}

\begin{abstract}
We prove that, given a planar bi-Lipschitz homeomorphism $u$ defined on the boundary of the unit square, it is possible to extend it to a function $v$ of the whole square, in such a way that $v$ is still bi-Lipschitz. In particular, denoting by $L$ and $\widetilde L$ the bi-Lipschitz constants of $u$ and $v$, with our construction one has $\widetilde L \leq C L^4$ (being $C$ an explicit geometrical constant). The same result was proved in 1980 by Tukia (see~\cite{Tukia}), using a completely different argument, but without any estimate on the constant $\widetilde L$. In particular, the function $v$ can be taken either smooth or (countably) piecewise affine.
\end{abstract}

\maketitle

\section{Introduction}

Given a set $C\subseteq \R^n$ and a function $u:C\to\R^n$, we say that $u$ is bi-Lipschitz with constant $L$ (or, shortly, $L$ bi-Lipschitz) if, for any $x\neq y\in C$, one has
\begin{equation}\label{biL}
\frac{1}{L} \, |y-x| \leq |u(y)- u(x)| \leq L |y-x|\,.
\end{equation}
Consider the following very natural question. If $u:C\to\R^n$ is bi-Lipschitz, is it true that there exists an extension $v: \R^n\to\R^n$ which is still bi-Lipschitz? Notice that, roughly speaking, we are asking whether the classical Kirszbraun Theorem holds replacing the Lipschitz condition with the bi-Lipschitz one. It is easy to observe that the answer to our question is, in general, negative. Indeed, let $C$ be the unit sphere plus its center $O$ and let $u$ be a function sending the sphere in itself via the identity, and $O$ in some point out of the sphere. Then, it is clear that all the continuous extensions of $u$ to the whole unit ball cannot be one-to-one. In fact, the real obstacle in this example is of topological nature. Therefore, one is lead to concentrate on the case in which $C$ is the boundary of a simply connected set. In particular, we will focus on the case in which the dimension is $n=2$, and $C=\partial \D$ is the boundary of the unit square $\D=(-1/2,1/2)^2$. In this case, to the best of our knowledge, the following first positive result was found in 1980~(\cite{Tukia}).
\begin{theorem}[Tukia]
Let $u:\partial \D\to \R^2$ be an $L$ bi-Lipschitz map. Then there exists an extension $v:\D\to \R^2$ which is also bi-Lipschitz, with constant $\widetilde L$ depending only on $L$. In particular, $v$ can be taken countably piecewise affine (that is, $\D$ is the locally finite union of triangles on which $v$ is affine).
\end{theorem}

Unfortunately, in the above result there is no explicit dependence of $\widetilde L$ on $L$, due to the fact that the existence of such $\widetilde L$ is obtained by compactness arguments. On the other hand, it is clear that in many situations one may need to have an explicit upper bound for $\widetilde L$. In particular, it would be interesting to understand whether the theorem may be true with $\widetilde L = L$, or at least $\widetilde L = C L$ for some geometric constant $C$. In this paper we prove that it is possible to bound $\widetilde L$ with $CL^4$. More precisely, our main result is the following.

\maintheorem{main}{bi-Lipschitz extension, piecewise affine case}{Let $u:\partial\D\to \R^2$ be an $L$ bi-Lipschitz and piecewise affine map. Then there exists a piecewise affine extension $v:\D\to\R^2$ which is $CL^4$ bi-Lipschitz, being $C$ a purely geometric constant. Moreover, there exists also a smooth extension $v:\D\to\R^2$, which is $C' L^{28/3}$ bi-Lipschitz.}

We can also extend the result of Theorem~\mref{main} to general maps $u$. Notice that, if $u$ is not piecewise affine on $\partial\D$, then of course it is not possible to find an extension $v$ which is (finitely) piecewise affine.
\maintheorem{maingen}{bi-Lipschitz extension, general case}{Let $u:\partial\D\to\R^2$ be an $L$ bi-Lipschitz map. Then there exists an extension $v:\D\to\R^2$ which is $C''L^4$ bi-Lipschitz, being $C''$ a purely geometric constant.}
Also in the general case, one may want the extending function $v$ to be either smooth or countably piecewise affine: we deal with this issue at the end of the paper, in Corollary~\ref{allast} and Remark~\ref{lastsec}. In particular, the constants $C, C'$ and $C''$ of Theorems~\mref{main} and~\mref{maingen} can be bounded as follows
\begin{align*}
C=636000\,, && C'= 70 C^{7/3}\,, && C'' = 81 C\,.
\end{align*}

Our proof of Theorem~\mref{main} is constructive and for this reason it is quite intricate. However, the overall idea is simple and we try to keep it as clear as possible.\par
The plan of the paper is the following. In Section~\ref{secpro} we briefly describe the construction that we will use to show Theorem~\mref{main}, and in Section~\ref{secnot} we fix some notation. Then, in Section~\ref{secmain} we give the proof of Theorem~\mref{main}. This section contains almost the whole paper, and it is subdivided in several subsections which correspond to the different steps of the proof. Finally, in Section~\ref{secmaingen} we show Theorem~\mref{maingen}, which follows from Theorem~\mref{main} thanks to an approximation argument.

\subsection{An overview of the proof of Theorem~\mref{main}\label{secpro}}

Let us briefly explain how the proof of Theorem~\mref{main} works. Given a bi-Lipschitz function $u:\partial\D\to\R^2$, its image is the boundary $\partial\Delta$ of a bounded Lipschitz domain $\Delta\subseteq\R^2$ (since $u$ is piecewise affine, in particular $\Delta$ is a polygon). Then, the extension must be a bi-Lipschitz function $v:\D\to\Delta$.\par\smallskip

First of all (Step~I) we determine a ``central ball'' $\widehat{\u\B}$, which is a suitable ball contained in $\Delta$ and whose boundary touches the boundary of $\Delta$ in some points $\u{A}_1,\, \u{A}_2,\, \dots ,\, \u{A}_N$, being $N\geq 2$. The image through $v$ of the central part of the square $\D$ will eventually be contained inside this central ball.\par
For any two consecutive points $\u{A}_i,\, \u{A}_{i+1}$ among those just described, we consider the part of $\Delta$ which is ``beyond'' the segment $\u{A}_i\u{A}_{i+1}$ (by construction, this segment lies in the interior of $\Delta$). We call ``primary sectors'' these regions, and we give the formal definition and study their main properties in Step~II. It is to be observed that the set $\Delta$ is the disjoint union of these primary sectors and of the ``internal polygon'' having the points $\u{A}_i$ as vertices (see Figure~\ref{Fig:ske} for an example).\par
We start then considering a given sector, with the aim of defining an extension of $u$ which is bi-Lipschitz between a suitable subset of the square $\D$ and this sector. In order to do so, we first give a method (Step~III) to partition a sector in triangles. Then, using this partition, for any vertex $\u{P}$ of the boundary of the sector we define a suitable piecewise affine path $\gamma$, which starts from $\u{P}$ and ends on a point $\u{P'}$ on the segment $\u{A}_i\u{A}_{i+1}$ (Step~IV). We also need a bound on the lengths of these paths, found in Step~V.\par
Then we can define our extension. Basically, the idea is the following. Take any point $P\in\partial\D$ such that $\u{P}:=u(P)$ is a vertex of $\partial\Delta$ inside our given sector. Denoting by $O$ the center of the square $\D$, we send the first part of the segment $PO$ of the square (say, a suitable segment $PP'\subseteq PO$) onto the path $\gamma$ found in Step~IV, while the last part $P'O$ of $PO$ is sent onto the segment connecting $\u{P}'$ with a special point $\u{O}$ of the central ball $\widehat{\u{\B}}$ (in most cases $\u{O}$ will be the center of $\widehat{\u{\B}}$). Unfortunately, this method does not work if we simply send $PP'$ onto $\gamma$ at constant speed; instead, we have to carefully define speed functions for all the different vertices $\u{P}$ of the sector, and the speed function of any point will affect the speed functions of the other points. This will be done in Step~VI.\par
At this stage, we have already defined the extension $v$ of $u$ on many segments of the square, thus it is easy to extend $v$ so to cover the whole primary sectors. To define formally this map, and in particular to check that it is $CL^4$ bi-Lipschitz, is the goal of Step~VII. Finally, in Step~VIII, we put together all the maps for the different primary sectors and fill also the ``internal polygon'', still keeping the bi-Lipschitz property. The whole construction is done in such a way that the resulting extending map $v$ is piecewise affine. Hence, to conclude the proof of Theorem~\mref{main}, we will only have (Step~IX) to show the existence of a smooth extension $v$. This will be obtained from the piecewise affine map thanks to a recent result by Mora-Corral and the second author in~\cite{MP}, see Theorem~\ref{carlos}.

\subsection{Notation\label{secnot}}

In this short section, we briefly fix some notation that will be used throughout the paper, and in particular in the proof of Theorem~\mref{main}, Section~\ref{secmain}. We list here only the notation which is common to all the different steps.\par

We call $\D=(-1/2,1/2)^2$ the open unit square in $\R^2$, and $O=(0,0)$ its center. The function $u$ is a bi-Lipschitz function from $\partial\D$ to $\R^2$, and $L$ is a bi-Lipschitz constant, according to~(\ref{biL}). The image $u(\partial\D)$ is a Jordan curve in the plane, therefore it is the boundary of a bounded open set, that we call $\Delta$. Notice that an extension $v$ as required by Theorems~\mref{main} and~\mref{maingen} must necessarily be such that $v(\D) = \Delta$.\par
The points of $\overline{\D}$ will be always denoted by capital letters, such as $A,\, B,\, P,\, Q$ and so on. On the other hand, points of $\overline{\Delta}$ will be always denoted by bold capital letters, such as $\u A,\, \u B,\, \u P,\, \u Q$ and similar. To shorten the notation and help the reader, whenever we use the same letter for a point in $\partial \D$ and (in bold) for a point in $\partial\Delta$, say $P\in\partial\D$ and $\u{P}\in\partial\Delta$, this always means that $u(P)=\u{P}$. Similarly, whenever the same letter refers to a point $P$ in $\D$ and (in bold) to a point $\u{P}$ in $\Delta$, this always means that the extension $v$ that we are constructing is done in such a way that $v(P)=\u{P}$.\par
For any two points $P,\, Q\in \D$, we call $PQ$ and $\ell(PQ)$ the segment connecting $P$ and $Q$ and its length. In the same way, for any $\u{P},\,\u{Q}\in\Delta$, by $\u{PQ}$ and by $\lsegm{PQ}$ we will denote the segment joining $\u{P}$ and $\u{Q}$ and its length. Since $\Delta$ is not, in general, a convex set, we will use the notation $\u{PQ}$ only if the segment $\u{PQ}$ is contained in $\Delta$.\par
For any $P,\, Q\in \partial\D$, we call $\arc{PQ}$ the shortest path inside $\partial\D$ connecting $P$ and $Q$, and by $\ell(\arc{PQ})\in [0,2]$ its lenght. Notice that $\arc{PQ}$ is well-defined unless $P$ and $Q$ are opposite points of $\partial\D$. In that case, the length $\ell(\arc{PQ})$ is still well-defined, being $2$, while the notation $\arc{PQ}$ may refer to any of the two minimizing paths (and we write $\arc{PQ}$ only after having specified which one). Accordingly, given two points $\u{P}$ and $\u{Q}$ on $\partial\Delta$, we write $\arc{\u{PQ}}$ to denote the path $u\big(\arc{PQ}\big)$, which is not necessarily the shortest path between $\u{P}$ and $\u{Q}$ in $\partial\Delta$. Observe that, if $u$ is piecewise affine on $\partial\D$, then $\arc{\u{PQ}}$ is a piecewise affine path for any $\u{P}$ and $\u{Q}$ in $\partial\Delta$.\par
Given a point $P\in\R^2$ and some $\rho>0$, we will call $\B(P,\rho)$ the open ball centered at $P$ with radius $\rho$. Given three non-aligned points $P$, $Q$ and $R$, we will call $\angle PQR\in (0,2\pi)$ the corresponding angle. Sometimes, for the ease of presentation, we will write the value of angles in degrees, with the usual convention that $\pi = 180^\circ$.\par
Finally, we will extensively use the following concepts. The \emph{central ball $\widehat{\u{\B}}$} is introduced in Step~I, while the  \emph{sectors} and the \emph{primary sectors} are introduced in Step~II. Moreover, in Step~III a partition of a sector in triangles is defined, where the triangles are suitably partially ordered and each triangle has its \emph{exit side}.

\section{The proof of Theorem~\mref{main}\label{secmain}}

In this section, which is the most extensive and important part of the paper, we show Theorem~\mref{main}. The proof is divided in several subsections, to distinguish the different main steps of the construction.

\bigstep{I}{Choice of a suitable ``central ball'' $\widehat{\u{\B}}$}

Our first step consists in determining a suitable ball, that will be called ``central ball'', whose interior is contained in the interior of $\Delta$, and whose boundary touches the boundary of $\partial\Delta$. Before starting, let us briefly explain why we do so. Consider a very simple situation, i.e. when $\Delta$ is convex. In this case, the easiest way to build an extension $u$ as required by Theorem~\mref{main} is first to select a point $\u{O}=v(O)$ having distance of order at least $1/L$ from $\partial\Delta$, and then to define the obvious piecewise affine extension of $u$, that is, for any two consecutive vertices $P,\,Q\in\partial\,\D$ we send the triangle $OPQ$ onto the triangle $\u{O}\u{P}\u{Q}$ in the affine way. This very coarse idea does not suit the general case, because in general $\Delta$ can be very complicated and, a priori, there is no reason why the triangle $\u{O}\u{P}\u{Q}$ should be contained in $\overline{\Delta}$. Nevertheless, our construction will be somehow reminiscent of this idea. In fact, we will select a suitable point $\u{O}=u(O)\in\Delta$ in such ``central ball'' and we will build the image of a triangle like $OPQ$ as a ``triangular shape'', suitably defining the ``sides'' $\u{O}\u{P}$ and $\u{O}\u{Q}$ which will be, in general, piecewise affine curves instead of straight lines. Thanks to the fact that the ``central ball'' is a sufficiently big convex subset of $\Delta$, in a neighborhood of $\u{O}$ of order at least $1/L$ the construction will be eventually carried out as in the convex case (in Step~VIII).\par\smallskip

The goal of this step is only to determine a suitable ``central ball'' $\widehat{\u{\B}}$. The actual point $\u{O}$ will be chosen only in Step~VIII, and it will be in the interior of this ball --in fact, in most cases $\u{O}$ will be the center of $\widehat{\u{B}}$.

\begin{lemma}\label{lemma:center}
There exists an open ball $\widehat{\u{\B}}\subseteq\Delta$ such that the intersection $\partial\widehat{\u{\B}}\cap \partial\Delta$ consists of $N\geq 2$ points $\u{A}_1,\, \u{A}_2,\, \dots \u{A}_N$, taken in the anti-clockwise order on the circle $\partial\widehat{\u{\B}}$, and with the property that $\partial\D$ is the union of the paths $\arc{A_iA_{i+1}}$, with the usual convention $N+1\equiv 1$.
\end{lemma}
\begin{remark}\label{rem-check}
Before giving the proof of our lemma, some remarks are in order. First of all, since the ball $\widehat{\u{\B}}$ is contained in $\Delta$, then $\partial \Delta\cap\widehat{\u{\B}}=\emptyset$. As a consequence, the path $\partial\Delta$ meets all the points $\u{A}_i$ in the same order as $\partial\widehat{\u{\B}}$, hence also the points $A_i\in\partial\D$ are in the anti-clockwise order (since $u$ is orientation preserving). Hence, the thesis is equivalent to say that for each $i$, among the two injective paths connecting $A_i$ and $A_{i+1}$ on $\partial\D$, the anti-clockwise one is shorter than the other.\par
In addition, notice that from the thesis one has two possibilities. If $N=2$, then necessarily $\ell(A_1A_2)=2$, so that the two paths $\arc{A_1A_2}$ and $\arc{A_2A_1}$ have the same length. On the other hand, if $N\geq 3$, then it is immediate to observe that there must be two points $A_i$ and $A_j$, not necessarily consecutive, such that $\ell\big(\arc{A_iA_j}\big)\geq 4/3$. By the bi-Lipschitz property of $u$, this ensures that the radius of $\widehat{\u{\B}}$ is at least $\frac{2}{3L}$, since the circle $\partial\widehat{\u{\B}}$ contains two points having distance at least $\frac{4}{3L}$.\par
Finally notice that, given a ball $\B$ contained in $\Delta$ and such that $\partial\Delta\cap\partial\B$ contains at least two points, there is a simple method to check whether $\widehat{\u{\B}}=\B$ satisfies all the requirements of Lemma~\ref{lemma:center}. Indeed, this is easily seen to be true unless there is an arc of length $2$ in $\partial\D$ whose image does not contain any point of $\partial\Delta\cap\partial\B$.
\end{remark}

\proofof{Lemma~\ref{lemma:center}}
First of all, we define the symmetric set 
\begin{equation*}
 S=\Big\{(\u{A},\u{B})\subseteq\partial\Delta\times\partial\Delta:\,\u{A}\neq\u{B}  \text{ and }\exists \text{ a ball }\B\subseteq\Delta\text{ s.t. }\{\u{A},\u{B}\}\subseteq\partial\B\cap\partial\Delta\Big\}\,.
\end{equation*}
This set is nonempty, since for instance the biggest ball contained inside $\Delta$ contains at least two points of $\partial\Delta$ in its boundary. Since for any $\delta>0$ the set
\[
\big\{(\u{A},\u{B})\in S:\,\ell(\arc{{AB}})\geq \delta \big\}
\]
is compact, we can select a pair $(\u A, \u B)$ maximizing $\ell(\arc{{AB}})$. We then distinguish two cases. If $\ell(\arc{{AB}})=2$, then by Remark \ref{rem-check} any ball $\CB$ such that $\{\u{A},\u{B}\}\subseteq\partial\CB\cap\partial\Delta$ satisfies our claim.\par

Suppose then that $\ell(\arc{AB})<2$. Since by definition there are balls $\B\subseteq \Delta$ such that $\{\u{A},\u{B}\}\subseteq\partial\Delta\cap\partial\B$, we let $\CB$ to be one of such balls maximizing the radius. We will conclude the thesis by checking that $\CB$ satisfies all the requirements. In particular, we will make use of the following\par\medskip\noindent
{\underline{\it Claim.}}\emph{ There is some point $\u{P}\in \partial \CB\cap \partial \Delta \setminus \arc{\u{AB}}$.}\par

Let us first observe that the thesis readily follows from this claim, then we will show its validity. In fact, let $\u{P}$ be a point in $\partial \CB \cap \partial \Delta \setminus \arc{\u{AB}}$, and consider the three points $A,\, B$ and $P$ on $\partial\D$ and the corresponding paths $\arc{AB}$, $\arc{AP}$ and $\arc{BP}$. Since $P\not\in \arc{AB}$ by construction, by the maximality of $\ell(\arc{AB})$ we derive that $\arc{AP}$ does not contain $B$, and similarly $\arc{BP}$ does not contain $A$. Thus, $\partial\D$ is the (essentially disjoint) union of the three paths $\arc{AB}$, $\arc{AP}$ and $\arc{BP}$. But then, if we take any path of length $2$ in $\partial\D$, this intersects at least one between $A,\,B$ and $P$. Thanks to the last observation of Remark~\ref{rem-check}, this shows the thesis.\par
\begin{figure}[htbp]
\begin{center}
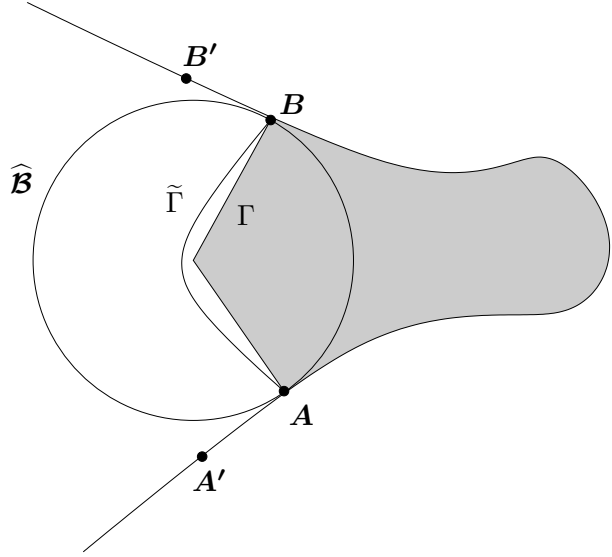\vspace{-10pt}
\caption{Geometric situation for the Claim in the proof of Lemma~\ref{lemma:center}.}\label{Fig:step1conclusion} 
\end{center}
\end{figure}
Let us now prove the claim. Call, as in Figure~\ref{Fig:step1conclusion}, $\u{A'}$ and $\u{B'}$ two points of $\partial \Delta$ sufficiently close to $\u{A}$ and $\u{B}$ respectively, so that $\arc{\u{A'B'}}\supseteq \arc{\u{AB}}$ (here we use the fact that $\ell(\arc{AB})<2$). Let now $\Gamma$ be the path connecting $\u{A}$ and $\u{B}$ obtained as union of the two radii of $\widehat{\u{\B}}$ passing through $\u{A}$ and $\u{B}$; moreover, let $\widetilde\Gamma$ be another path connecting $\u{A}$ and $\u{B}$ inside $\widehat{\u{\B}}$, close to $\Gamma$ but contained out of the subset of $\Delta$, coloured in the figure, having $\arc{\u{AB}} \cup \Gamma$ as its boundary. For any $\u{Q}\in\widetilde\Gamma$, consider a point $\u{R}\in\partial\Delta$ minimizing $\ell(\u{QR})$. By construction, $\u{R}$ cannot belong to the open path $\arc{\u{AB}}$; moreover, if we assume that the claim is false, then by continuity $\u{R}$ must belong either to $\arc{\u{AA'}}$ or to $\arc{\u{BB'}}$, provided that $\widetilde\Gamma$ is chosen sufficiently close to $\Gamma$. Of course, if $\u{Q}\in\widetilde\Gamma$ is close to $\u{A}$ (resp. $\u{B}$), then so is $\u{R}$. Therefore, by continuity, there exists some $\u{Q}\in\widetilde\Gamma$ for which there are two points $\u{R_A}$ and $\u{R_B}$ minimizing the length $\ell(\u{QR})$ within $\partial\Delta$, with $\u{R_A}\in \arc{\u{AA'}}$ and $\u{R_B}\in \arc{\u{BB'}}$. Let then $\u{\B'}$ be the ball centered in $\u{Q}$ and with radius $\ell(\u{QR_A})$. By definition, this ball is contained inside $\Delta$, then we have that $(\u{R_A},\u{R_B})\in S$. Moreover, since both $\u{R_A}$ and $\u{R_B}$ belong to $\arc{\u{A'B'}}$, then $\arc{\u{R_AR_B}}\supseteq \arc{\u{AB}}$, hence $\ell(\arc{R_AR_B})\geq \ell(\arc{AB})$. This gives a contradiction with the maximality of $\ell(\arc{AB})$, unless $\u{R_A}=\u{A}$ and $\u{R_B}=\u{B}$. But also in this case we have a contradiction, because $\u{\B'}$ is a ball contained in $\Delta$, having $\u{A}$ and $\u{B}$ in its boundary, and with radius strictly bigger than that of $\CB$. This shows the validity of the Claim, thus concluding the proof.
\end{proof}

\bigstep{II}{Definition and first properties of the ``sectors'' and of the ``primary sectors''}

In this step, we will give the definition of ``sectors'' of $\Delta$, we will study their main properties, and we will call some of them ``primary sectors''. Let us start with some notation.
\begin{definition}
Let $\u{A}$ and $\u{B}$ be two points in $\partial\Delta$ such that the open segment $\u{AB}$ is entirely contained in the interior of $\Delta$. Let moreover $\arc{\u{AB}}$ be, as usual, the image under $u$ of the shortest path connecting $A$ and $B$ on $\partial\D$ (or of a given one of the two injective paths, if $A$ and $B$ are opposite). We will call \emph{sector between $\u{A}$ and $\u{B}$}, and denote it as $\DoubleS(\u{AB})$, the subset of $\Delta$ enclosed by the closed path $\u{AB} \cup \arc{\u{AB}}$.
\end{definition}

\begin{remark}\label{sottosettori}
It is useful to notice what follows. If $\u{A},\,\u{B},\, \u{C},\,\u{D} \in \partial\Delta$, and $\u{C},\,\u{D} \in \arc{\u{AB}}$, then $\arc{\u{CD}}\subseteq \arc{\u{AB}}$. Moreover, if both the open segments $\u{AB}$ and $\u{CD}$ lie in the interior of $\Delta$, then one also has
\[
\DoubleS(\u{CD}) \subseteq \DoubleS(\u{AB})\,.
\]
\end{remark}

We observe now a very simple property, which will play a crucial role in our future construction, namely that the length of a shortest path in $\partial\D$ can be bounded by the length of the corresponding segment in $\Delta$.

\begin{lemma}
Let $\u{P},\,\u{Q}$ be two points in $\partial\Delta$ such that the segment $\u{P}\u{Q}$ is contained in $\overline\Delta$. Then one has
\begin{equation}\label{stimaL^2}
\ell\big(\arc{PQ}\big) \leq \sqrt{2} L\,\ell (\u{PQ})\,.
\end{equation}
\end{lemma}
\begin{proof}
The inequality simply comes from the Lipschitz property of $u$, and from the fact that $\D$ is a square. Indeed,
\[
\ell\big(\arc{PQ}\big) \leq \sqrt{2} \,\ell(PQ) \leq \sqrt{2} L\, \ell (\u{PQ} ) \,.
\]
\end{proof}

\begin{remark}
We observe that, of course, the estimate~(\ref{stimaL^2}) holds true because $\arc{PQ}$ is the shortest path between $P$ and $Q$ in $\partial\D$ (however, this does not necessarily imply that $\arc{\u{PQ}}$ is the shortest path between $\u{P}$ and $\u{Q}$ in $\partial\Delta$). The validity of the estimate \eqref{stimaL^2} is the reason why we had to perform the construction of Step~I so as to find points $A_j$ on $\partial\Delta$ such that each path $\arc{A_i A_{i+1}}$ does not pass through the other points $A_j$. 
\end{remark}

We can now define the ``primary sectors'', which are the sectors between the consecutive points $\u{A}_i$ given by Lemma~\ref{lemma:center}.
\begin{definition}
We call \emph{primary sector} each of the sectors $\DoubleS(\u{A}_i\u{A}_{i+1})$, being $\u{A}_j$ the points obtained by Lemma~\ref{lemma:center}.
\end{definition}
\begin{figure}[htbp]
\begin{center}
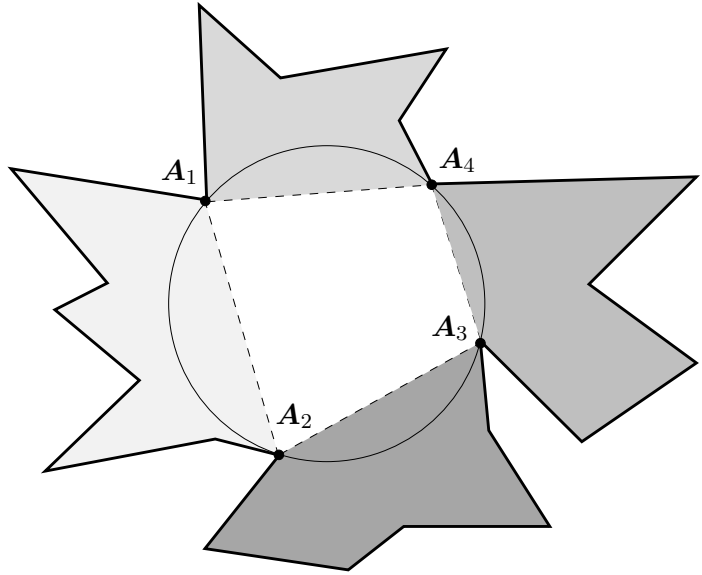\vspace{-10pt}
\caption{A set $\Delta$ with four (coloured) primary sectors.}\label{Fig:ske} 
\end{center}
\end{figure}
Notice that the above definition makes sense, because the points $\u{A}_i$ are all on the boundary of $\widehat{\u{\B}}$ and $\widehat{\u{\B}}$ does not intersect $\partial\Delta$, thus the open segments $\u{A}_i\u{A}_{i+1}$ are entirely contained in the interior of $\Delta$. Moreover, by the claim of Lemma~\ref{lemma:center} it follows that the sectors $\DoubleS(\u{A}_i\u{A}_{i+1})$ are essentially pairwise disjoint. The set $\Delta$ is thus the essentially disjoint union of the sectors $\DoubleS(\u{A}_i\u{A}_{i+1})$ and of the polygon $\u{A}_1\u{A}_2\dots \u{A}_N$, as Figure~\ref{Fig:ske} illustrates.

\bigstep{III}{Partition of a sector in triangles}

In view of the preceding steps, we aim to extend the function $u$ in order to cover a whole given sector. This extension of the function $u$, which is the main part of the proof, will be quite delicate and long, being the scope of the Steps~III--VII. Later on, in Step~VIII, we will use this result to cover all the primary sectors and we will have also to take care of the remaining polygon. In this step, we describe a method to partition a given sector in triangles. Let us then start with a technical definition.

\begin{definition}
Let $\DoubleS(\u{AB})$ be a sector, and let $\u{P},\, \u{Q}$ and $\u{R}$ be three points in $\arc{\u{AB}}$ such that the triangle $\u{PQR}$ is not degenerate and is contained in $\Delta$. We say that $\u{PQR}$ is an \emph{admissible triangle} if each of its open sides is entirely contained either in $\partial\Delta$, or in $\Delta\setminus\partial\Delta$. If $\u{PQR}$ is an admissible triangle, we say that $\u{PR}$ is its \emph{exit side} if $\arc{\u{PR}}=\arc{\u{PQ}}\cup\arc{\u{QR}}$.
\end{definition}
\begin{figure}[htbp]
\begin{center}
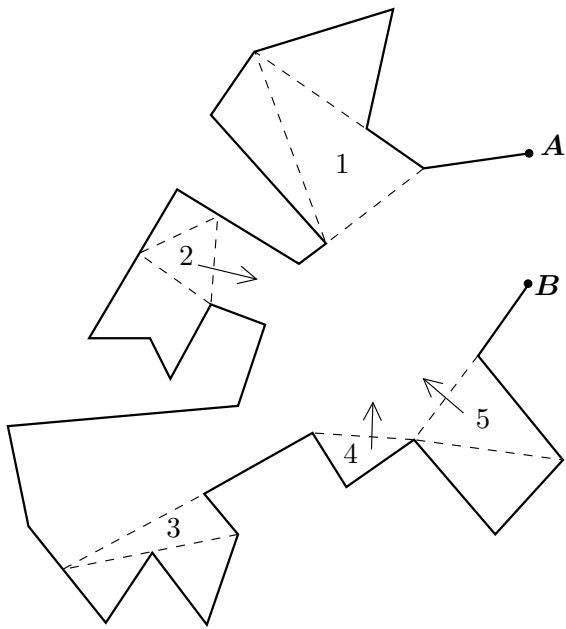\vspace{-10pt}
\caption{Some (admissible or not) triangles in a sector.}\label{Fig:admiss}
\end{center}
\end{figure}
Figure~\ref{Fig:admiss} shows a sector $\DoubleS(\u{AB})$, drawn in black, with five numbered triangles, having dotted sides. Triangles~1 and~3 are not admissible because they contain a side which is neither all contained in $\partial \Delta$, nor all in $\Delta\setminus\partial\Delta$, in particular triangle~1 has a side which is half in $\partial \Delta$ and half in $\Delta\setminus\partial\Delta$, while triangle~3 has a side which is all contained in $\Delta\setminus\partial\Delta$ except for a point. On the other hand, triangles~2, 4 and~5 are admissible, and an arrow indicates the exit side for each of them.

\begin{remark}
It is important to observe that each admissible triangle has exactly one exit side. As the figure shows, an admissible triangle can have all the three sides in the interior of $\Delta$, as triangle~2, or two, as triangle~5, or just one, as triangle~4. In any case, the exit side is always in the interior of $\Delta$.\par
It is also useful to understand the reason for the choice of the name. Consider a point $\u{T}\in \arc{\u{PR}}$, being $\u{PR}$ the exit side of the admissible triangle $\u{PQR}$, and consider the segment $TO$ which connects $T=u^{-1}(\u{T})$ to the center $O$ of the square $\D$. If $v:\D\to\Delta$ is an extension as required by Theorem~\mref{main}, then the image of the segment $TO$ under $v$ must be a path inside $\Delta$ which connects $\u{T}$ to $\u{O}$. This path must clearly exit from the triangle $\u{PQR}$ through the exit side $\u{PR}$.
\end{remark}

Before stating and proving the main result of this step we fix some further notation. Recall that $\Delta$ is a non-intersecting polygon obtained as the image of $\partial\D$ under $u$. Hence, $\partial\D$ is divided in a finite number of segments and $u$ is affine on each of these segments. We will then call \emph{vertex} on $\partial\D$ each extreme point of any of these segments. Therefore, the four corners of $\partial\D$ are of course vertices, but there are usually much more vertices. Correspondingly, we call \emph{vertex} on $\partial\Delta$ the image of each vertex on $\partial\D$. Thus, all the points of $\partial\Delta$ which are ``vertices'' in the usual sense of the polygon (i.e., corners), are clearly also vertices in our notation. However, there may be also other vertices which are not corners, hence which are in the interior of some segment contained in $\partial\Delta$. We will also call \emph{side} in $\partial \D$ or in $\partial\Delta$ any segment connecting two consecutive vertices on $\partial \D$ or on $\partial\Delta$. Hence, some of the segments which are sides of $\partial\Delta$ in the sense of polygons are in fact sides according to our notation, but there might be also some segments contained in $\partial\Delta$ which are not sides, but finite union of sides.\par
Finally, notice that it is admissible to add (finitely many!) new vertices to $\partial\D$ and then correspondingly to $\partial\Delta$. This means that we will possibly decide to consider some particular side as a union of two or more sides, thus increasing the total number of vertices: this is possible since of course $u$ is affine on each of those ``new sides''.
\begin{remark}\label{angle<1}
As an immediate application of this possibility of adding a finite set of new vertices, we will assume without loss of generality that for any two consecutive vertices $P$ and $Q$ in $\D$, one always has $\angle POQ \leq 1/50L$.
\end{remark}
We can finally state and prove the main result of this step.

\begin{lemma}\label{lemmastep3}
Let $\DoubleS(\u{AB})$ be a sector. There exists a partition of $\DoubleS(\u{AB})$ in a finite number of admissible triangles such that:
\begin{enumerate}
\item[{\bf a)}] each vertex in $\arc{\u{AB}}$ is vertex of some triangle of the partition,
\item[{\bf b)}] for each triangle $\u{PQR}$ of the partition, whose exit side is $\u{PR}$, the orthogonal projection of $\u{Q}$ on the straight line through $\u{PR}$ lies in the closed segment $\u{PR}$ (equivalently, the angles $\uangle PRQ$ and $\uangle RPQ$ are at most $\pi/2$).
\end{enumerate}
\end{lemma}

To show this result, it will be convenient to associate to any possible sector a number, which we will call ``weight''.

\begin{definition}
Let $\DoubleS(\u{AB})$ be a sector, and for any point $\u{P}\in\arc{\u{AB}}$ (different from $\u{A}$ and $\u{B}$) let us call $\u{P}_\perp$ the orthogonal projection of $\u{P}$ onto the straight line through $\u{AB}$. We will say that \emph{$\u{AB}$ ``sees'' $\u{P}$} if $\u{P}_\perp$ belongs to the closed segment $\u{AB}$ and the open segment $\u{PP}_\perp$ is entirely contained in the interior of $\Delta$. Let now $\omega$ be the number of sides of the path $\u{\arc{AB}}$. We will say that the \emph{weight of the sector $\DoubleS(\u{AB})$} is $\omega$ if $\u{AB}$ sees at least a \emph{vertex} $\u{P}$ in $\arc{\u{AB}}$. Otherwise, we will say that weight of $\DoubleS(\u{AB})$ is $\omega+\frac 12$.
\end{definition}

\begin{figure}[htbp]
\begin{center}
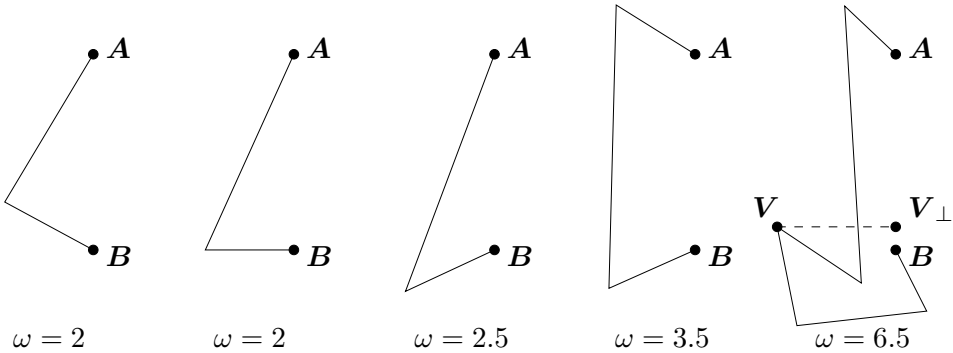\vspace{-10pt}
\caption{Some simple sectors and their weights.}\label{Fig:weight}
\end{center}
\end{figure}

In other words, the weight of any sector is an half-integer corresponding to the number of sides of the sector, augmented of a ``penalty'' $1/2$ in case that the segment $\u{AB}$ does not see any vertex of $\u{\arc{AB}}$. For instance, Figure~\ref{Fig:weight} shows some simple sectors and the corresponding weights. Notice that the last sector has non-integer weight because $\u{AB}$ does \emph{not} see the vertex $\u{V}$, since the segment $\u{VV}_\perp$ does not entirely lie inside $\Delta$. We now show a simple technical lemma, and then to pass to the proof of Lemma~\ref{lemmastep3}.

\begin{lemma}\label{lem:tech}
If the sector $\DoubleS(\u{AB})$ has a non-integer weight, then there exists a side $\u{A}^+\u{B}^-$ in $\arc{\u{AB}}$ such that $\u{AB}$ sees only points of the side $\u{A}^+\u{B}^-$.
\end{lemma}
\begin{proof}
First, notice that the property that we are going to show appears evident from the last three examples of Figure~\ref{Fig:weight}.\par
Let us now pass to the proof. For any point $\u{D}\in \u{AB}$, there exists exactly a point $\u{C}\in\arc{\u{AB}}$ such that $\u{AB}$ sees $\u{C}$ and $\u{C}_\perp = \u{D}$. This point is simply obtained by taking the half-line orthogonal to $\u{AB}$, starting from $\u{D}$ and going inside the sector: $\u{C}$ is the first point of this half-line which belongs to $\partial\Delta$, and in particular it belongs to $\arc{\u{AB}}$ by construction.\par
The proof is then concluded once we show that all such points $\u{C}$'s are on a same side of $\arc{\u{AB}}$. Indeed, if it were not so, there would clearly be some such $\u{C}$ which is a vertex, contradicting the fact that the sector has non-integer weight.
\end{proof}

\proofof{Lemma~\ref{lemmastep3}}
We will show the result by induction on the (half-integer) weight of the sector.\par
If $\DoubleS(\u{AB})$ has weight $2$, which is the least possible weight, then the two sides of the sector must be $\u{AC}$ and $\u{CB}$ for a vertex $\u{C}$. Moreover, $\u{AB}$ sees $\u{C}$, because otherwise the weight would be $2.5$. Hence, the sector coincides with the triangle $\u{ABC}$, which is a (trivial) partition as required.\par
Let us now consider a sector of weight $\omega>2$, and assume by induction that we already know the validity of our claim for all the sectors of weight less than $\omega$. In the proof, we distinguish three cases.
\case{1}{$\omega\in\N$.}
In this case, there are by definition some vertices which are seen by $\u{AB}$. Among these vertices, let us call $\u{C}$ the one which is closest to the segment $\u{AB}$. Let us momentarily assume that neither $\u{AC}$ nor $\u{BC}$ is entirely contained in $\partial \Delta$. Then, by the minimality property of $\u{C}$, the open segments $\u{AC}$ and $\u{BC}$ lie entirely in the interior of $\Delta$, as depicted in Figure~\ref{Fig:division}~(left). Hence, one can consider the sectors $\DoubleS(\u{AC})$ and $\DoubleS(\u{BC})$, as ensured by Remark~\ref{sottosettori}. Moreover, of course the weights of both $\DoubleS(\u{AC})$ and $\DoubleS(\u{BC})$ are strictly less than $\omega$, so by inductive assumption we know that it is possible to find a suitable partition in triangles for both the sectors $\DoubleS(\u{AC})$ and $\DoubleS(\u{BC})$. Finally, since by construction the sectors $\DoubleS(\u{AC})$ and $\DoubleS(\u{BC})$ are essentially disjoint, and the union of them with the triangle $\u{ABC}$ is the whole sector $\DoubleS(\u{AB})$, putting together the two decompositions and the triangle $\u{ABC}$ we get the desired partition of $\DoubleS(\u{AB})$.\par
Let us now consider the possibility that $\u{AC}\subseteq \partial\Delta$ (if, instead, $\u{BC}\subseteq\partial\Delta$, then the completely symmetric argument clearly works). If it is so, we can anyway repeat almost exactly the same argument as before. In fact, $\u{BC}$ is entirely contained in the interior of $\Delta$, again by the minimality property of $\u{C}$ and by the fact that $\omega>2$. Moreover, the sector $\DoubleS(\u{BC})$ has weight strictly less than $\omega$, so by induction we can find a good partition of $\DoubleS(\u{BC})$, and adding the triangle $\u{ABC}$ we get the desired partition of $\DoubleS(\u{AB})$.
\begin{figure}[htbp]
\begin{center}
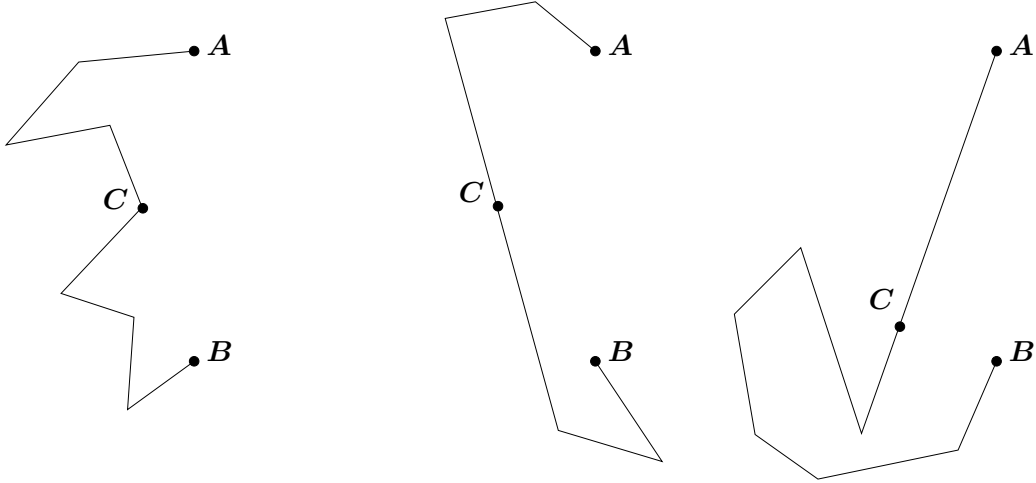\vspace{-10pt}
\caption{The three possible cases in Lemma~\ref{lemmastep3}.}\label{Fig:division}
\end{center}
\end{figure}
\case{2}{$\omega\not\in\N$, $\u{A}^+\not\equiv \u{A}$, $\u{B}^-\not\equiv\u{B}$.}
In this case, we can use the same idea of Case~1 with a slight modification. In fact, define $\u{C}\in \u{A}^+\u{B}^-$ the point such that $\u{C}_\perp$ is the middle point of the segment $\u{AB}$ (this point is well-defined as shown in the proof of Lemma~\ref{lem:tech}). Again, by definition and by Lemma~\ref{lem:tech} we have that the open segments $\u{AC}$ and $\u{BC}$ are in the interior of $\Delta$, see Figure~\ref{Fig:division}~(center).\par
Let us then \emph{decide} that the point $\u{C}$ is a new vertex of $\partial\Delta$. This means that from now on we consider the point $\u{C}$ as a vertex, and consequently we stop considering $\u{A}^+\u{B}^-$ as a side of $\partial\Delta$, instead, we think of it as the union of the two sides $\u{A}^+\u{C}$ and $\u{CB}^-$. However, notice carefully that this choice modifies the weight of $\DoubleS(\u{AB})$! In fact, the number of sides of $\DoubleS(\u{AB})$ is increased by $1$, and since $\u{AB}$ sees $\u{C}$ by construction, then the new weight of $\DoubleS(\u{AB})$ is $\omega+\frac 12$.\par
We can now argue as in Case~1. In fact, again the sector $\DoubleS(\u{AB})$ is the union of the triangle $\u{ABC}$ with the two sectors $\DoubleS(\u{AC})$ and $\DoubleS(\u{BC})$, so it is enough to put together the triangle $\u{ABC}$ and the two partitions given by the inductive assumption applied on the sectors $\DoubleS(\u{AC})$ and $\DoubleS(\u{BC})$. To do so, we have of course to be sure that the weight of both sectors is strictly less than the {original} weight of $\DoubleS(\u{AB})$, that is, $\omega$ (and not $\omega+\frac 12$!). This is clear by the assumption that $\u{A}^+\not\equiv \u{A}$ and $\u{B}^-\not\equiv\u{B}$, since the side $\u{A}^+\u{B}^-$ is neither the first nor the last of the path $\arc{\u{AB}}$, thus the weight of both sectors is at most $\omega-1$.
\case{3}{$\omega\not\in\N$ and $\u{A}^+\equiv \u{A}$ or $\u{B}^-\equiv\u{B}$.}
 By symmetry, let us assume that $\u{A}^+\equiv \u{A}$. In this case, we cannot argue exactly as in Case~2, because if we did so the sector $\DoubleS(\u{BC})$ might have weight either $\omega$ or $\omega-\frac 12$, and in the first case we could not use the inductive hypothesis.\par
Anyway, it is enough to make a slight variation of the argument of Case~2. Define $\u{C}$, as in Figure~\ref{Fig:division}~(right), the point of $\u{AB}^-$ such that $\u{BC}$ is orthogonal to $\u{AB}^-$, so that clearly the open segment $\u{BC}$ lies in the interior of $\Delta$. Let us now \emph{decide}, exactly as in Case~2, that the point $\u{C}$ is from now on an extreme, thus changing the weight of $\DoubleS(\u{AB})$ from $\omega$ to $\omega+\frac 12$.\par
By construction, the segment $\u{AB}$ sees the point $\u{C}$, and the sector $\DoubleS(\u{AB})$ is the union of the sector $\DoubleS(\u{BC})$ and of the triangle $\u{ABC}$. Hence, we conclude exactly as in the other cases if we can use the inductive assumption on the sector $\DoubleS(\u{BC})$. Notice that the number of sides of $\DoubleS(\u{BC})$ equals exactly the original number of sides of $\DoubleS(\u{AB})$, that is, $\omega-\frac 12$. Hence, in principle, the weight of $\DoubleS(\u{BC})$ could be either $\omega-\frac 12$ or $\omega$, as observed before. But in fact, by our definition of $\u{C}$, we have that the segment $\u{BC}$ sees the vertex $\u{B}^-$, so that the actual weight of $\DoubleS(\u{BC})$ is $\omega-\frac 12$, hence strictly less than $\omega$, and then we can use the inductive assumption.
\end{proof}

To give some examples, let us briefly consider the three cases drawn in Figure~\ref{Fig:division}. In the left case, the weight of $\DoubleS(\u{AB})$ was $\omega=8$, and the weights of the sectors $\DoubleS(\u{AC})$ and $\DoubleS(\u{BC})$ are both $4$. In the central case, the weight of $\DoubleS(\u{AB})$ was $\omega=5.5$, then it becomes $6$ because we add the new vertex $\u{C}$, and the weights of the sectors $\DoubleS(\u{AC})$ and $\DoubleS(\u{BC})$ are respectively $3$ and $3.5$. Finally, in the right case, the weight of $\DoubleS(\u{AB})$ was $\omega=7.5$, it becomes $8$ as we add $\u{C}$, and the weight of the sector $\DoubleS(\u{BC})$ is $7$.\par

An explicit example of a sector with a partition in triangles done according with the construction of Lemma~\ref{lemmastep3} can be seen in Figure~\ref{Fig:divsect}.\par

We conclude this step by setting a natural partial order on the triangles of the partition given by Lemma~\ref{lemmastep3} and by adding some remarks and a last definition.
\begin{definition}\label{defparord}
Let $\DoubleS(\u{AB})$ be a sector, and consider a partition satifying the properties of Lemma~\ref{lemmastep3}. We define a partial order $\leq$ between the triangles of the partition as the partial order induced by letting $\u{PQR} \leq \u{STU}$ if the exit side of $\u{PQR}$ is one of the sides of $\u{STU}$. Equivalently, let $\u{PQR}$ and $\u{STU}$ be two triangles of the partition, being $\u{SU}$ the exit side of the latter. One has $\u{PQR}\leq \u{STU}$ if and only if the points $\u{P},\, \u{Q}$ and $\u{R}$ belong to the path $\arc{\u{SU}}$.
\end{definition}

\begin{figure}[htbp]
\begin{center}
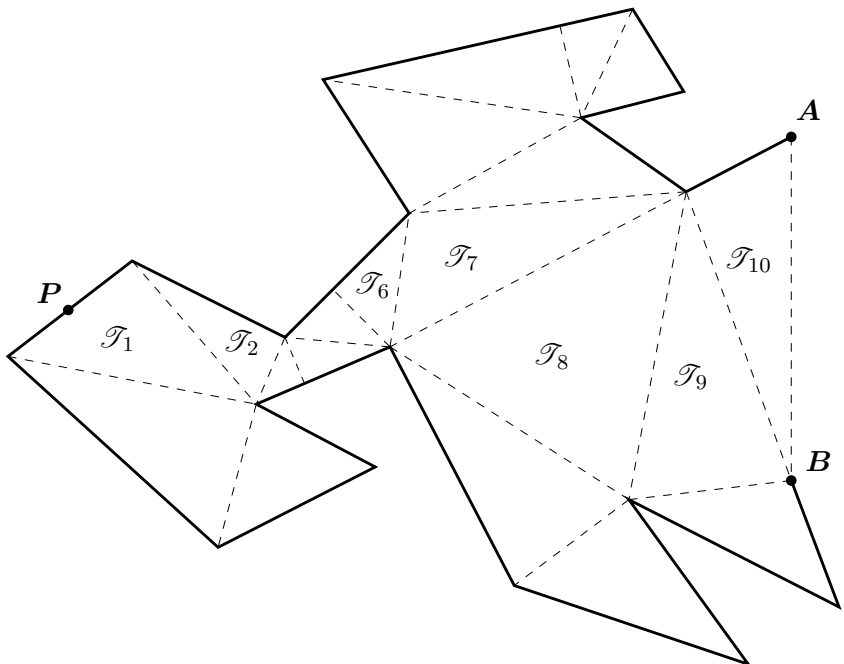\vspace{-10pt}
\caption{Partition of a sector in triangles, and natural sequence of triangles related to some $\u{P}$.}\label{Fig:divsect}
\end{center}
\end{figure}

\begin{remark}\label{rem:induction}
Notice that the relation defined above admits as greatest element the unique triangle having $\u{AB}$ as its exit side. Moreover, each triangle $\T$ except the maximizer has a unique successor.\par
We remark also that, since the triangles are a finite number, in all the future constructions we will always be allowed to consider a single triangle of the partition and to assume that the costruction has been done in all the triangles which are smaller in the sense of the order.
\end{remark}

\begin{definition}\label{def: nat seq}
Let $\DoubleS(\u{AB})$ be a sector subdivided in triangles according to Lemma~\ref{lemmastep3}, and consider a point $\u{P}\in\u{\arc{AB}}$. We will call \emph{natural sequence of triangles related to $\u{P}$} the sequence $\big(\T_1,\, \T_2,\, \dots\,,\, \T_N\big)$ of triangles of the partition satisfying the following requirements,
\begin{itemize}
\item $\T_1$ is the minimal triangle containing $\u{P}$ (minimality is intended with respect to $\leq$),
\item $\T_N$ is the triangle having $\u{AB}$ as its exit side,
\item $\T_{i+1}$ is the successor of $\T_i$ for all $1\leq i\leq N-1$.
\end{itemize}
\end{definition}
It is immediate, thanks to the above remarks, to observe that this sequence is univoquely determined.
Figure~\ref{Fig:divsect} shows a sector subdivided in triangles and a point $\u{P}$ with the related natural sequence of triangles $\big(\T_1,\, \dots \,,\, \T_{10}\big)$, with the arrows on the exit sides.

\bigstep{IV}{Definition of the paths inside a sector}

In this step we define non-intersecting piecewise affine paths starting from any point $\u{P}\in \arc{\u{AB}}$ and ending on $\u{AB}$, where $\DoubleS(\u{AB})$ is a given sector. This is the most important and delicate point of our construction. The goal of this step is to provide the ``first part'' of the piecewise affine path from a vertex $\u{P}$ to the center $\u{O}$ which will eventually be the image of $PO$ under $v$; namely, the part which is inside the primary sector $\DoubleS(\u{A}_i\u{A}_{i+1})$ to which $\u{P}$ belongs. Of course, to obtain the bi-Lipschitz property for the function $v$, we have to take care that all the paths starting from different points $\u{P}\neq \u{Q}$ do not become neither too far nor too close to each other. We can now give a simple definition and then state and prove the result of this step.

\begin{definition}
Let $\DoubleS(\u{AB})$ be a sector, and let $\u{P}\in \arc{\u{AB}}$. Let moreover $\big( \T_1,\, \T_2,\, \dots\,,\, \T_N\big)$ be the natural sequence of triangles related to $\u{P}$, according to Definition~\ref{def: nat seq}. We will call \emph{good path corresponding to $\u{P}$} any piecewise affine path $\u{PP}_1\u{P}_2\cdots\u{P}_N$ such that each $\u{P}_i$ belongs to the exit side of the triangle $\T_i$ (then $\u{P}_N\in \u{AB}$). Notice that $N$ depends on $\u{P}$.
\end{definition}

Figure~\ref{Fig:goodpath} shows a sector $\DoubleS(\u{AB})$ subdivided in triangles as in Lemma~\ref{lemmastep3} and shows two good paths corresponding to the points $\u{P}$ and $\u{Q}$.

\begin{figure}[htbp]
\begin{center}
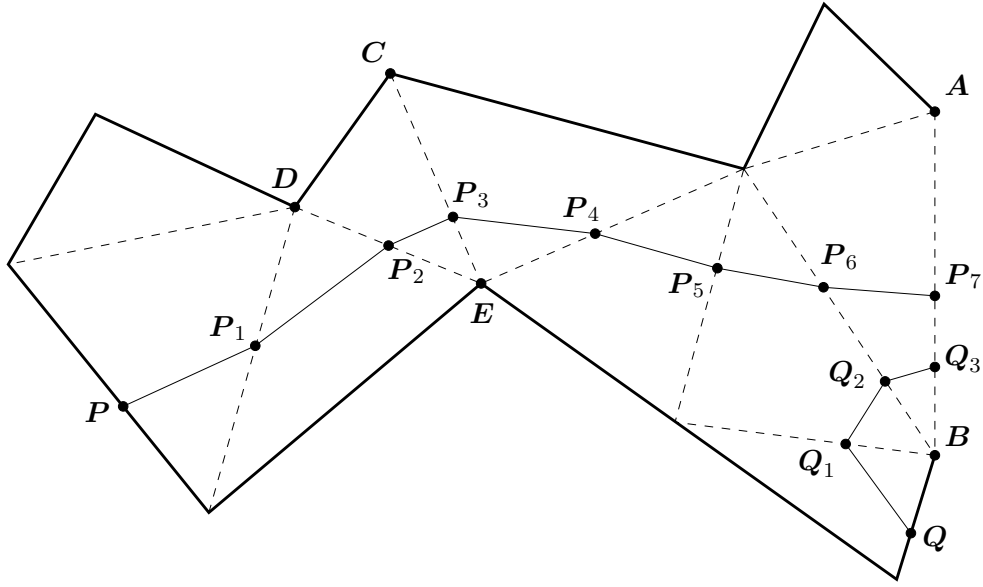\vspace{-10pt}
\caption{A sector with two good paths corresponding to $\u{P}$ and $\u{Q}$.}\label{Fig:goodpath}
\end{center}
\end{figure}

\begin{lemma}\label{lemmastep4}
Let $\DoubleS(\u{AB})$ be a sector. Then there exist good paths $\u{PP}_1\u{P}_2\cdots \u{P}_N$ corresponding to each vertex $\u{P}$ of $\arc{\u{AB}}$, with $N=N(\u{P})$, satisfying the following properties:
\begin{enumerate}
\item[(i)] for any $\u{P}$ and for any $1\leq i\leq N(\u{P})$, the segment $\u{P}_{i-1}\u{P}_i$ makes an angle of at least $\arcsin\big( \frac 1{6L^2}\big)$ with the side of $\T_i$ to which $\u{P}_{i-1}$ belongs, and an angle of at least $\pi/12=15^\circ$ with the exit side of $\T_i$;
\item[(ii)] for any $\u{P}$, $\ell\big( \arc{\u{PP}_N}\big)= \ell(\u{PP}_1)+ \ell(\u{P}_1\u{P}_2)+ \cdots + \ell(\u{P}_{N-1}\u{P}_N)\leq 4\, \lcurve{AB}$\,;
\item[(iii)] for any $\u{P},\, \u{Q}$, if for some $1\leq i \leq N(\u{P})$ and $1\leq j \leq N(\u{Q})$ one has that $\u{P}_i$ and $\u{Q}_j$ belong to the same exit side of some triangle, then
\[
\frac{\ell\big(\arc{PQ}\big)}{7L} \leq \ell\big( \u{P}_i\u{Q}_j\big) \leq \lcurve{PQ}\,,
\]
and moreover, if $i<N(\u{P})$ then 
\[
\ell\big( \u{P}_{i+1}\u{Q}_{j+1}\big) \leq \ell\big( \u{P}_i\u{Q}_j\big)\,;
\]
\item[(iv)] the piecewise affine paths $\u{PP}_1\u{P}_2\cdots \u{P}_N$ are pairwise disjoint\,.
\end{enumerate}
\end{lemma}
For the sake of clarity, let us briefly discuss the meaning of the requirements of Lemma~\ref{lemmastep4}, having in mind the example of Figure~\ref{Fig:goodpath}. Condition~(i), considered for the point $\u{P}$ and with $i=3$ (so that $\T_i = \u{CDE}$) means that
\begin{align*}
\sin\Big(\angle{\u{P}_3}{\u{P}_2}{\u{D}}\Big) \geq \frac{1}{6L^2}\,, &&
\sin\Big(\angle{\u{P}_3}{\u{P}_2}{\u{E}}\Big) \geq \frac{1}{6L^2}\,, &&
\angle{\u{P}_2}{\u{P}_3}{\u{C}} \geq \frac{\pi}{12}\,, &&
\angle{\u{P}_2}{\u{P}_3}{\u{E}} \geq \frac{\pi}{12}\,.
\end{align*}
Condition~(ii) just means that $\ell\big(\arc{\u{PP}_7}\big) \leq 4 \,\lcurve{AB}$, where $\arc{\u{PP}_7}$ denotes the piecewise affine path $\u{PP}_1\u{P}_2 \cdots \u{P}_7$. Similarly, $\ell\big(\arc{\u{QQ}_3}\big) \leq 4 \,\lcurve{AB}$.\par
Condition~(iii) ensures that
\[
\frac{\ell\big(\arc{PQ}\big)}{7L} \leq \ell\big( \u{P}_7\u{Q}_3\big)\leq \ell\big( \u{P}_6\u{Q}_2\big) \leq \lcurve{PQ}\,.
\]
In particular, concerning the second half of~(iii), notice that by construction if $\u{P}_i$ and $\u{Q}_j$ belong to the same exit side of a triangle, then also the points $\u{P}_{i+1}$ and $\u{Q}_{j+1}$ belong to the same exit side of a triangle and so on. Hence, the second half of~(iii) is saying that the function $l \mapsto \ell\big(\u{P}_{i+l}\u{Q}_{j+l}\big)$ is a decreasing function of $l$ for $0\leq l \leq N(\u{P})-i= N(\u{Q})-j$.\par
Finally, condition~(iv) illustrates the whole idea of the construction of this step, that is, the piecewise affine paths starting from the curve $\arc{\u{AB}}$ and arriving to the segment $\u{AB}$ do not intersect to each other, as in Figure~\ref{Fig:goodpath}.

\proofof{Lemma~\ref{lemmastep4}}
We will show the thesis arguing by induction on the weight of the structure $\DoubleS(\u{AB})$, as in Lemma~\ref{lemmastep3}. In fact, instead of proving that the thesis is true for structures of weight $2$ (recall that this is the minimal possible weight) and then giving an inductive argument, we will prove everything at once. In other words, we take a structure $\DoubleS(\u{AB})$ and we assume that \emph{either} $\DoubleS(\u{AB})$ has weight $2$, \emph{or} the result has been already shown for all the structures of weight less than the weight of $\DoubleS(\u{AB})$.\par

Let us call $\u{C}\in \arc{\u{AB}}$ the point such that $\u{ABC}$ is the greatest triangle of the partition of $\DoubleS(\u{AB})$ with the order of Definition~\ref{defparord}.\par

Consider now the segment $\u{BC}$, which lies entirely either in the interior of $\Delta$ or on $\partial\Delta$. In the first case, $\DoubleS(\u{BC})$ is a sector of weight strictly less than that of $\DoubleS(\u{AB})$. Then, by inductive assumption, there are piecewise affine paths $\u{PP}_1\cdots\u{P}_{N-1}$ for each vertex $\u{P}\in\arc{\u{BC}}$, with $\u{P}_{N-1}\in \u{BC}$, satisfying conditions~(i)--(iv) with $\DoubleS(\u{BC})$ in place of $\DoubleS(\u{AB})$. We have then to connect the points $\u{P}_{N-1}$ on $\u{BC}$ with the segment $\u{AB}$. In the second case, i.e. if $\u{BC}\subseteq \partial\Delta$, then $\arc{\u{BC}}=\u{BC}$, thus we have to connect all the vertices contained in $\u{BC}$ (which are not necessarily only $\u{B}$ and $\u{C}$!) with the segment $\u{A}\u{B}$. The same considerations hold for $\u{AC}$ in place of $\u{BC}$.\par

The construction of the segments between $\u{AC}\cup\u{BC}$ and $\u{AB}$ will be divided, for clarity, in several parts.
~
\part{1}{Definition of $\u{C}_1$.}
By definition, $\u{C}$ is a vertex of $\partial\Delta$. Hence, the first thing to do is to define the good path corresponding to $\u{C}$, that is a suitable segment $\u{CC}_1$ with $\u{C}_1\in \u{AB}$. Let us first define two points $\u{C}^+$ and $\u{C}^-$, on the straight line containing $\u{AB}$, as in Figure~\ref{Fig:step4}. These two points are defined by
\begin{align*}
\ell(\u{BC}^+)=\ell(\u{BC}) \,, && \ell(\u{AC}^-)=\ell(\u{AC})\,.
\end{align*}
In the figure, $\u{C}^\pm$ both belong to the segment $\u{AB}$, but of course it may even happen that $\u{C}^+$ stays above $\u{A}$, and/or that $\u{C}^-$ stays below $\u{B}$. Let us now give a temptative definition of $\u{C}_1$ by letting $\widetilde{\u{C}}_1$ be the point of $\u{AB}$ such that
\begin{equation}\label{defwtc1}
\frac{\ell\big(\arc{AC}\big)}{\ell\big(\arc{AB}\big)} = \frac{\ell(\u{A}\widetilde{\u{C}}_1)}{\lsegm{AB}}\,.
\end{equation}
Taking $\u{C}_1 = \widetilde{\u{C}}_1$ would be a good choice from many points of view, but unfortunately one would eventually obtain estimates weaker than~(i)--(iv).\par
Instead, we give the following definition: we let $\u{C}_1$ be the point of the segment $\u{C}^-\u{C}^+$ which is closest to $\widetilde{\u{C}}_1$. In other words, we can say that we set $\u{C}_1 = \widetilde{\u{C}}_1$ if $\widetilde{\u{C}}_1$ belongs to $\u{C}^+\u{C}^-$, while otherwise we set $\u{C}_1=\u{C}^+$ (resp. $\u{C}_1=\u{C}^-$) if $\widetilde{\u{C}}_1$ is above $\u{C}^+$ (resp. below $\u{C}^-$).\par
\begin{figure}[htbp]
\begin{center}
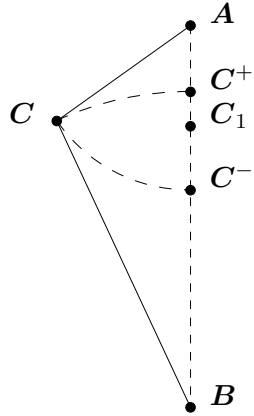\vspace{-10pt}
\caption{The triangle $\u{ABC}$ with the points $\u{C}^+,\, \u{C}^-$ and $\u{C}_1$.}\label{Fig:step4}
\end{center}
\end{figure}
Notice that $\u{C}_1$ belongs to $\u{AB}$, since so does $\widetilde{\u{C}}_1$ thanks to~(\ref{defwtc1}). It is also important to underline that
\begin{align}\label{stimeAC}
\ell\big(\arc{AC}\big) \leq \sqrt{2} L\, \ell\big(\u{AC}_1\big)\,, &&
\ell\big(\arc{BC}\big) \leq \sqrt{2} L\, \ell\big(\u{BC}_1\big)\,.
\end{align}
By symmetry, let us only show the first inequality. Recall that by~(\ref{stimaL^2}) we know
\begin{align*}
\ell\big(\arc{AC}\big) \leq \sqrt{2} L\,\lsegm{AC} \,, &&
\ell\big(\arc{AB}\big) \leq \sqrt{2} L\,\lsegm{AB}\,.
\end{align*}
As a consequence, either $\u{C}_1=\u{C}^-$, and then
\[
\ell\big( \u{AC}_1\big) = \ell\big(\u{AC}^- \big) = \lsegm{AC} \geq \frac{\ell\big(\arc{AC}\big)}{\sqrt{2} L}\,,
\]
or $\ell\big(\u{AC}_1\big) \geq \ell\big(\u{A}\widetilde{\u{C}}_1\big)$, and then by~(\ref{defwtc1})
\[
\ell\big(\u{AC}_1\big) \geq \ell\big(\u{A}\widetilde{\u{C}}_1\big) =
\ell\big(\arc{AC}\big)\, \frac{\lsegm{AB}}{\ell\big(\arc{AB}\big)} 
\geq \frac{\ell\big(\arc{AC}\big)}{\sqrt{2} L}\,.
\]
Recall now that, to show the thesis, all we have to do is to take each vertex $\u{D}\in \u{AC}\cup\u{BC}$ and to find a suitable corresponding point $\u{D'}\in\u{AB}$, in such a way that the requirements~(i)--(v) are satisfied. Having defined $\u{C}_1$, we have then to send the points $\u{P}_N$ of $\u{AC}$ in $\u{AC}_1$ and those of $\u{BC}$ in $\u{BC}_1$.\par
We claim that the two segments can be considered independently, that is, we can limit ourselves to describe how to send $\u{BC}$ on $\u{BC}_1$ and check that the properties~(i)--(iv) hold for points of $\arc{\u{BC}}$. Indeed, if we do so, by symmetry the same definitions can be repeated for $\u{AC}$, and the properties~(i)--(iv) hold separately for points of $\arc{\u{BC}}$ and $\arc{\u{AC}}$. The only thing which would be missing, then, would be to check the validity of~(iii) for two points $\u{P}\in\arc{\u{AC}}$ and $\u{Q}\in\arc{\u{BC}}$. Moreover, this will be trivially true, because since $\u{C}$ belongs to both the segments $\u{AC}$ and $\u{BC}$, then it is enough to use~(iii) once with $\u{P}$ and $\u{C}$, and once with $\u{C}$ and $\u{Q}$, recalling that clearly
\begin{align*}
\ell\big(\arc{PQ}\big)=\ell\big(\arc{PC}\big)+\ell\big(\arc{CQ}\big) \,, &&  \ell\big( \u{P}_i\u{Q}_j\big) =\ell\big( \u{P}_i\u{C}_1\big) + \ell\big( \u{C}_1\u{Q}_j\big)\,.
\end{align*}
For this reason, from now on we will concentrate ourselves only on the segment $\u{BC}$. We will call $\u{D}$ the generic point of $\u{BC}$, which clearly corresponds to $\u{P}_{N-1}$ for some $\u{P}\in\arc{\u{BC}}$, as discussed at the beginning of the proof.
\part{2}{Construction for the case $\u{C}_1 = \u{C}^+$.}
In this case, for any $\u{D}\in \u{BC}$ we set its image as the point $\u{D'}\in \u{BC}_1$ for which $\ell(\u{BD})=\ell(\u{BD'})$. Then in particular all the segments $\u{DD'}$ are parallel to $\u{CC}_1$. Let us now check the validity of~(i)--(iii), since~(iv) is trivially true.\par\smallskip

We start with~(i). Given $\u{D}\in\u{BC}$, and $\u{D'}$ its image, call $\beta = \uangle ABC\in (0,\pi/2]$. Then one has
\begin{align*}
\uangle D{D'}B = \uangle {D'}DB = \frac{\pi - \beta}{2} \,, &&
\uangle D{D'}A = \uangle {D'}DC = \frac{\pi + \beta}{2}\,,
\end{align*}
thus~(i) holds true.\par
Let us now consider~(ii). Given a point $\u{D}\in\u{BC}$, by construction one has
\begin{equation}\label{byconstr}
\lsegm{DD'}\leq \ell\big(\u{CC}_1\big)\leq \lsegm{AC}\leq \lcurve{AC}\,.
\end{equation}
We can then consider separately two cases. If $\u{BC}\subseteq\partial\Delta$, then one simply has $\u{P}\equiv \u{D}$ and $\u{P}_N\equiv \u{P}_1\equiv \u{D'}$, so clearly
\[
\ell\big( \arc{\u{PP}_N} \big) = \lsegm{DD'} \leq \lcurve{AC} \leq \lcurve{AB}\,.
\]
On the other hand, if the open segment $\u{BC}$ lies in the interior of $\Delta$, then one has
\begin{equation}\label{byconstr2}
\ell\big(\arc{\u{PP}_{N-1}}\big) \leq 4 \lcurve{BC}
\end{equation}
by inductive assumption, thus~(\ref{byconstr}) and~(\ref{byconstr2}) give
\[
\ell\big( \arc{\u{PP}_N} \big) = \ell\big(\arc{\u{PP}_{N-1}}\big) + \lsegm{DD'}
\leq 4 \lcurve{BC} + \lcurve{AC} \leq 4 \lcurve{AB}\,,
\]
hence also~(ii) is done.\par
It remains now to consider~(iii). Thus we take two points $\u{D}\equiv \u{P}_{N-1}$ and $\u{E}\equiv \u{Q}_{\widetilde{N}-1}$ on $\u{BC}$, denoting for brevity $N=N(\u{P})$ and $\widetilde N = N(\u{Q})$. We have to consider separately the two cases arising if $\u{BC}$ lies in the boundary or in the interior of $\Delta$. In the first case, $\u{P}\equiv \u{D}$ and $\u{Q}\equiv \u{E}$, thus by the Lipschitz property of $u$ we have
\[
\frac{\ell\big(\arc{PQ}\big)}{L} \leq \lcurve{PQ}= \lsegm{DE} = \lsegm{D'E'}\,,
\]
so that~(iii) is trivially true. In the second case, $\lsegm{D'E'}=\lsegm{DE}$, so~(iii) is true by inductive assumption.\par\bigskip

To conclude the proof, we now have to see what happens when $\u{C}_1\neq \u{C}^+$. We will further subdivide this last case depending on whether $\beta > \pi/12$ or not, being $\beta=\uangle{A}{B}{C}$.

\part{3}{Construction for the case $\u{C}_1 \neq \u{C}^+$, $\beta\geq  15^\circ$.}
 In this case, for any $\u{D}\in \u{BC}$ we define $\u{D'}\in \u{BC}_1$ as the point satisfying
\begin{equation}\label{defd'}
\lsegm{BD'} = \min \bigg\{ \lsegm{BD}\,,\, \ell\big(\u{BC}_1\big) - \frac{\ell\big(\arc{PC}\big)}{7L}  \bigg\} \,,
\end{equation}
being as usual $\u{P}\in \arc{\u{BC}}$ the point such that $\u{D}=\u{P}_{N-1}$. Observe that this definition makes sense since, also using~(\ref{stimeAC}), one has that the minimum in~(\ref{defd'}) is between $0$ and $\ell\big(\u{BC}_1\big)$ for each $\u{D}\in\u{BC}$. In particular, the minimum is strictly increasing between $0$ and $\ell\big(\u{BC}_1\big)$ as soon as $\u{D}$ moves from $\u{B}$ to $\u{C}$, so~(iv) is already checked. Let us then check the validity of~(i)--(iii).\par

We first concentrate on~(i). Just for a moment, let us call $\u{D^*}\in \u{BC}^+$ the point for which $\lsegm{BD}=\lsegm{BD^*}$, so that the triangle $\u{BDD^*}$ is isosceles. Therefore, one immediately has
\begin{align}\label{anglep1}
\uangle{D}{D'}{B} \geq \uangle{D}{D^*}{B}= \frac{\pi- \beta}{2} \geq \frac \pi 4\,, &&
\uangle{D'}{D}{C}\geq \uangle{D^*}{D}{C}= \frac{\pi+\beta}{2}\geq \frac\pi 2\,.
\end{align}
Moreover, by construction it is clear that
\begin{equation}\label{anglep2}
\uangle{D}{D'}{A} \geq \uangle{D}{B}{A} = \beta \geq \frac \pi{12}\,.
\end{equation}
To conclude, we have to estimate $\uangle{D'}{D}{B}$, and we start claiming the bound
\begin{equation}\label{anglep2.5}
\lsegm{BD'} \geq \frac{\lsegm{BD}}{\sqrt{2} L^2}\,.
\end{equation}
In fact, recalling~(\ref{defd'}), either $\lsegm{BD'}=\lsegm{BD}$, and then~(\ref{anglep2.5}) clearly holds, or otherwise by~(\ref{stimeAC}) and the Lipschitz property of $u$
\[\begin{split}
\lsegm{BD'} &= \ell\big(\u{BC}_1\big) - \frac{\ell\big(\arc{PC}\big)}{7L}
\geq \frac{\ell\big(\arc{BC}\big)}{\sqrt{2}L}- \frac{\ell\big(\arc{PC}\big)}{7L}
\geq \frac{\ell\big(\arc{BC}\big) - \ell\big(\arc{PC}\big)}{\sqrt{2}L}
=\frac{\ell\big(\arc{BP}\big)}{\sqrt{2}L}
\geq \frac{\lcurve{BP}}{\sqrt{2}L^2}\\
&\geq \frac{\lsegm{BD}}{\sqrt{2}L^2}\,,
\end{split}\]
thus again~(\ref{anglep2.5}) is checked. Concerning the last inequality, namely $\lcurve{BP}\geq \lsegm{BD}$, this is an equality if the segment $\u{BC}$ belongs to $\partial\Delta$, while otherwise it is true by inductive assumption on the sector $\DoubleS(\arc{\u{BC}})$, applying~(iii) to the points $\u{P}$ and $\u{Q}\equiv \u{B}$. Consider now the triangle $\u{DBD'}$: immediate trigonometric arguments tell us that
\begin{align*}
\lsegm{DD'} \sin \big( \uangle{D'}{D}{B} \big) = \lsegm{BD'} \sin\beta\,, &&
\lsegm{BD} \sin\beta = \lsegm{DD'} \sin \Big( \uangle{D'}{D}{B}+\beta\Big)\,,
\end{align*}
from which we get, using also~(\ref{anglep2.5}),
\begin{equation}\label{anglep3}
\sin\big( \uangle{D'}{D}{B} \big) = \frac{\lsegm{BD'}}{\lsegm{BD}}\, \sin\Big( \uangle{D'}{D}{B}+\beta\Big)
\geq \frac{\sin 15^\circ}{\sqrt{2} L^2} \geq \frac{1}{6L^2}\,.
\end{equation}
Putting together~(\ref{anglep1}), (\ref{anglep2}) and~(\ref{anglep3}), we conclude the inspection of~(i).\par
Concerning~(ii), it is enough to observe that
\begin{equation}\label{stepii3}
\frac{\lsegm{DD'}}{\lcurve{AC}} \leq 
\frac{\lsegm{DD'}}{\lsegm{AC}} \leq
\frac{\sin\big(\uangle CAB\big)}{\sin\big(\uangle D{D'}A\big)}
 \leq \frac{1}{\sin 15^\circ }\leq 4\,.
\end{equation}
Therefore, as in Part~2, either $\u{BC}\subseteq\partial\Delta$, and then
\[
\ell\big( \arc{\u{PP}_N} \big) = \lsegm{DD'} \leq 4 \lcurve{AC} \leq 4 \lcurve{AB}\,,
\]
or thanks to the inductive assumption one has
\[
\ell\big( \arc{\u{PP}_N} \big) = \ell\big(\arc{\u{PP}_{N-1}}\big) + \lsegm{DD'}
\leq 4 \lcurve{BC} + 4 \lcurve{AC} = 4 \lcurve{AB}\,,
\]
so~(ii) is again easily checked.\par
Let us now consider~(iii). As in Part~2, we take on $\u{BC}$ two points $\u{D}\equiv \u{P}_{N-1}$ and $\u{E}\equiv \u{Q}_{\widetilde{N}-1}$ with $N=N(\u{P})$ and $\widetilde N = N(\u{Q})$, and we assume by symmetry that $\lsegm{BD}\leq \lsegm{BE}$. Since it is surely $\lsegm{DE} \leq \lcurve{PQ}$, either as a trivial equality if $\u{BC}\subseteq\partial\Delta$, or by inductive assumption otherwise, showing~(iii) consists in proving that
\begin{equation}\label{step3iii}
\frac{\ell\big(\arc{PQ}\big)}{7L} \leq \lsegm{D'E'}\leq \lsegm{DE}\,.
\end{equation}
We start with the right inequality. Recalling the definition~(\ref{defd'}), if $\lsegm{BD'}=\lsegm{BD}$ then, since $\lsegm{BE'}\leq\lsegm{BE}$, one has
\[
\lsegm{D'E'} = \lsegm{BE'}-\lsegm{BD'} \leq \lsegm{BE} - \lsegm{BD} = \lsegm{DE}\,.
\]
On the other hand, if
\[
\lsegm{BD'}=\ell\big(\u{BC}_1\big) - \frac{\ell\big(\arc{PC}\big)}{7L}\,,
\]
then we get
\[\begin{split}
\lsegm{D'E'} &= \lsegm{BE'}-\lsegm{BD'} 
\leq \Bigg(\ell\big(\u{BC}_1\big) - \frac{\ell\big(\arc{QC}\big)}{7L}\,\Bigg) - \Bigg(\ell\big(\u{BC}_1\big) - \frac{\ell\big(\arc{PC}\big)}{7L}\,\Bigg)\\
&= \frac{\ell\big(\arc{PQ}\big)}{7L} \leq \lsegm{DE}\,,
\end{split}\]
where again the last inequality is true either by the Lipschitz property of $u$ if $\u{PQ}=\u{DE}$, or by inductive assumption otherwise. Thus, the right inequality in~(\ref{step3iii}) is established, and we pass to consider the left one.\par
Still recalling~(\ref{defd'}), if $\lsegm{BE'}=\lsegm{BE}$ then
\[
\lsegm{D'E'} = \lsegm{BE'}-\lsegm{BD'} \geq \lsegm{BE} - \lsegm{BD} = \lsegm{DE} \geq \frac{\ell\big(\arc{PQ}\big)}{7L}\,,
\]
being again the last equality true either by the Lipschitz property of $u$ or by inductive assumption. Finally, if
\[
\lsegm{BE'}= \ell\big(\u{BC}_1\big) - \frac{\ell\big(\arc{QC}\big)}{7L}\,,
\]
then again we get
\[\begin{split}
\lsegm{D'E'} &= \lsegm{BE'}-\lsegm{BD'} 
\geq \Bigg(\ell\big(\u{BC}_1\big) - \frac{\ell\big(\arc{QC}\big)}{7L}\,\Bigg) - \Bigg(\ell\big(\u{BC}_1\big) - \frac{\ell\big(\arc{PC}\big)}{7L}\,\Bigg)
= \frac{\ell\big(\arc{PQ}\big)}{7 L}\,,
\end{split}\]
so the estimate~(\ref{step3iii}) is completely shown and then this part is concluded.
\part{4}{Construction for the case $\u{C}_1 \neq \u{C}^+$, $\beta<  15^\circ$.}
We are now ready to consider the last --and hardest-- possible situation, namely when $\u{C}_1\neq \u{C}^+$ and the angle $\beta$ is small. Roughly speaking, the fact that $\u{C}_1$ is below $\u{C}^+$ tells us that the segment $\u{BC}$ has to shrink, in order to fit into $\u{BC}_1$. On the other hand, the fact that $\beta$ is small makes it hard to obtain simultaneously the estimate~(iii) on the lengths and the~(i) on the angles. As in Figure~\ref{Fig:step4bis}, we call $\u{H}$ the orthogonal projection of $\u{C}$ on $\u{AB}$.\par
Since $\beta<\pi/12$, the point $\u{C}^-$ belongs to the segment $\u{AB}$, and then we obtain, by a trivial geometrical argument, that
\begin{equation}\label{shrfac}
\ell\big(\u{BC}_1\big) \geq \ell\big(\u{BC}^-\big) \geq \lsegm{BH} - \lsegm{CH} = \lsegm{BC} \Big( \cos\beta - \sin\beta\Big) \geq \frac{\sqrt{2}}2 \,\lsegm{BC}\,.
\end{equation}
Let us immediately go into our definition of $\u{P}_N$ for every vertex $\u{P}\in\arc{\u{BC}}$. First of all, since we need to work with consecutive vertices, let us enumerate all the vertices of $\arc{\u{BC}}$ as $\u{P}^0=\u{B},\, \u{P}^1,\, \u{P}^2,\,\dots,\, \u{P}^M=\u{C}$. The simplest idea to define the points $\u{P}^i_N$ would be to shrink all the segment $\u{BC}$ so to fit into $\u{BC}_1$, thus getting, for any pair $\u{P}^i,\,\u{P}^{i+1}$ of consecutive vertices,
\[
\ell\big(\u{P}^i_N\u{P}^{i+1}_N\big) = \frac{\ell\big(\u{BC}_1\big)}{\lsegm{BC}}\,\ell\big( \u{P}^i_{N-1}\u{P}^{i+1}_{N-1}\big)\,.
\]
Unfortunately, this does not work, since from the inductive assumption
\[
\ell\big(\u{P}^i_{N-1}\u{P}^{i+1}_{N-1}\big) \geq \frac{1}{7 L} \,\ell\big(\arc{P^i P^{i+1}}\big)
\]
one would be led to deduce
\[
\ell\big(\u{P}^i_N\u{P}^{i+1}_N\big) \geq \frac{\ell\big(\u{BC}_1\big)}{\lsegm{BC}}\,\frac{1}{7L}\,\ell\big(\arc{P^i P^{i+1}}\big)
\geq \frac{\sqrt{2}}{14 L}\,\ell\big(\arc{P^iP^{i+1}}\big)\,,
\]
by~(\ref{shrfac}), so the induction would not work.\par
However, our idea to overcome the problem is very simple, that is, among all the pairs $\u{P}^i,\,\u{P}^{i+1}$ of consecutive vertices we will shrink only those which are still ``shrinkable'', that is, for which the ratio
\begin{equation}\label{defrhoi}
\varrho_i := \frac{\ell\big(\u{P}^i_{N-1}\u{P}^{i+1}_{N-1}\big)}{\ell\big(\arc{P^iP^{i+1}}\big)}
\end{equation}
is not already too small, more precisely, not smaller than $1/(3L)$. Let us make this formally. Define
\begin{equation}\label{defdelta}
\delta := \sum \bigg\{ \ell\big(\u{P}^i_{N-1}\u{P}^{i+1}_{N-1}\big):\, \varrho_i  \leq \frac{1}{3L} \bigg\}\,,
\end{equation}
and notice that
\[
\ell\big(\arc{BC}\big) \geq\sum \bigg\{  \ell\big(\arc{P^iP^{i+1}}\big):\,\varrho_i\leq\frac{1}{3L} \bigg\} \geq 3L \delta\,,
\]
then by~(\ref{stimaL^2})
\begin{equation}\label{estidelta}
\delta \leq \frac{\ell\big(\arc{BC}\big)}{3L} \leq \frac{\sqrt{2}}{3}\, \lsegm{BC}\,.
\end{equation}
Finally, we define the points $\u{P}^i_N$ in such a way that any segment $\u{P}^i_N\u{P}^{i+1}_N$ has the same length as $\u{P}^i_{N-1}\u{P}^{i+1}_{N-1}$ if $\varrho_i$ is small, and otherwise it is rescaled by a factor $\lambda<1$ (constant through all $\u{BC}$). In other words, setting the increasing sequence $\delta_i$ as
\begin{equation}\label{defdeltai}
\delta_i := \sum \bigg\{ \ell\big(\u{P}^j_{N-1}\u{P}^{j+1}_{N-1}\big):\, j<i,\, \varrho_j  \leq \frac{1}{3L} \bigg\}\,,
\end{equation}
so that comparing with~(\ref{defdelta}) one has $\delta_0=0$ and $\delta_M=\delta$, we define $\u{P}^i_N$ to be the point of $\u{BC}_1$ such that
\begin{equation}\label{defpin}
\ell\big( \u{BP}^i_N\big) = \delta_i + \lambda \Big( \ell\big(\u{BP}^i_{N-1}\big) - \delta_i\Big)\,.
\end{equation}
The constant $\lambda$ is easily estimated by the constraint that $\u{P}^M_N=\u{C}_1$ and by~(\ref{shrfac}) and~(\ref{estidelta}), getting
\begin{equation}\label{estilambda}
1> \lambda =  \frac{\ell\big(\u{BC}_1\big) - \delta}{ \lsegm{BC} - \delta }
\geq \frac{\frac{\sqrt{2}}{2}\lsegm{BC} - \delta}{ \lsegm{BC} - \delta }
\geq \frac{\frac{\sqrt{2}}{2}-\frac{\sqrt{2}}{3}}{1-\frac{\sqrt{2}}{3}}> \frac 37\,.
\end{equation}
For future reference, it is also useful to notice here another estimate of $\lambda$ which depends on $\beta$, obtained exactly as the one above from~(\ref{shrfac}) and~(\ref{estidelta}), that is,
\begin{equation}\label{estilambda2}
\lambda =  \frac{\ell\big(\u{BC}_1\big) - \delta}{ \lsegm{BC} - \delta }
\geq \frac{\lsegm{BC}\big(\cos\beta - \sin\beta\big) - \delta}{ \lsegm{BC} - \delta }
\geq \frac{\cos\beta-\sin\beta-\frac{\sqrt{2}}{3}}{1-\frac{\sqrt{2}}{3}}\,.
\end{equation}
Notice that by~(\ref{defdeltai}) and~(\ref{defpin}) one readily gets
\begin{equation}\label{lpipi+1}
\ell\big( \u{P}^i_N\u{P}^{i+1}_N  \big) = \left\{\begin{array}{rl}
\bal\ell\big( \u{P}^i_{N-1}\u{P}^{i+1}_{N-1}\big)\eal \qquad  &\bal\mathrm{if}\ \varrho_i \leq \frac{1}{3L}\,,\eal\\
\bal\lambda\,\ell\big( \u{P}^i_{N-1}\u{P}^{i+1}_{N-1}\big)\eal\qquad  &\bal\mathrm{otherwise}\eal\,.
\end{array}\right.
\end{equation}
Now that we have given the definition of the points $\u{P}^i_N$, we only have to check the validity of~(i)--(iii), since~(iv) is again trivial by definition.\par
\begin{figure}[htbp]
\begin{center}
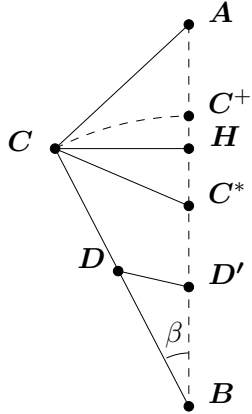\vspace{-10pt}
\caption{The triangle $\u{ABC}$ in Part~4.}\label{Fig:step4bis}
\end{center}
\end{figure}
Let us start with~(i). Take $0\leq i\leq M$ and call, as before, $\u{D}=\u{P}^i_{N-1}$ and $\u{D'}=\u{P}^i_N$. Since by construction $\lsegm{BD'}\leq \lsegm{BD}$, then one immediately gets $\uangle D{D'}B \geq \uangle {D'}DB$, from which one directly gets
\begin{align}\label{firsttwo}
\uangle D{D'}B \geq \frac{\pi - \beta}{2} \geq \frac{11}{24}\, \pi\,, &&
\uangle {D'}DC = \pi - \uangle {D'}DB \geq \frac{\pi+\beta}{2}\geq \frac \pi 2\,,
\end{align}
so that the first two angles are checked and we need to estimate $\uangle {D'}DB$ and $\uangle D{D'}A$. To do so, let us call $\u{C^*}\in\u{AB}$ the point such that $\lsegm{BC^*} = \lambda\, \lsegm{BC}$, so that by construction
\begin{align}\label{trivial}
\uangle {D'}DB \geq \uangle {C^*}CB \,, && \uangle D{D'}A \geq \uangle C{C^*} A\,.
\end{align}
The point $\u{C^*}$ must lie either between $\u{H}$ and $\u{C^+}$ or between $\u{B}$ and $\u{H}$. In the first case also the other two angles are immediately estimated, since then by~(\ref{trivial}) one has
\begin{align}\label{lasttwo}
\uangle{D'}DB \geq \uangle {C^*}CB \geq \uangle HCB = \frac \pi 2 - \beta \geq \frac{5}{12}\,\pi\,, &&
\uangle D{D'}A \geq \uangle C{C^*} A \geq \frac \pi 2\,.
\end{align}
Assume then that, as in Figure~\ref{Fig:step4bis}, $\u{C^*}$ is between $\u{B}$ and $\u{H}$. Then we can estimate, also recalling~(\ref{estilambda2}),
\[\begin{split}
\lsegm{C^*H} &= \lsegm{BH}-\lsegm{BC^*} 
= \lsegm{BC} \Big( \cos \beta - \lambda\Big)\\
&\leq \lsegm{BC} \bigg( \cos \beta - \frac{\cos\beta-\sin\beta-\frac{\sqrt{2}}{3}}{1-\frac{\sqrt{2}}{3}} \bigg) 
=\lsegm{BC}\frac{\frac{\sqrt{2}}{3}\,\frac{\sin\beta}{1+\cos\beta} +1}{1-\frac{\sqrt{2}}{3}} \,\sin\beta\,.
\end{split}\]
As a consequence, we have
\[\begin{split}
\uangle HC{C^*} &= \arctan \bigg( \frac{\lsegm{C^*H}}{\lsegm{CH}}\bigg)
\leq \arctan \bigg(\frac{\frac{\sqrt{2}}{3}\,\frac{\sin\beta}{1+\cos\beta} +1}{1-\frac{\sqrt{2}}{3}}\bigg)\\
&\leq \arctan \bigg(\frac{\frac{\sqrt{2}}{3}\,\frac{\sin 15^\circ}{1+\cos 15^\circ} +1}{1-\frac{\sqrt{2}}{3}}\bigg)
\leq 0.36 \,\pi < 65^\circ \,.
\end{split}\]
Finally, from this estimate and~(\ref{trivial}) we get
\begin{equation}\label{thirdfourth}\begin{split}
\uangle {D'}DB &\geq \uangle {C^*}CB = \frac \pi 2 - \beta - \uangle HC{C^*} > \frac{\pi}{18} \,, \\
\uangle D{D'}A &\geq \uangle C{C^*} A = \frac \pi 2 - \uangle HCC^* \geq 25^\circ \,.
\end{split}\end{equation}
Putting together the first two estimates from~(\ref{firsttwo}), and the last two estimates either from~(\ref{lasttwo}) or from~(\ref{thirdfourth}), we conclude the proof of~(i).\par
Let us now check~(ii). Repeating the argument of Part~3, we have that~(ii) follows at once as soon as one shows~(\ref{stepii3}), that is, $\lsegm{DD'} \leq 4\, \lcurve{AC}$. But in fact, using~(\ref{thirdfourth}), we immediately get
\[
\lsegm{DD'} \leq \frac{\lsegm{CH}}{\sin \big(\uangle D{D'}A\big)} 
\leq \frac{\lsegm{AC}}{\sin \big(\uangle D{D'}A\big)}
\leq \frac{\lcurve{AC}}{\sin 25^\circ} 
< 4\, \lcurve{AC}\,.
\]
Let us then consider~(iii). It is of course sufficient to check the validity of the inequality only when $\u{P}$ and $\u{Q}$ are two consecutive vertices of $\u{\arc{BC}}$. Let us then take $0\leq i< M$ and recall that we have to show
\begin{equation}\label{step4iii}
\frac{\ell\big( \arc{P^iP^{i+1}}\big)}{7L} \leq \ell\big( \u{P}^i_N\u{P}^{i+1}_N\big)\leq\ell\big(\u{P}^i_{N-1}\u{P}^{i+1}_{N-1}\big)
\end{equation}
knowing, again either by inductive assumption or by the Lipschitz property,
\begin{equation}\label{assstep4iii}
\frac{\ell\big( \arc{P^iP^{i+1}}\big)}{7L} \leq \ell\big( \u{P}^i_{N-1}\u{P}^{i+1}_{N-1}  \big) \leq \ell\big( \arc{\u{P}^i\u{P}^{i+1}}\big)\,.
\end{equation}
The right inequality in~(\ref{step4iii}) is an immediate consequence of~(\ref{lpipi+1}), being $\lambda<1$. Concerning the left inequality, it is also quick to check, distinguishing whether $\varrho_i$ is small or not. In fact, if $\varrho_i\leq 1/(3L)$, then by~(\ref{lpipi+1}) also the left inequality in~(\ref{step4iii}) derives from the analogous inequality in~(\ref{assstep4iii}). Otherwise, if $\varrho_i>1/(3L)$, then one directly has by~(\ref{lpipi+1}), (\ref{defrhoi}) and~(\ref{estilambda}) that
\[
\ell\big( \u{P}^i_N\u{P}^{i+1}_N  \big) = 
\lambda\,\ell\big( \u{P}^i_{N-1}\u{P}^{i+1}_{N-1}\big)
=\lambda\varrho_i \, \ell\big(\arc{P^iP^{i+1}}\big)
> \frac{1}{3L} \,\lambda\, \ell\big(\arc{P^iP^{i+1}}\big)
> \frac{1}{7L} \,\ell\big(\arc{P^iP^{i+1}}\big)\,,
\]
thus concluding the proof.
\end{proof}

\bigstep{V}{Bound on the lengths of the paths $\arc{\u{PP}_N}$}

In Step~IV, we have described how to get a piecewise affine path $\u{PP}_1\u{P}_2\cdots \u{P}_N$ which starts from any vertex $\u{P}\in\arc{\u{AB}}$ and ends on the segment $\u{AB}$, being $\DoubleS(\u{AB})$ a given sector. In this step, we want to improve the estimate from above of the length of this path. This is important because this path will be (up to a small correction in the future) part of the image of the segment $PO\subseteq \D$ under the extension $v$ of $u$ that we are building, and then its length gives a lower bound to the Lipschitz constant of the map $v$. After a short definition, we will state the main result of this step.

\begin{definition}
Let $\DoubleS(\u{AB})$ be a given sector, $\u{P}\in\arc{\u{AB}}$ and let $\u{PP}_1\u{P}_2\cdots \u{P}_N$ be the piecewise affine path given by Lemma~\ref{lemmastep4}. We will then denote this piecewise affine path as $\arc{\u{PP}_N}$. More in general, for any $1\leq i < j \leq N$, we will denote by $\arc{\u{P}_i\u{P}_j}$ the piecewise affine path $\u{P}_i\u{P}_{i+1}\cdots \u{P}_j$.
\end{definition}

\begin{lemma}\label{lemmastep5}
Let $\DoubleS(\u{AB})$ be a sector. Then, for any $\u{P}\in\arc{\u{AB}}$ one has
\[
\ell\big( \arc{\u{PP}_N}\big) \leq 113\, \min \Big\{ \lcurve{AP},\, \lcurve{PB}\Big\}\,.
\]
\end{lemma}

Before entering into the proof, which is quite involved, let us quickly give a rough idea of how it works, together with some useful notation. Let us fix a generic point $\u{P}\in\arc{\u{AB}}$. The proof of the lemma will require a detailed analysis of the different triangles of the natural sequence of triangles related to $\u{P}$. Recall that the natural sequence of triangles, according with Definition~\ref{def: nat seq}, is the sequence $\big(\T_1,\, \T_2,\, \dots\,,\, \T_N\big)$ such that every $\u{P}_i$ of the path $\arc{\u{PP}_N}$  belongs to the exit side of $\T_i$. Let us start by calling for simplicity $\u{A}_i\u{B}_i$ the exit side of the triangle $\T_i$, being $\u{A}_i \in \arc{\u{AP}}$ and $\u{B}_i\in\arc{\u{PB}}$, so that in particular $\u{A}_N = \u{A}$ and $\u{B}_N=\u{B}$. Moreover, we call $\u{A}_0\u{B}_0$ the side of $\T_1$ which contains $\u{P}=\u{P}_0$. Notice that, by the construction of the triangles done in Step~III, for any $i$ the exit side of the triangle $\T_i$ is a side of the triangle $\T_{i+1}$, thus the exit sides of $\T_i$ and $\T_{i+1}$ have exactly one point in common. In other words, either $\u{A}_{i+1}=\u{A}_i$, or $\u{B}_{i+1}=\u{B}_i$. Let us then assume, by simmetry, that $\ell(\u{\arc{PB}})\leq\ell(\u{\arc{AP}})$, so that the claim of Lemma~\ref{lemmastep5} can be rewritten as
\begin{equation}\label{wewant}
\sum_{i=0}^{N-1} \ell(\u{P}_i\u{P}_{i+1}) \leq 113 \bigg( \ell(\u{P}_0\u{B}_0) + \sum_{i=0}^{N-1} \ell(\u{B}_i\u{B}_{i+1}) \bigg)\,.
\end{equation}
Pick now a generic $0 \leq i < N$: on one hand, if $\u{B}_{i+1}\neq\u{B}_i$, then we will see that property $(i)$ of Lemma \ref{lemmastep4} implies
\[
\ell(\u{P}_i\u{P}_{i+1})\leq4\ell(\u{B}_i\u{B}_{i+1})\,,
\]
and this is clearly in accordance with the validity of~(\ref{wewant}). But if, instead, $\u{B}_i=\u{B}_{i+1}$, then the length of the segment $\u{P}_i\u{P}_{i+1}$ does not apparently contribute to the increase of the path $\ell(\arc{\u{P}_0\u{B}_N})$. However, since by~(iii) of Lemma~\ref{lemmastep4} one has $\ell(\u{P}_{i+1}\u{B}_i)=\ell(\u{P}_{i+1}\u{B}_{i+1})\leq \ell(\u{P}_i\u{B}_i)$, then it is reasonable to guess that the total length $\ell(\arc{\u{P}_i \u{P}_j})$ for $\u{B}_i=\u{B}_j$ cannot be too large: obtaining such a precise estimate is basically what we need to show Lemma~\ref{lemmastep5}. To do so, our strategy will be to group the triangles $\T_i$ in a suitable way, in order to get the informations that we need. In particular, we will first subdivide the natural sequence of triangles $\big(\T_1,\, \T_2,\, \dots\,,\, \T_N\big)$ into sequences of consecutive triangles $\U=\big(\T_{i},\, \T_{i+1},\, \dots\,,\, \T_{i+j}\big)$ called ``units'', then we will group consecutive sequences of ``units'' into ``systems of units'' $\DoubleSS=\big(\U_{i},\, \U_{i+1},\, \dots\,,\, \U_{i+j}\big)$, and finally consecutive sequences of ``systems of units'' into ``blocks of systems'' $\BB=\big(\DoubleSS_{i},\, \DoubleSS_{i+1},\, \dots\,,\, \DoubleSS_{i+j}\big)$. At the end, this construction will lead to the validity of~(\ref{wewant}).\bigskip

We can now start our construction introducing the first category.

\begin{definition}
Let $0\leq i \leq j \leq N$ be such that $\{ i,\, i+1 ,\, \cdots \,,\, j-1,\, j\}$ is a maximal sequence with the property that $\u{B}_l$ is the same point for all $i\leq l \leq j$ (by ``maximal'' we mean that either $i=0$ or $\u{B}_{i-1}\neq \u{B}_i$, as well as either $j=N$ or $\u{B}_j\neq \u{B}_{j+1}$). We will then say that $\U=\big(\T_{i+1},\, \T_2,\, \dots\,,\, \T_{j+1}\big)$ is a \emph{unit of triangles}, where $j+1$ is substituted by $j$ if $j=N$, and then no unit is defined if $i=j=N$. To any unit we associate two angles, namely,
\begin{align*}
\theta^+:=  \angle{\u{A}_{i}}{\u{B}_{i}}{\u{A}_j}\,, &&
\theta^-:= \angle{\u{B}_{j}}{\u{A}_j}{\u{B}_{j+1}}\,,
\end{align*}
with the convention that $\theta^-=0$ if $j=N$.
\end{definition}

\begin{figure}[htbp]
\begin{center}
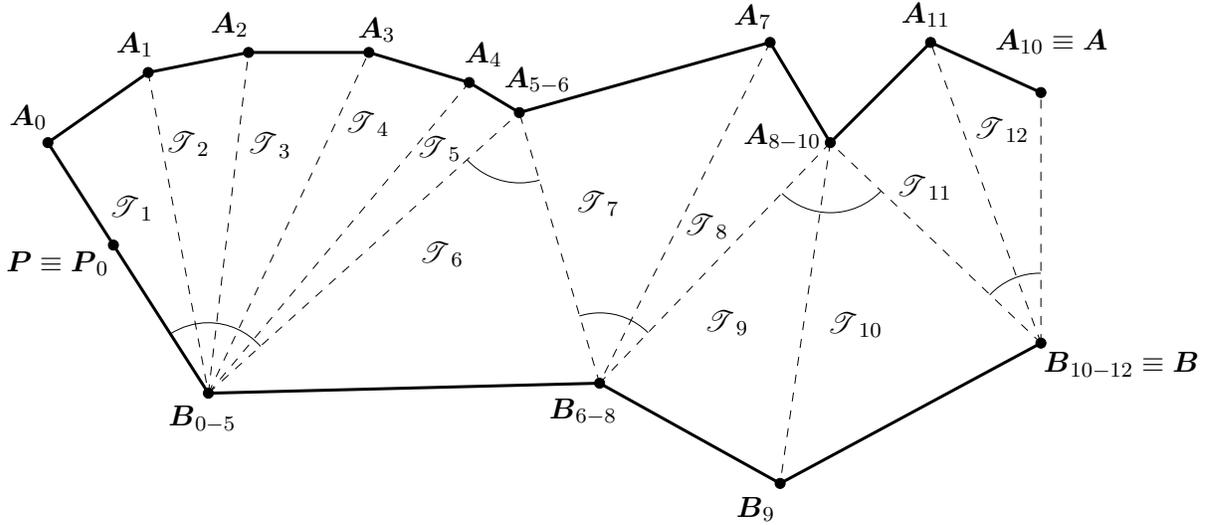\vspace{-10pt}
\caption{A natural sequence of triangles $\T_i$ with the points $\u{A}_i$ and $\u{B}_i$ and the angles $\theta^\pm$.}\label{Fig:units}
\end{center}
\end{figure}

The reason for this strange definition with $i+1$ and $j+1$ will soon become clear. The meaning of the definition is quite simple: the first unit starts with $\T_1$ and ends with $\T_j$, where $j$ is the smaller index such that $\u{B}_j\neq \u{B}_1$. The second unit starts with $\T_{j+1}$ and ends with $\T_{j'}$, where $j'$ is the smaller index, possibly $j+1$ itself, for which $\u{B}_j\neq \u{B}_{j'}$. And so on, until one reaches $\T_N$, and then one has to stop regardless of whether or not $\u{B}_N$ is different from $\u{B}_{N-1}$. It is immediate from the definition to observe that the sequence of triangles $\big(\T_1,\, \T_2,\, \dots\,,\, \T_N\big)$ is the concatenation of the units of triangles. To understand how the units work, it can be useful to check the example of Figure~\ref{Fig:units}, where $N=12$ and the units of triangles are $\big( \T_1,\, \T_2,\,\T_3,\,\T_4,\,\T_5,\,\T_6\big)$, $\big( \T_7,\,\T_8,\,\T_9\big)$, $\big( \T_{10}\big)$ and $\big( \T_{11},\,\T_{12}\big)$. Notice also that for any unit of triangles one has $\theta^+>0$, unless the unit is made by a single triangle, as $\big(\T_{10}\big)$ in the figure. Similarly, one has that $\theta^->0$, unless $j=N$ and $\u{B}_j=\u{B}_{j-1}$, as $\big( \T_{11},\,\T_{12}\big)$ in the figure.\par

The role of the units is contained in the following result.

\begin{lemma}\label{lemmaunits}
Let $\U=\big(\T_i,\, \T_{i+1},\, \dots ,\, \T_j\big)$ be a unit of triangles. Then one has
\begin{align}
\ell\big(\arc{\u{P}_{i-1}\u{P}_j}\big)&\leq \big(1+\theta^+\big)\,\ell\big(\u{P}_{i-1}\u{B}_{i-1}\big)-\ell\big(\u{P}_j\u{B}_j\big)+5\, \ell\big(\u{B}_{i-1}\u{B}_j\big)\,, \label{step51} \\
\ell\big(\u{B}_{i-1}\u{B}_j\big) &\geq \frac{\theta^-}{\pi}\,\ell\big(\u{P}_j\u{B}_j\big)\,,\label{step52}\\
\ell\big(\u{P}_j\u{B}_j\big) &\leq \ell\big(\u{P}_{i-1}\u{B}_{i-1}\big) + \ell\big(\u{B}_{i-1}\u{B}_j\big)\,.\label{step53}
\end{align}
\end{lemma}
\begin{proof}
The proof will follow from simple geometric considerations thanks to Lemma~\ref{lemmastep4}. To help the reader, the situation is depicted in Figure~\ref{Fig:lemmaunits}. First of all, one has by definition
\begin{equation}\label{units10}
\ell\big(\arc{\u{P}_{i-1}\u{P}_j}\big) = \ell\big(\arc{\u{P}_{i-1}\u{P}_{j-1}}\big) + \ell\big(\u{P}_{j-1}\u{P}_j\big)\,.
\end{equation}
We claim that
\begin{equation}\label{units11}
\ell\big(\arc{\u{P}_{i-1}\u{P}_{j-1}}\big) \leq \big(1+\theta^+ \big)\,\ell\big(\u{P}_{i-1}\u{B}_{i-1}\big)-\ell\big(\u{P}_{j-1}\u{B}_{i-1}\big)\,.
\end{equation}
\begin{figure}[htbp]
\begin{center}
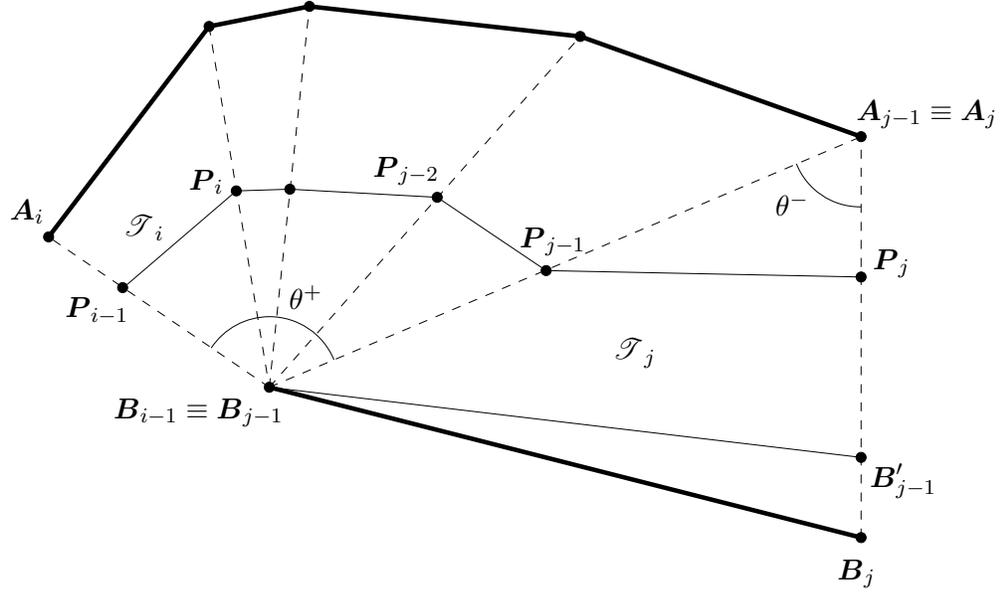\vspace{-10pt}
\caption{Situation in Lemma~\ref{lemmaunits}.}\label{Fig:lemmaunits}
\end{center}
\end{figure}
In fact, if $i=j$ then $\ell\big(\arc{\u{P}_{i-1}\u{P}_{j-1}}\big)=0$ and then~(\ref{units11}) is trivially true. Otherwise, let us consider the triangle $\u{P}_{i-1}\u{B}_{i-1}\u{P}_i$. Thanks to property~(iii) in Lemma~\ref{lemmastep4}, one has
\[
\ell\big(\u{P}_i\u{B}_{i-1}\big)\leq \ell\big(\u{P}_{i-1}\u{B}_{i-1}\big)\,,
\]
and then an immediate trigonometric argument tells us that
\[\begin{split}
\ell\big(\u{P}_{i-1}\u{P}_i\big) &\leq 2 \ell\big(\u{P}_{i-1}\u{B}_{i-1}\big) \sin \bigg( \frac{\angle{\u{P}_{i-1}}{\u{B}_{i-1}}{\u{P}_i}}{2} \bigg)+ \ell\big(\u{P}_{i-1}\u{B}_{i-1}\big)-\ell\big(\u{P}_i\u{B}_{i-1}\big) \\
&\leq \ell\big(\u{P}_{i-1}\u{B}_{i-1}\big)\cdot \angle{\u{P}_{i-1}}{\u{B}_{i-1}}{\u{P}_i}\ + \ell\big(\u{P}_{i-1}\u{B}_{i-1}\big)-\ell\big(\u{P}_i\u{B}_{i-1}\big)\,.
\end{split}\]
We can repeat the same argument more in general. In fact, for any $i \leq l \leq j-1$ one has from Lemma~\ref{lemmastep4} that
\begin{equation}\label{decreasing}
\ell\big(\u{P}_l\u{B}_{i-1}\big)\leq \ell\big(\u{P}_{l-1}\u{B}_{i-1}\big) \leq \cdots \leq \ell\big(\u{P}_{i-1}\u{B}_{i-1}\big)\,,
\end{equation}
hence the previous trigonometric argument implies
\[
\ell\big(\u{P}_{l-1}\u{P}_l\big) 
\leq \ell\big(\u{P}_{i-1}\u{B}_{i-1}\big)\cdot \angle{\u{P}_{l-1}}{\u{B}_{i-1}}{\u{P}_l}\ + \ell\big(\u{P}_{l-1}\u{B}_{i-1}\big)-\ell\big(\u{P}_l\u{B}_{i-1}\big)\,.
\]
Adding this inequality for all $i\leq l \leq j-1$ one gets
\[\begin{split}
\ell\big(\arc{\u{P}_{i-1}\u{P}_{j-1}}\big) &= \sum_{l=i}^{j-1} \ell\big(\u{P}_{l-1}\u{P}_l\big)\\
&\leq \sum_{l=i}^{j-1} \ell\big(\u{P}_{i-1}\u{B}_{i-1}\big)\cdot \angle{\u{P}_{l-1}}{\u{B}_{i-1}}{\u{P}_l}\ + \ell\big(\u{P}_{l-1}\u{B}_{i-1}\big)-\ell\big(\u{P}_l\u{B}_{i-1}\big)\\
&= \theta^+ \ell\big(\u{P}_{i-1}\u{B}_{i-1}\big) +\ell\big(\u{P}_{i-1}\u{B}_{i-1}\big)-\ell\big(\u{P}_{j-1}\u{B}_{i-1}\big)\,,
\end{split}\]
which is~(\ref{units11}).\par
Let us now point our attention to the triangle $\T_j$. First of all, let us call $\u{H}$ (resp. $\u{B}_\perp$) the orthogonal projection of $\u{P}_{j-1}$ (resp. $\u{B}_{i-1}$) on the straight line passing through $\u{A}_j\u{B}_j$ (these two points are not indicated in the figure, for the sake of clarity). Since by~(i) of Lemma~\ref{lemmastep4} we have $\angle{\u{P}_{j-1}}{\u{P}_j}{\u{H}}\geq 15^\circ$, it is
\begin{equation}\label{units12}
\ell\big(\u{P}_{j-1}\u{P}_j\big) = \frac{\ell\big(\u{P}_{j-1}\u{H}\big)}{\sin\Big(\angle{\u{P}_{j-1}}{\u{P}_j}{\u{H}}\Big)}
\leq \frac{1}{\sin 15^\circ}\,\ell\big(\u{P}_{j-1}\u{H}\big) \leq 4 \,\ell\big(\u{P}_{j-1}\u{H}\big)\,,
\end{equation}
and similarly
\begin{equation}\begin{split}\label{units13}
\ell\big(\u{B}_{i-1}\u{B}_j\big) &\geq \ell\big(\u{B}_{i-1}\u{B}_\perp\big) = \ell\big(\u{A}_{j-1}\u{B}_{i-1}\big)\sin\theta^-
\geq \ell\big(\u{P}_{j-1}\u{B}_{i-1}\big) \sin\theta^-\\
&\geq \frac{2\theta^-}{\pi}\,\ell\big(\u{P}_{j-1}\u{B}_{i-1}\big)\,,
\end{split}\end{equation}
recalling that by definition of the triangles of the sectors one has $\theta^-\leq \pi/2$. Moreover, since $\u{P}_{j-1}\in \u{A}_{j-1}\u{B}_{i-1}$, then clearly $\ell\big(\u{P}_{j-1}\u{H}\big)\leq \ell\big(\u{B}_{i-1}\u{B}_\perp\big)$, so~(\ref{units12}) and~(\ref{units13}) imply
\begin{equation}\label{willimply}
\ell\big(\u{P}_{j-1}\u{P}_j\big) \leq 4\, \ell\big(\u{B}_{i-1}\u{B}_j\big)\,.
\end{equation}
Let us now call, as in the figure, $\u{B}'_{j-1}$ the first point of the piecewise affine path which starts from $\u{B}_{j-1}$ and arrives to $\u{AB}$ according to Lemma~\ref{lemmastep4} --with the notation of Lemma~\ref{lemmastep4} we should have called that point $(\u{B}_{j-1})_1$. Applying twice condition~(iii) of Lemma~\ref{lemmastep4} we get
\[
\ell\big(\u{P}_j\u{B}_j\big) = \ell\big(\u{P}_j\u{B}'_{j-1}\big) + \ell\big(\u{B}'_{j-1}\u{B}_j\big)
\leq \ell\big(\u{P}_{j-1}\u{B}_{i-1}\big) + \ell\big(\u{B}_{i-1}\u{B}_j\big)\,.
\]
This inequality allows us to conclude. Indeed, together with~(\ref{units10}), (\ref{units11}) and~(\ref{willimply}) it concludes the proof of~(\ref{step51}). Then, together with~(\ref{decreasing}), it yields~(\ref{step53}). And finally, together with~(\ref{units13}), it gives~(\ref{step52}) since
\[\begin{split}
2\ell\big(\u{B}_{i-1}\u{B}_j\big) 
&\geq \frac{2\theta^-}{\pi}\, \ell\big(\u{B}_{i-1}\u{B}_j\big) + \ell\big(\u{B}_{i-1}\u{B}_j\big)\\
&\geq \frac{2\theta^-}{\pi}\, \Big(\ell\big(\u{P}_j\u{B}_j\big) - \ell\big(\u{P}_{j-1}\u{B}_{j-1}\big) \Big)+ 
\frac{2\theta^-}{\pi}\,\ell\big(\u{P}_{j-1}\u{B}_{j-1}\big)
= \frac{2\theta^-}{\pi}\, \ell\big(\u{P}_j\u{B}_j\big) \,.
\end{split}\]
\end{proof}
After this result, we can stop thinking about triangles, and we can start working only with units. In fact, notice that any unit of triangles, say $\U=\big(\T_i,\, \T_{i+1},\, \dots \,,\,\T_j\big)$, starts with the exit side of $\T_{i-1}$ and finishes with the exit side of $\T_j$ and that the estimates~(\ref{step51}), (\ref{step52}) and~(\ref{step53}) are already written only in terms of points of those sides. Let us then number the units as $\U_1,\,\U_2,\, \dots \,,\, \U_M$, with $M\leq N$, and let us define $i_l$ and $j_l$, for $1\leq l \leq M$, in such a way that $\U_l = \big(\T_{i_l},\, \T_{i_l+1},\, \dots \,,\,\T_{j_l}\big)$. Notice that $i_1=1$, $j_M=N$, and $j_l+1=i_{l+1}$ for each $1\leq l < M$. Let us give the following definitions,
\begin{align}\label{notunits}
\u{Q}_l := \u{P}_{j_l} \,, && \u{C}_l := \u{A}_{j_l}\,, && \u{D}_l:= \u{B}_{j_l}\,, && \u{Q}_0:=\u{P}_0=\u{P}\,, && \u{D}_0:= \u{B}_0 \,,
\end{align}
where the last two definitions are done to be consistent. Call also $\theta^\pm_l$ the angles $\theta^\pm$ related to the unit $\U_l$. Hence, the claim of Lemma~\ref{lemmaunits} can be rewritten as
\begin{align}
\ell\big(\arc{\u{Q}_{l-1}\u{Q}_l}\big)&\leq \big(1+\theta^+_l\big)\,\ell\big(\u{Q}_{l-1}\u{D}_{l-1}\big)-\ell\big(\u{Q}_l\u{D}_l\big)+5\, \ell\big(\u{D}_{l-1}\u{D}_l\big)\,,\tag{\ref{step51}'}\label{step51'} \\
\ell\big(\u{D}_{l-1}\u{D}_l\big) &\geq \frac{\theta^-_l}{\pi}\,\ell\big(\u{Q}_l\u{D}_l\big)\,,\tag{\ref{step52}'}\label{step52'}\\
\ell\big(\u{Q}_l\u{D}_l\big) &\leq \ell\big(\u{Q}_{l-1}\u{D}_{l-1}\big) + \ell\big(\u{D}_{l-1}\u{D}_l\big)\,.\tag{\ref{step53}'}\label{step53'}
\end{align}
Before passing to the definition of ``systems'' of units, and in order to help understanding its meaning, it can be useful to give a proof of Lemma~\ref{lemmastep5} in a very peculiar case.
\begin{lemma}\label{intrsystem}
The claim of Lemma~\ref{lemmastep5} holds true if
\begin{align}
\ell\big(\arc{\u{D}_0\u{D}_{M-1}}\big)&\leq \frac{\ell\big(\u{Q}_0\u{D}_0\big)} 4 \,,\label{ass1}\\
\ell\big(\u{Q}_l\u{D}_l\big)&\geq \frac{\ell\big(\u{Q}_0\u{D}_0\big)} 2 \qquad \forall\,1\leq l\leq M-1 \,.\label{ass2}
\end{align}
\end{lemma}
\begin{proof}
First of all notice that, by the two assumptions and an easy geometrical argument (recalling that all the triangles $\T_i$ are disjoint, hence in particular the segments $\u Q_l\u D_l$ cannot intersect), one finds that
\begin{equation}\label{boundangles}
\sum_{l=1}^M \theta_l^+ - \sum_{l=1}^{M-1}\theta_l^- \leq \frac{13}{6}\,\pi\,.
\end{equation}
Moreover, by~(\ref{step53'}) and~(\ref{ass1}), one gets
\begin{equation}\label{segmcorti}
\ell\big(\u{Q}_l\u{D}_l\big) \leq \frac 54\, \ell\big(\u{Q}_0\u{D}_0\big)\qquad \forall\, 0\leq l \leq M-1\,.
\end{equation}
We can now evaluate, using~(\ref{step51'}), (\ref{segmcorti}), (\ref{boundangles}), (\ref{ass2}) and~(\ref{step52'}),
\begin{align}
\ell\big(\arc{\u{Q}_0\u{Q}_M}\big)
&=\sum_{l=1}^{M} \ell\big(\arc{\u{Q}_{l-1}\u{Q}_l}\big)
\leq\sum_{l=1}^{M}\big(1+\theta^+_l\big)\,\ell\big(\u{Q}_{l-1}\u{D}_{l-1}\big)-\ell\big(\u{Q}_l\u{D}_l\big)+5\,\ell\big(\u{D}_{l-1}\u{D}_l\big)\nonumber\\
&\leq \frac 54\, \ell\big(\u{Q}_0\u{D}_0\big) \sum_{l=1}^{M} \theta^+_l + \ell\big(\u{Q}_0\u{D}_0\big)-\ell\big(\u{Q}_M\u{D}_M\big)+5\ell\big(\arc{\u{D}_0\u{D}_M}\big)\nonumber\\
&\leq \ell\big(\u{Q}_0\u{D}_0\big) \bigg(1+\frac{65}{24}\,\pi\bigg) + \frac 54\, \ell\big(\u{Q}_0\u{D}_0\big) \sum_{l=1}^{M-1} \theta^-_l +5\ell\big(\arc{\u{D}_0\u{D}_M}\big)\label{gensch}\\
&\leq 10\,\ell\big(\u{Q}_0\u{D}_0\big)+ \frac 52\, \sum_{l=1}^{M-1} \theta^-_l\,\ell \big(\u{Q}_l\u{D}_l\big) +5\ell\big(\arc{\u{D}_0\u{D}_M}\big)\nonumber\\
&\leq 10\,\ell\big(\u{Q}_0\u{D}_0\big)+ \frac 52\,\pi\,  \sum_{l=1}^{M-1} \ell\big(\u{D}_{l-1}\u{D}_l\big) +5\ell\big(\arc{\u{D}_0\u{D}_M}\big)\nonumber\\
&\leq 10\,\ell\big(\u{Q}_0\u{D}_0\big)+ \big(5+\frac 52\,\pi\big)\ell\big(\arc{\u{D}_0\u{D}_M}\big)
\leq 10\,\ell\big(\u{Q}_0\u{D}_0\big)+ 13 \ell\big(\arc{\u{D}_0\u{D}_M}\big)\,.\nonumber
\end{align}
Finally, recall that
\[
\lcurve{PB}=\ell\big(\u{PB}_0\big)+ \ell\big(\arc{\u{B}_0\u{B}_N}\big)
= \ell\big(\u{Q}_0\u{D}_0\big)+ \ell\big(\arc{\u{D}_0\u{D}_M}\big)\,,
\]
hence from~(\ref{gensch}) we directly get $\ell\big( \arc{\u{PP}_N}\big)= \ell\big(\arc{\u{Q}_0\u{Q}_M}\big) \leq 13 \lcurve{PB}$. Since it is admissible to assume, by symmetry, that $\lcurve{PB}\leq \lcurve{AP}$, we conclude the proof of Lemma~\ref{lemmastep5} under the assumptions~(\ref{ass1}) and~(\ref{ass2}).
\end{proof}

It is to be noticed carefully that the key point in the above proof is the validity of~(\ref{boundangles}), which is a simple consequence of~(\ref{ass1}) and~(\ref{ass2}), but which one cannot hope to have in general. Basically, (\ref{boundangles}) fails whenever the sector $\DoubleS(\u{AB})$ has a spiral shape, and in fact~(\ref{ass1}) and~(\ref{ass2}) precisely prevent the sector to be an enlarging and a shrinking spiral respectively.\par
Since the assumptions~(\ref{ass1}) and~(\ref{ass2}) do not hold, in general, through all the units, we will group the units in ``systems'' in which they are valid.
\begin{definition}\label{defsystems}
Let $k_0=0$. We define recursively the increasing finite sequence $\{k_1,\, \cdots\,,\,  k_W\}$ as follows. For each $j\geq 0$, if $k_j=M$ then we conclude the construction (and thus $W=j$), while otherwise we define $k_j<k_{j+1}\leq M$ to be the biggest number such that
\begin{align}
\ell\big(\arc{\u{D}_{k_j}\u{D}_{k_{j+1}-1}}\big)&\leq \frac{\ell\big(\u{Q}_{k_j}\u{D}_{k_j}\big)} 4 \,,\tag{\ref{ass1}'}\label{ass1'}\\
\ell\big(\u{Q}_l\u{D}_l\big)&\geq \frac{\ell\big(\u{Q}_{k_j}\u{D}_{k_j}\big)} 2 \qquad \forall\,k_j < l < k_{j+1} \,.\tag{\ref{ass2}'}\label{ass2'}
\end{align}
Notice that the sequence is well-defined, since if $k_j<M$ then the assumptions~(\ref{ass1'}) and~(\ref{ass2'}) trivially hold with $k_{j+1}=k_j+1$. Hence, $W\leq M\leq N$. We define then \emph{system of units} each collection of units of the form $\DoubleSS_j=\big(\U_{k_{j-1}+1},\, \U_{k_{j-1}+2},\, \dots \,,\,\U_{k_j}\big)$, for $1\leq j \leq W$.
\end{definition}

Thanks to this definition, we can rephrase the claim of Lemma~\ref{intrsystem} as follows: ``the claim of Lemma~\ref{lemmastep5} holds true if there is only one system of units''. But in fact, the argument of Lemma~\ref{intrsystem} still gives some useful information for each different system, as we will see in a moment with Lemma~\ref{lemmasystem}. Before doing so, in order to avoid too many indices, it is convenient to introduce some new notation in order to work only with systems instead of with units. Hence, in analogy with~(\ref{notunits}), we set
\begin{align}\label{notsystems}
\u{R}_j := \u{Q}_{k_j} \,, && \u{E}_j := \u{C}_{k_j}\,, && \u{F}_j:= \u{D}_{k_j}\,, && \u{R}_0:=\u{Q}_0=\u{P}\,, && \u{F}_0:= \u{D}_0 = \u{B}_0\,.
\end{align}
We can now observe an estimate for the systems which comes directly from the argument of Lemma~\ref{intrsystem}.
\begin{lemma}\label{lemmasystem}
Let $\DoubleSS_j$ be a system of units. Then one has
\begin{equation}\label{stimasystem}
\ell\big(\arc{\u{R}_{j-1}\u{R}_j}\big) \leq 13 \,\ell\big(\arc{\u{F}_{j-1}\u{F}_j}\big) + 10\,\ell\big(\u{R}_{j-1}\u{F}_{j-1}\big)\,,
\end{equation}
and moreover
\begin{equation}\label{stepsystem}
\ell\big(\u{R}_j\u{F}_j\big) \leq \ell\big(\u{R}_{j-1}\u{F}_{j-1}\big) + \ell\big(\arc{\u{F}_{j-1}\u{F}_j}\big)\,.
\end{equation}
\end{lemma}
\begin{proof}
First of all, repeat {\it verbatim}, substituting $0$ with $k_{j-1}$ and $M$ with $k_j$, the proof of Lemma~\ref{intrsystem} until the estimate~(\ref{gensch}), which then reads as
\[
\ell\big(\arc{\u{Q}_{k_{j-1}}\u{Q}_{k_j}}\big) \leq 10\,\ell\big(\u{Q}_{k_{j-1}}\u{D}_{k_{j-1}}\big) + 13 \,\ell\big(\arc{\u{D}_{k_{j-1}}\u{D}_{k_j}}\big)\,.
\]
This estimate is exactly~(\ref{stimasystem}), rewritten with the new notations~(\ref{notsystems}).
On the other hand, concerning~(\ref{stepsystem}), it is enough to add the inequality~(\ref{step53'}) with all $k_{j-1}+1 \leq l \leq k_j$, thus obtaining
\[
\sum_{l=k_{j-1}+1}^{k_j} \ell\big(\u{Q}_l\u{D}_l\big) \leq \sum_{l=k_{j-1}+1}^{k_j} \ell\big(\u{Q}_{l-1}\u{D}_{l-1}\big) +\sum_{l=k_{j-1}+1}^{k_j}  \ell\big(\u{D}_{l-1}\u{D}_l\big)\,,
\]
which is equivalent to
\[
\ell\big(\u{Q}_{k_j}\u{D}_{k_j}\big)  \leq \ell\big(\u{Q}_{k_{j-1}}\u{D}_{k_{j-1}}\big) + \ell\big(\arc{\u{D}_{k_{j-1}} \u{D}_{k_j} }\big)\,.
\]
This estimate corresponds to~(\ref{stepsystem}) when using the new notations.
\end{proof}
Notice that, by adding~(\ref{stimasystem}) for all $1\leq j \leq W$, one obtains
\[
\ell\big( \arc{\u{PP}_N}\big)= \ell\big(\arc{\u{Q}_0\u{Q}_M}\big) = \ell\big(\arc{\u{R}_0\u{R}_W}\big)\leq 13\,\ell\big(\arc{\u{F}_0\u{F}_W}\big) + 10 \sum_{j=0}^{W-1} \ell\big(\u{R}_j\u{F}_j\big) \,,
\]
and since $\arc{\u{F}_0\u{F}_W}=\arc{\u{B}_0\u{B}_N}\subseteq \arc{\u{PB}}$, to conclude Lemma~\ref{lemmastep5} one needs to estimate the last sum.\par
Having done this remark, we can now introduce our last category, namely the ``blocks'' of systems. To do so, notice that by Definition~\ref{defsystems} of systems of units and using the new notations~(\ref{notsystems}), for any $1\leq j < W$ one must have, by maximality of $k_j$,
\begin{align}\label{conclsyst}
\hbox{either} \qquad \ell\big(\arc{\u{F}_{j-1}\u{F}_j}\big)> \frac{\ell\big(\u{R}_{j-1}\u{F}_{j-1}\big)} 4  \,, &&
\hbox{or} \qquad \ell\big(\u{R}_j\u{F}_j\big)< \frac{\ell\big(\u{R}_{j-1}\u{F}_{j-1}\big)} 2\,.
\end{align}
We can then give our definition.
\begin{definition}
Let $p_0=0$. We define recursively the increasing sequence $\{p_1,\, \cdots\,,\,  p_H\}$ as follows. For each $i\geq 0$, if $p_i=W$ then we conclude the construction (and thus $H=i$), while otherwise we define $p_i<p_{i+1}\leq W$ to be the biggest number such that
\[
\ell\big(\u{R}_j\u{F}_j\big)< \frac{\ell\big(\u{R}_{j-1}\u{F}_{j-1}\big)} 2 \qquad \forall\, p_i< j <p_{i+1}\,.
\]
Notice again that this strictly increasing sequence is well-defined since the inequality is emptily true for $p_{i+1}=p_i+1$. We then define \emph{block of systems} each collection $\BB_i=\big(\DoubleSS_{p_{i-1}+1},\, \DoubleSS_{p_{i-1}+2},\, \dots \,,\,\DoubleSS_{p_i}\big)$, for $1\leq i \leq H$.
\end{definition}
We can now show the important properties of the blocks of systems.
\begin{lemma}
For any $0\leq i < H$, the following estimate concerning the block $\BB_i$ holds true,
\begin{equation}\label{blocks1}
\ell\big(\arc{\u{R}_{p_i}\u{R}_{p_{i+1}}}\big) \leq 13\, \ell\big(\arc{\u{F}_{p_i} \u{F}_{p_{i+1}}}\big) + 20\, \ell\big(\u{R}_{p_i} \u{F}_{p_i}\big)\,.
\end{equation}
Moreover, for any $0\leq i<H-1$, one also has
\begin{equation}\label{blocks2}
\ell\big(\u{R}_{p_{i+1}}\u{F}_{p_{i+1}}\big)\leq 5\,\ell\big(\arc{\u{F}_{p_i}\u{F}_{p_{i+1}}}\big)\,.
\end{equation}
\end{lemma}
\begin{proof}
It is enough to add~(\ref{stimasystem}) for $p_i+1\leq j \leq p_{i+1}$ to obtain
\[\begin{split}
\ell\big(\arc{\u{R}_{p_i}\u{R}_{p_{i+1}}}\big)&=\sum_{j=p_i+1}^{p_{i+1}} \ell\big(\arc{\u{R}_{j-1}\u{R}_j}\big) \leq 13 \sum_{j=p_i+1}^{p_{i+1}} \ell\big(\arc{\u{F}_{j-1}\u{F}_j}\big) + 10 \sum_{j=p_i+1}^{p_{i+1}}\,\ell\big(\u{R}_{j-1}\u{F}_{j-1}\big)\\
&=  13\, \ell\big(\arc{\u{F}_{p_i}\u{F}_{p_{i+1}}}\big) + 10\, \sum_{j=p_i}^{p_{i+1}-1}\,\ell\big(\u{R}_j\u{F}_j\big)
<  13 \,\ell\big(\arc{\u{F}_{p_i}\u{F}_{p_{i+1}}}\big) + 20 \,\ell\big(\u{R}_{p_i}\u{F}_{p_i}\big)\,,
\end{split}\]
thus~(\ref{blocks1}) is already obtained.\par
Consider now~(\ref{blocks2}). Recalling the definition of the blocks, the maximality of $p_{i+1}$ tells us that either $p_{i+1}=W$ (and this is excluded by $i<H-1$) or
\[
\ell\big(\u{R}_{p_{i+1}}\u{F}_{p_{i+1}}\big)\geq  \frac{\ell\big(\u{R}_{p_{i+1}-1}\u{F}_{p_{i+1}-1}\big)} 2\,.
\]
Hence, keeping in mind~(\ref{conclsyst}) with $j=p_{i+1}$, we also have that
\[
\ell\big(\arc{\u{F}_{p_{i+1}-1}\u{F}_{p_{i+1}}}\big)> \frac{\ell\big(\u{R}_{p_{i+1}-1}\u{F}_{p_{i+1}-1}\big)} 4\,.
\]
Let us apply now~(\ref{stepsystem}) with $j=p_{i+1}$, to get
\[
\ell\big(\u{R}_{p_{i+1}}\u{F}_{p_{i+1}}\big) \leq \ell\big(\u{R}_{p_{i+1}-1}\u{F}_{p_{i+1}-1}\big) + \ell\big(\arc{\u{F}_{p_{i+1}-1}\u{F}_{p_{i+1}}}\big)
\leq 5\,\ell\big(\arc{\u{F}_{p_{i+1}-1}\u{F}_{p_{i+1}}}\big)
\leq 5\,\ell\big(\arc{\u{F}_{p_i}\u{F}_{p_{i+1}}}\big)\,,
\]
and so also~(\ref{blocks2}) is proved.
\end{proof}

We finally end this step with the proof of Lemma~\ref{lemmastep5}.

\proofof{Lemma~\ref{lemmastep5}}
By symmetry, we can assume that $\min \Big\{ \lcurve{AP},\, \lcurve{PB}\Big\}=\lcurve{PB}$. Using~(\ref{blocks1}) and~(\ref{blocks2}), we then estimate
\begin{align*}
\ell\big(\arc{\u{P}_0\u{P}_N}\big)
&=\ell\big(\arc{\u{Q}_0\u{Q}_M}\big)
=\ell\big(\arc{\u{R}_0\u{R}_W}\big)
=\sum_{i=0}^{H-1} \ell\big(\arc{\u{R}_{p_i}\u{R}_{p_{i+1}}}\big)\\
&\leq \sum_{i=0}^{H-1} 13\, \ell\big(\arc{\u{F}_{p_i} \u{F}_{p_{i+1}}}\big) + \sum_{i=0}^{H-1} 20\, \ell\big(\u{R}_{p_i} \u{F}_{p_i}\big)\\
&= 13\, \sum_{i=0}^{H-1} \ell\big(\arc{\u{F}_{p_i} \u{F}_{p_{i+1}}}\big)+20\, \ell\big(\u{R}_0 \u{F}_0\big) + 20\, \sum_{i=0}^{H-2} \ell\big(\u{R}_{p_{i+1}} \u{F}_{p_{i+1}}\big)\\
&\leq 13\, \sum_{i=0}^{H-1} \ell\big(\arc{\u{F}_{p_i} \u{F}_{p_{i+1}}}\big)+20\, \ell\big(\u{R}_0 \u{F}_0\big) + 100\, \sum_{i=0}^{H-2} \ell\big(\arc{\u{F}_{p_i} \u{F}_{p_{i+1}}}\big)\\
&\leq 113\,\sum_{i=0}^{H-1} \ell\big(\arc{\u{F}_{p_i} \u{F}_{p_{i+1}}}\big)+20\, \ell\big(\u{R}_0 \u{F}_0\big)
= 113\, \ell\big(\arc{\u{F}_0\u{F}_W}\big)+20\, \ell\big(\u{R}_0 \u{F}_0\big)\\
&= 113\, \ell\big(\arc{\u{B}_0 \u{B}_N}\big)+20\,\ell\big(\u{P}_0 \u{B}_0\big)
\leq 113\, \ell\big(\arc{\u{P}_0 \u{B}_N}\big) = 113\, \lcurve{PB}\,.
\end{align*}
\end{proof}

\bigstep{VI}{Setting the speed of the piecewise affine paths inside a sector}

Keep in mind that we have to define a piecewise affine path from $\u{P}$ to $\u{O}$ as the image under $v$ of the segment $PO\subseteq \D$. This path will start with the curve $\arc{\u{PP}_N}$ that we defined in Step~IV. However, sending the (beginning of the) segment $PO$ on the path $\arc{\u{PP}_N}$ at constant speed is not the right choice. Basically, the reason is the following: if two points $\u{P}$ and $\u{Q}$ in $\arc{\u{AB}}$ have distance $\eps>0$, the lengths of $\arc{\u{PP}_N}$ and of $\arc{\u{QQ}_N}$ may differ of $K\eps$ for any big constant $K$ (e.g., when $\DoubleS(\u{AB})$ has a spiral shape), thus if we use the constant speed in the definition of $v$ we end up with a piecewise affine function with triangles having arbitrarily small angles, thus with an arbitrarily large bi-Lipschitz constant. For this reason, we parameterize the paths $\arc{\u{PP}_N}$ with a non constant speed. Choosing the correct speed is precisely the aim of this step.\par\smallskip

Let us start with the definition of a ``possible speed function''.
\begin{definition}
Let $\DoubleS(\u{AB})$ be a sector, and let $\Sigma$ be the union of the paths $\arc{\u{PP}_N}$ for all the vertices $\u{P}$ of $\arc{\u{AB}}$ (such union is disjoint by Lemma~\ref{lemmastep4}). We say that $\tau:\Sigma\to\R^+$ is a \emph{possible speed function} if for any vertex $\u{P}\in\arc{\u{AB}}$ one has
\begin{itemize}
\item $\tau(\u{P}) = 0$\,,
\item for each vertex $\u{P}\in\arc{\u{AB}}$ and each $0\leq i < N(\u{P})$, the restriction of $\tau$ to the closed segment $\u{P}_i\u{P}_{i+1}$ is affine\,.
\end{itemize}
Moreover, for any $\u{S}$ belonging to the open segment $\u{P}_i\u{P}_{i+1}$, we shall write
\begin{equation}\label{deftau'}
\tau'(\u{S}) := \frac{\tau(\u{P}_{i+1})- \tau(\u{P}_i)}{\ell\big(\u{P}_i\u{P}_{i+1}\big)}\,.
\end{equation}
\end{definition}
To avoid misunderstandings in the following result, we point the reader's attention to the fact that, if one considers $\tau(\u{S})$ as the time at which the curve $\arc{\u{PP}_N}$ passes through $\u{S}$, then in fact $\tau'(\u{S})$ corresponds to the \emph{inverse} of the speed of the curve. Let us now state and prove the main result of this step.

\begin{lemma}\label{lemmastep6}
There exists a possible speed function $\tau$ such that
\begin{align}
&\frac 1{60L} \leq \tau'(\u{S}) \leq 1 \qquad \forall\,\u{S}\in \Sigma\,,&&\phantom{a}\label{cond1}\\
&\begin{array}{c}
\hbox{if $\u{P}_i$ and $\u{Q}_j$ belong to the same exit side of a triangle, then}\\
|\tau(\u{P}_i) - \tau(\u{Q}_j)|\leq 170 L\, \lcurve{PQ}\,.
\end{array}\label{cond2}
\end{align}
\end{lemma}
\begin{proof}
We start noticing that, in order to define $\tau$, it is enough to fix $\tau'$ within the whole path $\arc{\u{PP}_N}$ for any vertex $\u{P}\in\arc{\u{AB}}$. We argue again by induction on the weight of the sector.
\case{I}{The weigth of $\DoubleS(\u{AB})$ is $2$.}
In this case, the sector is a triangle $\u{ABC}$, and we directly set $\tau'\equiv 1$ within all $\Sigma$, so that~(\ref{cond1}) is clearly true. Consider now~(\ref{cond2}). Since there is only a single triangle, then necessarily $i=j=1$ and $\u{P}_1$ and $\u{Q}_1$ belong to $\u{AB}$, so that
\begin{align*}
\tau(\u{P}_1) = \ell\big(\u{PP}_1\big)\,, && \tau(\u{Q}_1) = \ell\big(\u{QQ}_1\big)\,,
\end{align*}
by the choice $\tau'\equiv 1$. It is then enough to recall Lemma~\ref{lemmastep4} (iii) and to use the triangular inequality to get
\[
|\tau(\u{P}_1) - \tau(\u{Q}_1)| = \big|\ell\big(\u{PP}_1\big) - \ell\big(\u{QQ}_1\big)\big|
\leq \lsegm{PQ} + \ell\big(\u{P}_1\u{Q}_1\big)
\leq 2\lsegm{PQ}\,,
\]
so that~(\ref{cond2}) holds true.

\case{II}{The weigth of $\DoubleS(\u{AB})$ is at least $3$.}

In this case, let us consider the maximal triangle $\u{ABC}$. Then, we can assume that $\tau$ has been already defined in the sectors $\DoubleS(\u{AC})$ and $\DoubleS(\u{BC})$, emptily if the segment $\u{AC}$ (resp. $\u{BC}$) belongs to $\partial\Delta$, and by inductive assumption otherwise, and with the properties that $1/60L \leq \tau'(\u{S}) \leq 1$ for every $\u{S}\in\DoubleS(\u{AC})\cup\DoubleS(\u{BC})$, and that
\begin{equation}
\big|\tau(\u{P}_{N-1}) - \tau(\u{Q}_{M-1})\big| \leq 170 L\, \ell\big(\u{PQ}\big) \label{indass}\\
\end{equation}
for every $\u{P},\,\u{Q}\in\arc{\u{AB}}$. Here we write for brevity $N=N(\u{P})$ and $M=N(\u{Q})$, so that both $\u{P}_{N-1}$ and $\u{Q}_{M-1}$ belong to $\u{AC}\cup \u{BC}$. Notice that~(\ref{indass}) follows by inductive assumption even if $\u{P}_{N-1}\in \u{AC}$ and $\u{Q}_{M-1}\in \u{BC}$, just applying~(\ref{cond2}) once to $\u{P}_{N-1}$ and $\u{C}$, and once to $\u{Q}_{M-1}$ and $\u{C}$.\par
Thus, we only have to define $\tau$ in the triangle $\u{ABC}$ and by definition of possible speed function it is enough to set $\tau$ on the segment $\u{AB}$ or, equivalently, to set $\tau'$ on the triangle $\u{ABC}$.\par
Let us begin with a temptative definition, namely, we define $\tt$ by putting $\tt'\equiv 1/60 L$ in $\u{ABC}$, and we will define $\tau$ as a modification --if necessary-- of $\tt$. Notice that, for any $\u{P}_{N-1}\in \u{AC}\cup\u{BC}$, our definition consists in setting
\begin{equation}\label{deftt}
\tt(\u{P}_N) = \tau(\u{P}_{N-1}) + \frac{1}{60L}\,\ell\big( \u{P}_{N-1}\u{P}_N\big)\,.
\end{equation}
Of course the function $\tt$ satisfies~(\ref{cond1}), but in general it is not true that~(\ref{cond2}) holds. 

We can now define the function $\tau$ by setting
\begin{equation}\label{deftau}
\tau(\u{P}_N) := \tt(\u{P}_N) \vee \max \Big\{ \tt(\u{Q}_M) - 170 L \,\ell\big(\arc{\u{PQ}}\big):\, \u{Q} \in \arc{\u{AB}}\Big\} \,,
\end{equation}
for any vertex $\u{P}\in\arc{\u{AB}}$. Since by definition $\tau\geq \tt$, it is also $\tau'\geq \tt' = 1/60L$ in the triangle $\u{ABC}$, so the first inequality in~(\ref{cond1}) holds true also for $\tau$.\par

It is also easy to check~(\ref{cond2}). Indeed, take $\u{P}$ and $\u{Q}$ in $\arc{\u{AB}}$, and consider two possibilities: if $\tau(\u{Q}_M) = \tt(\u{Q}_M)$, then
\begin{equation}\label{bothcases}
\tau(\u{P}_N) \geq \tt(\u{Q}_M) - 170 L \, \ell\big(\arc{\u{PQ}}\big) = \tau(\u{Q}_M)-170 L \, \ell\big(\arc{\u{PQ}}\big)\,.
\end{equation}
On the other hand, if $\tau(\u{Q}_M)=\tt(\u{R}_K) - 170 L \, \ell\big(\arc{\u{QR}}\big)$ for some $\u{R}\in\arc{\u{AB}}$ with $K=N(\u{R})$, then
\[\begin{split}
\tau(\u{P}_N) &\geq \tt(\u{R}_K) - 170 L \, \ell\big(\arc{\u{PR}}\big)
\geq \tt(\u{R}_K) - 170 L \,\ell\big(\arc{\u{PQ}}\big)- 170 L\,\ell\big(\arc{\u{QR}}\big)\\
&= \tau(\u{Q}_M) - 170 L\,\ell\big(\arc{\u{PQ}}\big)\,,
\end{split}\]
so that~(\ref{bothcases}) is true in both cases. Exchanging the roles of $\u{P}$ and $\u{Q}$ immediately yields~(\ref{cond2}). Summarizing, to conclude the thesis we only have to check that $\tau'\leq 1$ on $\u{ABC}$, which by induction amounts to check that for any $\u{P}\in \arc{\u{AB}}$ one has
\[
\tau(\u{P}_N) - \tau(\u{P}_{N-1}) \leq \ell\big(\u{P}_{N-1}\u{P}_N\big) \,.
\]
Let us then assume the existence of some vertex $\u{P}\in\arc{\u{AB}}$ such that
\begin{equation}\label{absurd}
\tau(\u{P}_N) - \tau(\u{P}_{N-1}) > \ell\big(\u{P}_{N-1}\u{P}_N\big) \,,
\end{equation}
and the searched inequality will follow once we find some contradiction. By symmetry, we assume that $\u{P}_{N-1}\in \u{AC}$. Of course, if $\tau(\u{P}_N)=\tt(\u{P}_N)$ then~(\ref{deftt}) already prevents the validity of~(\ref{absurd}). Therefore, keeping in mind~(\ref{deftau}), we obtain the existence of some vertex $\u{Q}\in\u{\arc{AB}}$ such that
\begin{equation}\label{condQ}
\tau(\u{P}_N)=\tt(\u{Q}_M) - 170 L\,\ell\big(\arc{\u{PQ}}\big)\,,
\end{equation}
which gives
\[
\tau(\u{P}_N)= \tau(\u{Q}_{M-1}) + \frac 1{60L}\, \ell\big(\u{Q}_{M-1}\u{Q}_M\big)- 170 L\,\ell\big(\arc{\u{PQ}}\big)\,.
\]
Recalling~(\ref{indass}) and~(\ref{absurd}), we deduce
\[\begin{split}
\tau(\u{P}_{N-1}) &\geq \tau(\u{Q}_{M-1}) - 170 L\,\ell\big(\arc{\u{PQ}}\big)
=\tau(\u{P}_N) - \frac 1{60L}\, \ell\big(\u{Q}_{M-1}\u{Q}_M\big)\\
&> \tau(\u{P}_{N-1}) + \ell\big(\u{P}_{N-1}\u{P}_N\big)- \frac 1{60L}\, \ell\big(\u{Q}_{M-1}\u{Q}_M\big)\,,
\end{split}\]
so that
\begin{equation}\label{fewlines}
\ell\big(\u{Q}_{M-1}\u{Q}_M\big) > 60 L\, \ell\big(\u{P}_{N-1}\u{P}_N\big)\,.
\end{equation}
\begin{figure}[htbp]
\begin{center}
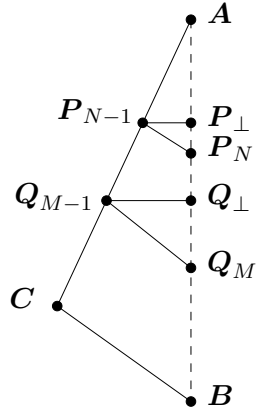\vspace{-10pt}
\caption{The triangle $\u{ABC}$ with the points $\u{P}_{N-1},\,\u{P}_N,\, \u{P}_\perp$ and $\u{Q}_{M-1},\,\u{Q}_M,\,\u{Q}_\perp$.}\label{Fig:VIcase1}
\end{center}
\end{figure}
Call now, as in Figure~\ref{Fig:VIcase1}, $\u{P}_\perp$ and $\u{Q}_\perp$ the orthogonal projections of $\u{P}_{N-1}$ and $\u{Q}_{M-1}$ on the segment $\u{AB}$, and note that by a trivial geometrical argument --recalling that $\u{P}_{N-1}\in\u{AC}$-- one has
\[
\frac{\ell\big(\u{P}_{N-1}\u{P}_\perp\big)}{\ell\big(\u{Q}_{M-1}\u{Q}_\perp\big)}
\geq \frac{\ell\big(\u{AP}_{N-1}\big)}{\ell\big(\u{AQ}_{M-1}\big)}\,,
\]
where the inequality is an equality if $\u{Q}_{M-1}\in\u{AC}$ as in the figure, while it is strict if $\u{Q}_{M-1}\in\u{BC}$. Then, recalling Lemma~\ref{lemmastep4}~(i) and~(\ref{fewlines}), one has
\[\begin{split}
\ell\big(\u{P}_{N-1}\u{P}_N\big) &\geq \ell\big(\u{P}_{N-1}\u{P}_\perp\big)
\geq \ell\big(\u{Q}_{M-1}\u{Q}_\perp\big)\,\frac{\ell\big(\u{AP}_{N-1}\big)}{\ell\big(\u{AQ}_{M-1}\big)}\\
&= \ell\big(\u{Q}_{M-1}\u{Q}_M\big)\,\sin\big(\angle{\u{Q}_{M-1}}{\u{Q}_M}{\u{A}}\big)\,\frac{\ell\big(\u{AP}_{N-1}\big)}{\ell\big(\u{AQ}_{M-1}\big)}
\geq \frac{1}{4}\,\ell\big(\u{Q}_{M-1}\u{Q}_M\big)\,\frac{\ell\big(\u{AP}_{N-1}\big)}{\ell\big(\u{AQ}_{M-1}\big)}\\
&> 15L\,\ell\big(\u{P}_{N-1}\u{P}_N\big)\,\frac{\ell\big(\u{AP}_{N-1}\big)}{\ell\big(\u{AQ}_{M-1}\big)}\,,
\end{split}\]
which means
\[
\ell\big(\u{AQ}_{M-1}\big) \geq 15 L\,\ell\big(\u{AP}_{N-1}\big)\,.
\]
Making again use of Lemma~\ref{lemmastep4}~(iii) and of the Lipschitz property of $u$, we then have
\[\begin{split}
\lcurve{PQ} &\geq \ell\big(\u{P}_{N-1}\u{Q}_{M-1}\big)
\geq \ell\big(\u{AQ}_{M-1}\big) - \ell\big(\u{AP}_{N-1}\big)
\geq 14L\,\ell\big(\u{AP}_{N-1}\big)
\geq 2\,\ell\big(\arc{AP}\big)\\
&\geq \frac 2 L\,\lcurve{AP}\,,
\end{split}\]
so that
\[
3\,\lcurve{PQ} \geq \bigg(1+ \frac 2 L\bigg)\,\lcurve{PQ}
\geq \frac 2 L\,\Big(\lcurve {AP} + \lcurve{PQ}\Big)
\geq \frac 2 L\,\lcurve{AQ}\,.
\]
Hence, by~(\ref{condQ}) and again by the Lipschitz property of $u$,
\begin{equation}\label{finalcontr}
\tt(\u{Q}_M) \geq 170 L\, \ell\big(\u{PQ}\big)
\geq \frac{340}{3}\, \lcurve{AQ}\,.
\end{equation}
On the other hand, by definition and inductive assumption,
\[
\tt(\u{Q}_M) = \tau(\u{Q}_{M-1}) + \frac 1{60L}\,\ell\big(\u{Q}_{M-1}\u{Q}_M\big)
\leq \ell\big(\arc{\u{QQ}_{M-1}}\big)+ \frac 1{60L}\,\ell\big(\u{Q}_{M-1}\u{Q}_M\big)
\leq \ell\big(\arc{\u{QQ}_M}\big)\,,
\]
which recalling Lemma~\ref{lemmastep5} of Step~V gives
\[
\tt(\u{Q}_M) \leq 113 \, \lcurve{AQ} <  \frac{340}{3}\, \lcurve{AQ}\,.
\]
Since this is in contradiction with~(\ref{finalcontr}), the proof of the lemma is concluded.
\end{proof}

\bigstep{VII}{Definition of the extension inside a primary sector}

We are finally ready to define the extension of $u$ inside a primary sector. The goal of this step is to take a primary sector $\DoubleS(\u{AB})$, being $\u{A}=u(A)$ and $\u{B}=u(B)$, with $A,\, B\in\partial\D$ as usual, and to define a piecewise affine bi-Lipschitz extension $u_{AB}$ of $u$ which sends a suitable subset $\D_{AB}$ of the square $\D$ onto $\DoubleS(\u{AB})$ (see Figure~\ref{Fig:step7}). First we observe a simple trigonometric estimate for the bi-Lipschitz constant of an affine map between two triangles and then we state and prove the main result of this step.
\begin{lemma}\label{liptriangle}
Let $\T$ and $\T'$ be two triangles in $\R^2$, and let $\phi$ be a bijective affine map sending $\T$ onto $\T'$. Call $a,\,b$ and $\alpha$ the lengths of two sides of $\T$ and the angle between them, and let $a',\, b'$ and $\alpha'$ be the correponding lengths and angle in $\T'$. Then, the Lipschitz constant of the map $\phi$ can be bounded as
\begin{equation}\label{lipschtriangle}
{\rm Lip} (\phi) \leq \frac {a'}a + \frac{b'\sin \alpha'}{b\sin\alpha} + \bigg| \frac{b'\cos \alpha'}{b\sin\alpha} - \frac{a'\cos\alpha}{a\sin\alpha}\bigg|
\leq \frac{a'}a + \frac{2b'}{b\sin\alpha} + \frac{a'}{a\sin\alpha}\,.
\end{equation}
\end{lemma}
\begin{proof}
Let us take an orthonormal basis $\{\ee_1,\,\ee_2\}$ of $\R^2$. Up to an isometry of the plane, we can assume that the two sides of lengths $a$ and $a'$ are both on the line $\{\ee_2=0\}$, that the two triangles $\T$ and $\T'$ both lie in the half-space $\{\ee_2\geq 0\}$ and that the vertices whose angles are given by $\alpha$, $\alpha'$ coincide with the point $(0,0)$. Hence, one has that $\phi(x) = M \, x + \omega$, for some vector $\omega\in\R^2$ and a $2\times 2$ matrix $M$. We have then
\[
{\rm Lip} (\phi) = | M | = \sup_{\nu\neq 0} \,\frac{|M \nu |}{|\nu|}\,.
\]
With our choice of coordinates, we have clearly
\begin{align*}
M\big(a, 0 \big) = \big(a' ,0 \big)\,, && M\big( b\cos \alpha, b \sin\alpha\big) = \big( b' \cos \alpha', b' \sin \alpha'\big)\,,
\end{align*}
which immediately gives
\[
M = \left( \begin{array}{cc}
\bal\frac{a'}{a}\eal\quad & \bal\frac{b'\cos\alpha'}{b\sin\alpha} - \frac{a'\cos\alpha}{a\sin\alpha}\eal \\[14pt]
0\quad & \bal\frac{b'\sin\alpha'}{b\sin\alpha}\eal
\end{array}\right)\,,
\]
from which the estimate~(\ref{lipschtriangle}) immediately follows.
\end{proof}

\begin{figure}[htbp]
\begin{center}
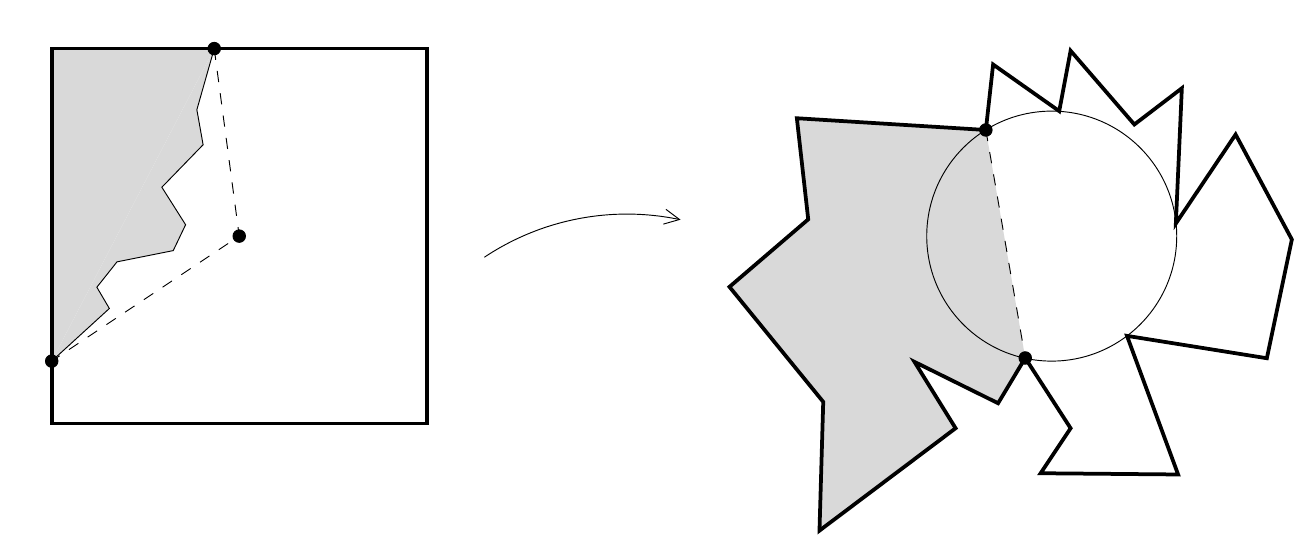\vspace{-10pt}
\caption{The function $u_{AB}:\D_{AB}\to \DoubleS(\u{AB})$.}\label{Fig:step7}
\end{center}
\end{figure}

\begin{lemma}\label{lemmastep7}
Let $\DoubleS(\u{AB})$ be a primary sector. Then there exists a polygonal subset $\D_{AB}$ of $\D$, and a piecewise affine map $u_{AB}:\D_{AB}\to \DoubleS(\u{AB})$ such that:
\begin{enumerate}
\item[(i)] for any $P\in \partial \D$, one has $\D_{AB}\cap OP = \emptyset$ if $P \notin \arc{AB}$, $\D_{AB}\cap OP = \{P\}$ if $P\in \{A,\, B\}$, and $\D_{AB}\cap OP=PP_N$ with $P_N=t O + (1-t) P$ and $0<t=t(P)\leq 4/5$ if $P \in\arc{AB}\setminus \{A,\, B\}\,.$
\item[(ii)] $u_{AB}=u$ on $\arc{AB}=\partial D\cap \D_{AB}$\,.
\item[(iii)] $u_{AB}$ is bi-Lipschitz with constant $212000 L^4$.
\item[(iv)] For any two consecutive vertices $P,\,Q\in\arc{AB}$, one has $\bal\angle{P_{N(P)}}{Q_{N(Q)}}{O}\geq\frac{1}{87L}\eal$.
\end{enumerate}
\end{lemma}
\begin{proof}
We will divide the proof in three parts.
\part{1}{Definition of $\Gamma$, $\u{\Gamma}$, $u_{AB}:\partial\Gamma\to \partial\u{\Gamma}$, and validity of~(i) and~(ii).}
First of all, we take a vertex $P\in\arc{AB}$ and, for any $1\leq i \leq N=N(\u{P})$, we set
\begin{align}\label{deft_i}
P_i = t_{P,i} \, O + (1-t_{P,i}) P\,, && {\rm with} && t_{P,i} =  \frac{\tau(\u{P}_i)}{10L}\,,
\end{align}
where $\tau$ is the function defined in Lemma~\ref{lemmastep6}. Then, we define $u_{AB}$ on the segment $PP_N$ as the piecewise affine function such that for all $i$ one has $u_{AB}(P_i)=\u{P}_i$. It is important to observe that
\begin{align}\label{under45}
0 \leq t_{P,i} \leq \frac 45 \,,&& \forall P\in \arc{AB},\, 1\leq i \leq N=N(\u{P})\,.
\end{align}
Indeed, using~(\ref{cond1}) in Lemma~\ref{lemmastep6}, (ii) in Lemma~\ref{lemmastep4}, and the Lipschitz property of $u$, one has that
\[
\tau(\u{P}_i)\leq \tau(\u{P}_N) \leq \sum_{j=1}^N \ell\big(\u{P}_{j-1}\u{P}_j\big) = \ell \big( \arc{\u{PP}_N}\big) \leq 4\,\lcurve{AB} \leq 4 L \,\ell\big(\arc{AB} \big) \leq 8 L\,,
\]
so by~(\ref{deft_i}) we get~(\ref{under45}).\par
We are now ready to define the set $\D_{AB}$. Let us enumerate, just for one moment, the vertices of $\arc{AB}$ as $P^0\equiv A,\, P^1,\, P^2 \,\dots \, ,\, P^{W-1},\, P^{W}\equiv B$, following the order of $\arc{AB}$. The set $\D_{AB}$ is then defined as the polygon whose boundary is the union of $\arc{AB}$ with the path $AP^1_{N(1)}P^2_{N(2)}\cdots P^{W-1}_{N(W-1)}B$, as in Figure~\ref{Fig:step7}, where for each $0< i < W$ we have written $N(i) = N(\u{P}^i)$. Hence, property~(i) is true by construction and by~(\ref{under45}).\par
\begin{figure}[htbp]
\begin{center}
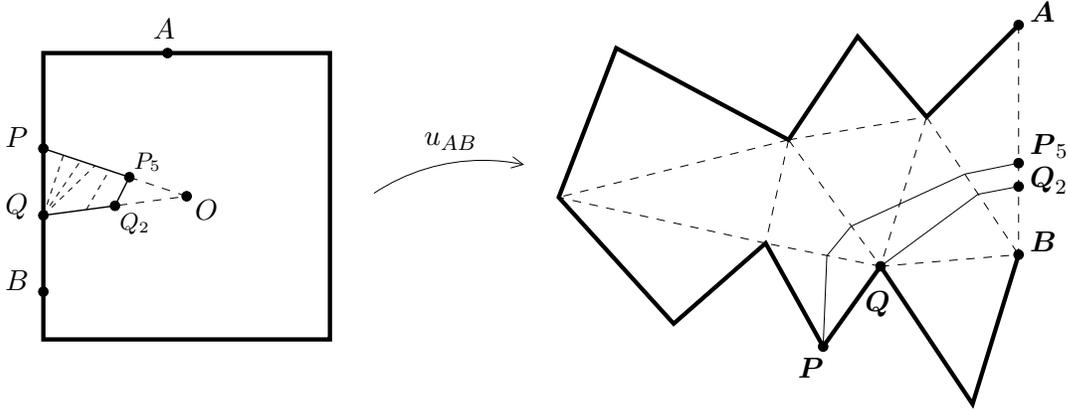\vspace{-10pt}
\caption{The sets $\Gamma$ and $\u{\Gamma}$.}\label{Fig:step72}
\end{center}
\end{figure}
Then we take two generic consecutive vertices $P,\, Q\in \arc{AB}$, and we call $\Gamma\subseteq \D_{AB}$ the quadrilater $PP_NQ_MQ$, and $\u{\Gamma}\subseteq \DoubleS(\u{AB})$ the polygon whose boundary is $\u{PQ} \cup \arc{\u{QQ}_M} \cup \u{Q}_M\u{P}_N\cup \arc{\u{P}_N\u{P}}$, where we have set $N=N(\u{P})$ and $M=N(\u{Q})$. Notice that, varying the consecutive vertices $P$ and $Q$, $\D_{AB}$ is the union of the different polygons $\Gamma$, while $\DoubleS(\u{AB})$ is the union of the polygons $\u{\Gamma}$. We will then define the function $u_{AB}$ so that $u_{AB} (\Gamma) = \u{\Gamma}$. Let us start with the definition of $u_{AB}$ from $\partial\Gamma$ to $\partial\u{\Gamma}$. The function $u_{AB}$ has been already defined from the segment $PP_N$ to the path $\arc{\u{PP}_N}$ and from the segment $QQ_M$ to the path $\arc{\u{QQ_M}}$. Hence we conclude defining $u_{AB}$ to be affine from the segment $PQ$ to the segment $\u{PQ}$, and from $P_NQ_M$ to $\u{P}_N\u{Q}_M$. Notice that, as a consequence, also property~(ii) is true by construction. 

Now we see how to extend $u_{AB}$ from the interior of $\Gamma$ to the interior of $\u{\Gamma}$ satisfying properties~(iii) and~(iv).\par
Recalling the partition of $\DoubleS(\u{AB})$ in triangles done in Step~III, $\u{PQ}$ is a side of some triangle $\u{PQR}$, and since $\u{PQ}\subseteq \partial\Delta$ it cannot be the exit side. Let us then assume, without loss of generality, that the exit side is $\u{QR}$. Hence, it follows that $N>M$. Moreover, if $\big(\T_1,\, \T_2,\, \dots\,,\, \T_N\big)$ is the natural sequence of triangles related to $\u{P}$, as in Definition~\ref{def: nat seq}, then it is immediate to observe that $\u{Q}$ belong to the exit side of $\T_i$ for all $1\leq i \leq N-M$. Figure~\ref{Fig:step72} shows an example in which $N=5$ and $M=2$. In the following two parts, we will define $u_{AB}$ separately on the triangle $PP_{N-M}Q$ and on the quadrilateral $P_{N-M}P_NQ_MQ$, whose union is $\Gamma$.

\part{2}{Definition of $u_{AB}$ in the triangle $PP_{N-M}Q$, and validity of~(iii) and~(iv).}

In this second part we define $u_{AB}$ from the triangle $PP_{N-M}Q$ to the polygon in $\Delta$ whose boundary is $\arc{\u{PP}_{N-M}}\cup \u{P}_{N-M}\u{Q} \cup \u{QP}$. The definition is very simple, namely, for any $0\leq i< N-M$ we let $u_{AB}$ be the affine function sending the triangle $P_iP_{i+1}Q$ onto the triangle $\u{P}_i\u{P}_{i+1}\u{Q}$, as shown in Figure~\ref{Fig:step73}. We now have to check the validity of~(iii) and~(iv) in the triangle $PP_{N-M}Q$. Keeping in mind Lemma~\ref{liptriangle}, to show~(iii) it is enough to compare the lengths of $P_iP_{i+1}$ and $\u{P}_i\u{P}_{i+1}$, those of $P_{i+1}Q$ and $\u{P}_{i+1}\u{Q}$, and the angles $\angle{P_i}{P_{i+1}}{Q}$ and $\angle{\u{P}_i}{\u{P}_{i+1}}{\u{Q}}$.\par
\begin{figure}[htbp]
\begin{center}
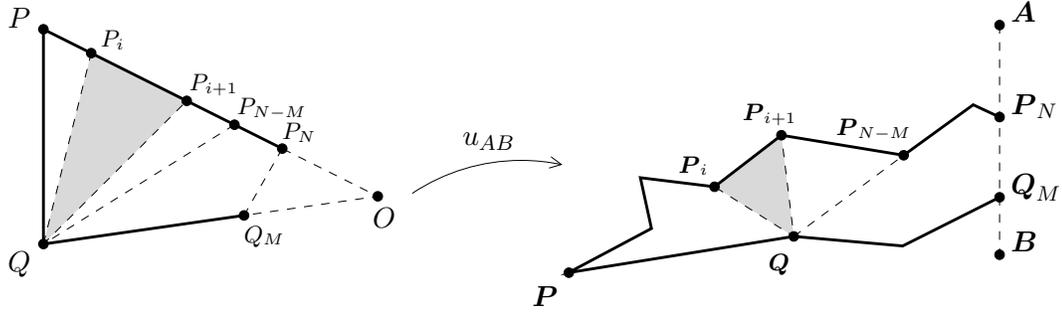\vspace{-10pt}
\caption{The situation in Part~2.}\label{Fig:step73}
\end{center}
\end{figure}
We start recalling that~(iii) in Lemma~\ref{lemmastep4}, together with the Lipschitz property of $u$, ensures
\begin{equation}\label{143}
\frac{\ell\big(PQ\big)}{7L} \leq \ell\big( \u{P}_{i+1}\u{Q}\big) \leq \ell\big(\u{PQ}\big)\leq L\ell\big(PQ\big)
\end{equation}
(keep in mind that, since $P$ and $Q$ are consecutive vertices, then $PQ=\arc{PQ}$ and $\u{PQ}=\arc{\u{PQ}}$). Recalling now~(\ref{cond2}) of Lemma~\ref{lemmastep6} and~(\ref{deft_i}), we get
\begin{equation}\label{142}
t_{P,i+1} = t_{P,i+1} - t_{Q,0} = \frac{\tau( \u{P}_{i+1}) - \tau(\u{Q}_0)}{10L} \leq 17 \,\lsegm{PQ}\leq 17 L\,\ell\big(PQ\big)\,.
\end{equation}
We want now to estimate $\ell\big(P_{i+1}Q\big)$. To do so, let us assume, as in Figure~\ref{Fig:step73} and without loss of generality, that $P$ and $Q$ belong to the left side of the square $\D$ and that $P$ is above $Q$. Call also $V\equiv (-\frac 12,-\frac12)$ the southwest corner of $\D$, and let $\delta_x$ and $\delta_y$ be the horizontal and vertical components of the vector $P_{i+1} - Q$, so that
\[
\ell\big(P_{i+1}Q\big) = \sqrt{\delta_x^2 + \delta_y^2}\,.
\]
By construction one clearly has $\delta_x = t_{P,i+1}/2$. We claim that
\begin{equation}\label{141}
\frac {\sqrt{2}}2 \, \ell\big(PQ\big) \leq \ell\big(P_{i+1}Q\big) \leq \frac{90}7 \, L\, \ell\big(PQ\big)\,.
\end{equation}
In fact, since $P_{i+1}$ belongs to the segment $PO$, then one surely has
\[
\ell\big(P_{i+1}Q\big) \geq \ell\big(PQ\big) \sin\big(\angle OPV\big) \geq \frac{\sqrt{2}}{2} \, \ell\big(PQ\big)\,,
\]
so that the left inequality in~(\ref{141}) holds. To show the right inequality in~(\ref{141}), notice that
\[
\frac 34\,\pi \geq \angle{P_{i+1}}PQ = \angle OPV \geq \frac \pi 4\,,
\]
so that by an immediate geometric argument $|\delta_y | \leq \ell(PQ) + \delta_x$. Thus, by~(\ref{142})
\begin{equation}\label{copy}\begin{split}
\ell\big(P_{i+1}Q\big) &= \sqrt{\delta_x^2 + \delta_y^2}
\leq \sqrt{\bigg(\frac{t_{P,i+1}}2\bigg)^2+ \bigg(\frac{t_{P,i+1}}2+\ell\big(PQ\big)}\bigg)^2\\
&\leq \ell\big(PQ\big) \sqrt{\bigg(\frac{17}2\, L\bigg)^2+ \bigg(\frac{17}2\, L+1\bigg)^2 }
\leq \frac{90} 7\, L\, \ell\big(PQ\big)\,,
\end{split}\end{equation}
and so also the right inequality in~(\ref{141}) is established.\par
Keeping in mind~(\ref{143}), from~(\ref{141}) we obtain
\begin{equation}\label{firstside}
\frac {\sqrt{2}}{2L}  \leq \frac{\ell\big(P_{i+1}Q\big)}{\ell\big(\u{P}_{i+1}\u{Q}\big)} \leq 90 L^2\,.
\end{equation}
It is much easier to compare $\ell\big(P_iP_{i+1}\big)$ and $\ell\big(\u{P}_i\u{P}_{i+1}\big)$. Indeed, by immediate geometrical argument, recalling~(\ref{deft_i}), (\ref{deftau'}) and condition~(\ref{cond1}) of Lemma~\ref{lemmastep6}, and letting $\u{S}$ be any point in the interior of $\u{P}_i\u{P}_{i+1}$, one has
\[\begin{split}
\ell\big(P_iP_{i+1}\big)  &\leq \frac{\sqrt{2}} 2\, \big(t_{P,i+1}- t_{P,i}\big) 
= \frac{\sqrt{2}} {20L}\, \big(\tau(\u{P}_{i+1})- \tau(\u{P}_i)\big)
= \frac{\sqrt{2}} {20L}\, \tau'(\u{S})\,\ell\big(\u{P}_i\u{P}_{i+1}\big)\\
&\leq \frac{\sqrt{2}} {20L}\, \ell\big(\u{P}_i\u{P}_{i+1}\big)\,,
\end{split}\]
and analogously
\[\begin{split}
\ell\big(P_iP_{i+1}\big) \geq \frac{t_{P,i+1}- t_{P,i}} 2 
=\frac{\tau(\u{P}_{i+1})- \tau(\u{P}_i)} {20L} 
=\frac{\tau'(\u{S})} {20L}\, \ell\big(\u{P}_i\u{P}_{i+1}\big)
\geq \frac{1} {1200L^2}\, \ell\big(\u{P}_i\u{P}_{i+1}\big)\,.
\end{split}\]
Thus, we have
\begin{equation}\label{secondside}
\frac{1} {1200L^2} \leq \frac{\ell\big(P_iP_{i+1}\big)}{\ell\big(\u{P}_i\u{P}_{i+1}\big)} \leq  \frac{\sqrt{2}} {20L}\,.
\end{equation}
Let us finally compare the angles $\angle{P_i}{P_{i+1}}{Q}$ and $\angle{\u{P}_i}{\u{P}_{i+1}}{\u{Q}}$. Concerning $\angle{\u{P}_i}{\u{P}_{i+1}}{\u{Q}}$, it is enough to recall~(i) of Lemma~\ref{lemmastep4} to obtain
\begin{equation}\label{firstangle}
15^\circ \leq \angle{\u{P}_i}{\u{P}_{i+1}}{\u{Q}} \leq 165^\circ\,.
\end{equation}
On the other hand, concerning $\angle{P_i}{P_{i+1}}Q$, we start observing
\begin{equation}\label{estangesy}
\angle{P_i}{P_{i+1}}Q = \angle P{P_{i+1}}Q \leq \pi - \angle OPQ \leq \frac 34\, \pi\,.
\end{equation}
To obtain an estimate from below to $\angle{P_i}{P_{i+1}}Q$, instead, we call for brevity $\alpha:= \angle {P_i}{P_{i+1}}Q=\angle P{P_{i+1}}Q$ and $\theta:= \angle OPV - \frac \pi 2 \in \big[ -\pi/4, \pi/4\big)$, so that an immediate trigonometric argument gives
\begin{equation}\label{1nstep}
\ell\big( PQ \big) = \frac{t_{P,i+1}}{2} \Big( \tan(\theta+\alpha) -\tan \theta \Big)\,.
\end{equation}
We aim then to show that
\begin{equation}\label{cmpangles}
\alpha \geq \frac{1}{19 L}\,.
\end{equation}
In fact, if
\[
\theta+\alpha\geq \frac \pi 4 + \frac 1 {19}\,,
\]
then since $\theta\leq \pi/4$ we immediately deduce the validity of~(\ref{cmpangles}). On the contrary, if
\[
\theta+\alpha  < \frac \pi 4 + \frac 1{19}\,,
\]
then recalling~(\ref{1nstep}), the fact that $\theta\geq -\pi/4$, and (\ref{142}), we get
\[\begin{split}
\ell\big( PQ \big) &= \frac{t_{P,i+1}}{2} \Big( \tan(\theta+\alpha) -\tan \theta \Big)
\leq \frac{t_{P,i+1}}{2}\, \frac{\alpha}{\cos^2 \Big( \frac \pi 4 + \frac 1 {19}\Big)}
\leq \frac{17}{2}\, L\,\ell\big(PQ\big) \,\frac{\alpha}{\cos^2 \Big( \frac \pi 4 + \frac 1 {19}\Big)}\,,
\end{split}\]
from which it follows
\[
\alpha \geq \frac{2\cos^2 \Big( \frac \pi 4 + \frac 1 {19}\Big)}{17L} \geq \frac{1}{19L}\,,
\]
so that~(\ref{cmpangles}) is concluded. Putting it together with~(\ref{estangesy}), we deduce
\begin{equation}\label{secondangle}
\frac{1}{19 L} \leq \angle {P_i}{P_{i+1}}Q \leq \frac 34\,\pi\,.
\end{equation}
Finally we show the validity of~(iii), simply applying~(\ref{lipschtriangle}) of Lemma~\ref{liptriangle}. Indeed, let us call $\phi$ the affine map which sends the triangle $P_iP_{i+1}Q$ onto $\u{P}_i\u{P}_{i+1}\u{Q}$ and, for brevity and according with the notation of Lemma~\ref{liptriangle}, let us write
\begin{align*}
a&=\ell\big(P_{i+1} Q\big) \,, & 
b&=\ell\big(P_iP_{i+1}\big)\,, & 
\alpha&=\angle {P_i}{P_{i+1}}Q\,, \\
a'&=\ell\big(\u{P}_{i+1}\u{Q}\big)\,, &
b'&=\ell\big(\u{P}_i\u{P}_{i+1}\big)\,,& 
\alpha'&=\angle {\u{P}_i}{\u{P}_{i+1}}{\u{Q}} \,.
\end{align*}
Then, the estimates~(\ref{firstside}), (\ref{secondside}), (\ref{firstangle}) and~(\ref{secondangle}) can be rewritten as
\begin{align}\label{ratiosII}
\frac{\sqrt{2}}{2L}\leq \frac a{a'} \leq 90 L^2\,, && \frac{1}{1200 L^2}\leq \frac b{b'}\leq \frac{\sqrt{2}}{20L}\,, && \sin\alpha' \geq \frac 14\,, && \sin\alpha \geq \frac 1{20L}\,,
\end{align}
where for the last estimate we used that
\begin{equation}\label{verysame}
\sin \alpha \geq \sin \bigg( \frac{1}{19L}\bigg) = \frac{1}{19L} \,  \bigg(19L \sin \bigg( \frac{1}{19L}\bigg) \bigg)
\geq  \frac{1}{19L} \,  \bigg(19 \sin \bigg( \frac{1}{19}\bigg) \bigg)
\geq  \frac{1}{20L}\,.
\end{equation}
Therefore, (\ref{lipschtriangle}) and~(\ref{ratiosII}) give us
\[
{\rm Lip} (\phi)
\leq \frac{a'}a + \frac{2b'}{b\sin\alpha} + \frac{a'}{a\sin\alpha}
\leq \sqrt 2\, L + 48000L^3  + 20\sqrt 2\, L^2\,.
\]
On the other hand, exchanging the roles of the triangles, we get
\[
{\rm Lip} (\phi^{-1})
\leq\frac a{a'} + \frac{2b}{b'\sin\alpha'} +\frac{a}{a'\sin\alpha'}
\leq 90 L^2 + \frac{2 \sqrt 2}{5 L}  + 360 L^2\,.
\]
\par

To conclude this part, we want to check~(iv) for the pairs of consecutive vertices $P,\,Q$ such that the side $P_NQ_M$ is in the triangle $PP_{N-M}Q$. Notice that this happens only when $M=0$, or in other words, if $Q\equiv A$ or $Q\equiv B$. Let us then assume that $Q$ is either $A$ or $B$, and let us show that~(iv) holds, that is,
\begin{align}\label{nitsch2}
\angle Q {P_N} O \geq \frac{1}{87L}\,, &&
\angle {P_N} Q O \geq \frac{1}{87L}\,.
\end{align}
Taking $i=N-1$ and applying the second inequality in~(\ref{secondangle}), we immediately find
\[
\angle Q{P_N} O = \pi - \angle {P_{N-1}} {P_N} Q \geq \frac \pi 4 > \frac 1{87L}\,.
\]
In the same way, applying the first inequality in~(\ref{secondangle}) and recalling Remark~\ref{angle<1}, one has
\[
\angle {P_N} Q O = \pi - \angle Q {P_N} O - \angle Q O {P_N}
= \angle {P_{N-1}} {P_N} Q - \angle P O Q \geq \frac 1{19L} - \frac{1}{50L} > \frac 1{87L}\,.
\]
Hence, (\ref{nitsch2}) is checked.

\part{3}{Definition of $u_{AB}$ in the quadrilateral $P_{N-M}P_NQ_MQ$, and validity of~(iii) and~(iv).}

The definition is again trivial: we take any $N-M \leq i<N$ and, setting $j=i-N+M\in [0,M)$, we have to send the quadrilateral $P_iP_{i+1}Q_{j+1}Q_j$ on the quadrilateral $\u{P}_i\u{P}_{i+1}\u{Q}_{j+1}\u{Q}_j$. To do so, we send the triangle $P_iP_{i+1}Q_{j+1}$ (resp. $Q_{j+1}Q_jP_i$) onto the triangle $\u{P}_i\u{P}_{i+1}\u{Q}_{j+1}$ (resp. $\u{Q}_{j+1}\u{Q}_j\u{P}_i$) in the bijective affine way, as depicted in Figure~\ref{Fig:step74}. Then, we have to check the validity of~(iii) and~(iv). As in Part~2, checking~(iii) basically relies, thanks to Lemma~\ref{liptriangle}, on a comparison between the lengths of the corresponding sides and between the corresponding angles. The argument will be very similar to that already used in Part~II, but for the sake of clarity we are going to underline all the changes in the proof.\par
\begin{figure}[htbp]
\begin{center}
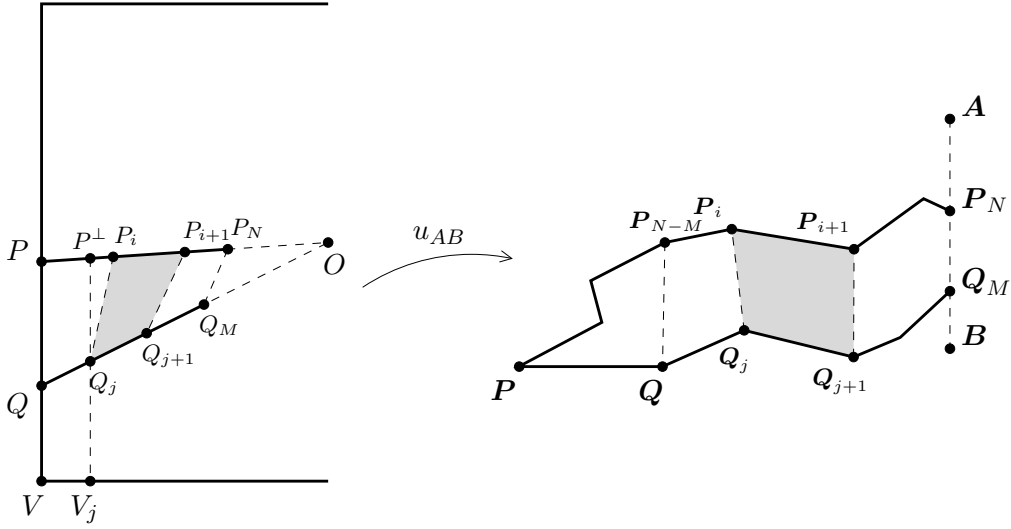\vspace{-10pt}
\caption{The situation in Part~3.}\label{Fig:step74}
\end{center}
\end{figure}
First of all, the argument leading to~(\ref{secondside}) can be \emph{verbatim} repeated for both the segments $P_iP_{i+1}$ and $Q_jQ_{j+1}$, leading to
\begin{align}\label{nsides1}
\frac{1} {1200L^2} \leq \frac{\ell\big(P_iP_{i+1}\big)}{\ell\big(\u{P}_i\u{P}_{i+1}\big)} \leq  \frac{\sqrt{2}} {20L} \,, &&
\frac{1} {1200L^2} \leq \frac{\ell\big(Q_jQ_{j+1}\big)}{\ell\big(\u{Q}_j\u{Q}_{j+1}\big)} \leq  \frac{\sqrt{2}} {20L}\,.
\end{align}
The argument that we used in Part~2 to bound the length of the segment $P_{i+1}Q$ works, with minor modifications, to estimate the lengths of $P_iQ_j$ and $P_{i+1}Q_{j+1}$. Let us do it in detail for $P_iQ_j$, being the case of $P_{i+1}Q_{j+1}$ exactly the same. First of all, assuming without loss of generality that $P$ and $Q$ lie on the left side of $\D$, and that $P$ is above $Q$, let us call $x_j\in (-1/2,-1/10)$ the first coordinate of $Q_j$, set $V_j\equiv (x_j,-1/2)$, $V\equiv (-1/2,-1/2)$, and define $P_\perp$ the point of the segment $OP$ having first coordinate equal to $x_j$.\par
As in~(\ref{142}), then, we obtain
\begin{align}\label{142bis}
|t_{P,i} - t_{Q,j} |\leq 17 L \,\ell\big(PQ\big)\,, &&  | t_{P,i+1} - t_{Q,j+1} | \leq 17 L \,\ell\big(PQ\big)\,.
\end{align}
We claim that
\begin{equation}\label{10not2}
\frac {\sqrt{2}}{10} \, \ell\big(PQ\big) \leq \ell\big(P_iQ_j\big) \leq \frac{90}{7}\, L\, \ell\big(PQ\big)\,.
\end{equation}
--notice the presence of $\sqrt{2}/10$ in the left hand side, while there was $\sqrt{2}/2$ in the corresponding term in~(\ref{141}). To show the left inequality in~(\ref{10not2}) we start observing that, being $P_i$ in $OP$, one has
\[
\ell\big(P_iQ_j\big) \geq \ell\big(P_\perp Q_j\big) \sin\big(\angle O{P_\perp}{Q_j}\big) 
=\ell\big(P_\perp Q_j\big) \sin\big(\angle OPV\big) 
\geq \frac{\sqrt{2}}{2} \, \ell\big(P_\perp Q_j\big)\,.
\]
Moreover, the segment $P_\perp Q_j$ is parallel to $PQ$, then~(\ref{under45}) immediately gives $\ell\big(P_\perp Q_j\big)\geq \ell\big(PQ\big)/5$. Hence, we get $\ell\big(P_iQ_j\big) \geq \frac{\sqrt{2}}{10}\ell\big(PQ\big)$, that is the left inequality of~(\ref{10not2}).\par
Let us now pass to the right inequality. To do so we call again $\delta_x$ and $\delta_y$ the horizontal and vertical components of $P_iQ_j$, so that $\ell(P_iQ_j)=\sqrt{\delta_x^2+\delta_y^2}$. Notice that by construction
\[
|\delta_x|=\frac{|t_{P,i}-t_{Q,j}|}2\leq \frac{17}2\, L\, \ell\big(PQ\big)\,.
\]
Moreover,
\[
\frac \pi 4 \leq \angle{P_i}{P_\perp}{Q_j} = \angle OPV \leq \frac 34\, \pi\,,
\]
hence $|\delta_y| \leq \ell(P_\perp Q_j) + |\delta_x|\leq \ell(PQ)+|\delta_x|$. As a consequence, exactly as in~(\ref{copy}) we get, using~(\ref{142bis}),
\[\begin{split}
\ell\big(P_iQ_j\big) &= \sqrt{\delta_x^2 + \delta_y^2}
\leq \ell\big(PQ\big) \sqrt{\bigg(\frac{17}2\, L\bigg)^2+ \bigg(\frac{17}2\, L+1\bigg)^2 }
\leq \frac{90} 7\, L\, \ell\big(PQ\big)\,.
\end{split}\]
Thus, (\ref{10not2}) is proved. Since~(iii) of Lemma~\ref{lemmastep4} gives
\[
\frac{\ell\big(PQ\big)}{7L} \leq \ell\big(\u{P}_i\u{Q}_j\big) \leq \lsegm{PQ}\leq L\,\ell\big(PQ\big)\,,
\]
from~(\ref{10not2}) we immediately obtain
\begin{equation}\label{nsides2}
\frac {\sqrt{2}}{10 L}  \leq \frac{\ell\big(P_iQ_j\big)}{\ell\big(\u{P}_i\u{Q}_j\big)} \leq 90 L^2\,.
\end{equation}
The same argument, exchanging $i$ and $j$ with $i+1$ and $j+1$ respectively, gives also
\begin{equation}\label{nsides3}
\frac {\sqrt{2}}{10 L}  \leq \frac{\ell\big(P_{i+1}Q_{j+1}\big)}{\ell\big(\u{P}_{i+1}\u{Q}_{j+1}\big)} \leq 90 L^2\,.
\end{equation}
We now have to consider the angles $\angle{P_i}{P_{i+1}}{Q_{j+1}}$, $\angle{Q_{j+1}}{Q_j}{P_i}$ and their correspondent ones in $\Delta$.
\begin{figure}[htbp]
\begin{center}
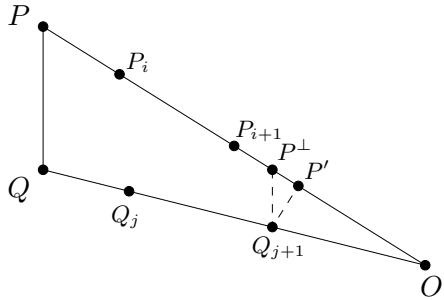\vspace{-10pt}
\caption{Position of the points $P_i$, $P_{i+1}$, $Q_j$, $Q_{j+1}$, $P^\perp$ and $P'$.}\label{Fig:step75}
\end{center}
\end{figure}
By Lemma~\ref{lemmastep4} (i), we already know that
\begin{align}\label{finalangles}
15^\circ \leq \angle{\u{P}_i}{\u{P}_{i+1}}{\u{Q}_{j+1}} \leq 165^\circ\,, &&
\sin\Big(\angle{\u{Q}_{j+1}}{\u{Q}_j}{\u{P}_i}\Big) \geq \frac{1}{6L^2}\,.
\end{align}
As in Figure~\ref{Fig:step75}, let us then call $P'$ the orthogonal projection of $Q_{j+1}$ on the segment $OP$, and $P^\perp$ the point of the segment $OP$ with the same first coordinate as $Q_{j+1}$. Assume for a moment that, as in the figure, $P'$ does not belong to $PP_{i+1}$. By~(\ref{142bis}) and by~(\ref{under45}) we have
\begin{equation}\label{somest}\begin{aligned}
\ell\big(P_{i+1}P^\perp\big)&=\frac{|t_{P,i+1}-t_{Q,j+1}|}{2\sin\big(\angle OPQ\big)}\leq\frac{\sqrt{2}}{2}\,17L\,\ell\big(PQ\big)\,,&
\ell\big( P^\perp Q_{j+1}\big) &\geq \frac{\ell\big(PQ\big)} 5\,, \\
\ell\big(P^\perp P'\big)&= \ell\big( P^\perp Q_{j+1}\big)\cos\big(\angle OPQ\big)\,, &
\ell\big(Q_{j+1}P'\big)&=\ell\big( P^\perp Q_{j+1}\big)\sin\big(\angle OPQ\big)\,.\hspace{-20pt}
\end{aligned}\end{equation}
Therefore, we can evaluate
\[\begin{split}
\tan\big(\angle{P'}{P_{i+1}}{Q_{j+1}}\big) &= \frac{\ell\big(Q_{j+1}P'\big)}{\ell\big(P_{i+1}P'\big)}
\geq \frac{\ell\big(Q_{j+1}P'\big)}{\ell\big(P^\perp P'\big)+\ell\big(P_{i+1}P^\perp\big)}
\geq\frac{\frac{\sqrt{2}}2\ell\big(P^\perp Q_{j+1}\big)}{\frac{\sqrt{2}}2\ell\big(P^\perp Q_{j+1}\big)+\frac{\sqrt{2}}{2}17L\,\ell\big(PQ\big)}\\
&= \frac{\ell\big(P^\perp Q_{j+1}\big)}{\ell\big(P^\perp Q_{j+1}\big)+ 17 L\,\ell\big(PQ\big)}
\geq \frac 1{86L}\,,
\end{split}\]
which immediately gives
\begin{equation}\label{ii+1j+1}
\angle{P_i}{P_{i+1}}{Q_{j+1}} = \pi - \angle{P'}{P_{i+1}}{Q_{j+1}} \leq \pi - \arctan\bigg(\frac{1}{86L}\bigg)\,.
\end{equation}
Notice that, if $P'$ belongs to $PP_{i+1}$, then $\angle{P_i}{P_{i+1}}{Q_{j+1}}\leq \pi/2$, so~(\ref{ii+1j+1}) holds \emph{a fortiori} true.\par
We claim that one also has
\begin{equation}\label{ii+1j+1bel}
\angle{P_i}{P_{i+1}}{Q_{j+1}} \geq \frac{1}{87 L}\,.
\end{equation}
To show this, we are going to argue in a very similar way to what already done in Part~2. In fact, if $t_{P,i+1}\leq t_{Q,j+1}$ then~(\ref{ii+1j+1bel}) trivially holds true. Assuming, on the contrary, that $t_{P,i+1}>t_{Q,j+1}$, we call for brevity $\alpha:=\angle{P_i}{P_{i+1}}{Q_{j+1}}$ and $\theta:=\angle{O}{P^\perp}{Q_{j+1}}-\frac \pi 2 \in \big[ -\frac \pi 4, \frac \pi 4\big)$, and we notice that an immediate trigonometric argument gives
\begin{equation}\label{1nstepbis}
\ell\big( P^\perp Q_{j+1} \big) = \frac{t_{P,i+1}-t_{Q,j+1}}{2} \Big( \tan(\theta+\alpha) -\tan \theta \Big)\,.
\end{equation}
We can assume that
\[
\theta+\alpha\leq \frac \pi 4 + \frac 1{87}\,,
\]
since otherwise~(\ref{ii+1j+1bel}) is already established. Hence, recalling~(\ref{somest}), (\ref{1nstepbis}), the fact that $\theta\geq -\pi/4$, (\ref{142bis}) and the Lipschitz property of $u$ we get
\[\begin{split}
\ell\big( PQ \big) & \leq 5 \,\ell\big(P^\perp Q_{j+1}\big)
= \frac 52 \,\big(t_{P,i+1}-t_{Q,j+1}\big) \Big( \tan(\theta+\alpha) -\tan \theta \Big)
\leq \frac {85 L \,\ell\big(PQ\big)}2 \, \frac{\alpha}{\cos^2 \Big( \frac \pi 4 + \frac 1 {87}\Big)}\,,
\end{split}\]
which implies
\[
\alpha \geq \frac{2\cos^2 \Big( \frac \pi 4 + \frac 1 {87}\Big)}{85L}\geq \frac{1}{87L}\,.
\]
Thus, (\ref{ii+1j+1bel}) is now established. If we repeat exactly the same argument that we used to obtain~(\ref{ii+1j+1}) and~(\ref{ii+1j+1bel}) in the symmetric way, that is, substituting $P_i$, $P_{i+1}$ and $Q_{j+1}$ with $Q_{j+1}$, $Q_j$ and $P_i$ respectively, then we get
\begin{align}\label{j+1ji}
\angle{Q_{j+1}}{Q_j}{P_i} \geq \arctan\bigg(\frac{1}{86L}\bigg) \,, &&
\angle{Q_{j+1}}{Q_j}{P_i} \leq \pi -  \frac{1}{87 L}\,.
\end{align}
We are finally in position to check the validity of~(iii) by making use of~(\ref{lipschtriangle}) of Lemma~\ref{liptriangle}. Indeed, let us call $\phi$ (resp. $\tilde\phi$) the affine map which send $P_iP_{i+1}Q_{j+1}$ on $\u{P}_i\u{P}_{i+1}\u{Q}_{j+1}$ (resp. $Q_{j+1}Q_jP_i$ on $\u{Q}_{j+1}\u{Q}_j\u{P}_i$). According with the notation of Lemma~\ref{liptriangle}, let us write
\begin{align*}
a&=\ell\big(P_{i+1} Q_{j+1}\big) \,, & 
b&=\ell\big(P_iP_{i+1}\big)\,, & 
\alpha&=\angle {P_i}{P_{i+1}}{Q_{j+1}}\,, \\
a'&=\ell\big(\u{P}_{i+1}\u{Q}_{j+1}\big)\,, &
b'&=\ell\big(\u{P}_i\u{P}_{i+1}\big)\,,& 
\alpha'&=\angle {\u{P}_i}{\u{P}_{i+1}}{\u{Q}_{j+1}} \,,\\
\tilde a&=\ell\big(P_i Q_j\big) \,, & 
\tilde b&=\ell\big(Q_jQ_{j+1}\big)\,, & 
\tilde \alpha&=\angle {Q_{j+1}}{Q_j}{P_i}\,, \\
\tilde a'&=\ell\big(\u{P}_i\u{Q}_j\big)\,, &
\tilde b'&=\ell\big(\u{Q}_j\u{Q}_{j+1}\big)\,,& 
\tilde \alpha'&=\angle {\u{Q}_{j+1}}{\u{Q}_j}{\u{P}_i} \,.
\end{align*}
The estimates~(\ref{nsides1}), (\ref{nsides2}) and~(\ref{nsides3}) for the sides, and~(\ref{finalangles}), (\ref{ii+1j+1}), (\ref{ii+1j+1bel}) and~(\ref{j+1ji}) for the angles, give us
\begin{align}
\frac{\sqrt{2}}{10L}\leq \frac a{a'} \leq 90 L^2\,, && \frac{1}{1200 L^2}\leq \frac b{b'}\leq \frac{\sqrt{2}}{20L}\,, && \sin\alpha' \geq \frac 14\,, && \sin\alpha \geq \frac 1{88L}\,, \label{ratiosIII}\\
\frac{\sqrt{2}}{10L}\leq \frac {\tilde a}{\tilde a'} \leq 90 L^2\,, && \frac{1}{1200 L^2}\leq \frac {\tilde b}{\tilde b'}\leq \frac{\sqrt{2}}{20L}\,, && \sin \tilde \alpha' \geq \frac 1{6L^2}\,, && \sin\tilde \alpha \geq \frac 1{88 L}\,,\label{ratiosIIIb}
\end{align}
where the estimates for $\alpha$ and $\tilde\alpha$ can be obtained in the very same way as~(\ref{verysame}). As in Part~2, then, we can apply~(\ref{lipschtriangle}) together with~(\ref{ratiosIII}) and~(\ref{ratiosIIIb}) to obtain
\[\begin{split}
{\rm Lip} (\phi)
&\leq\frac {a'}a + \frac{2b'}{b\sin\alpha} + \frac{a'}{a\sin\alpha}
\leq 5\sqrt 2\, L + 211200L^3  + 440\sqrt 2\, L^2\,,\\
{\rm Lip} (\phi^{-1})
&\leq\frac a{a'} + \frac{2b}{b'\sin\alpha'} +\frac{a}{a'\sin\alpha'}
\leq 90 L^2 + \frac{2\sqrt 2}{5 L}  + 360 \, L^2\,,\\
{\rm Lip} (\tilde\phi)
&\leq\frac {\tilde a'}{\tilde a} + \frac{2\tilde b'}{\tilde b\sin\tilde\alpha} + \frac{\tilde a'}{\tilde a\sin\tilde \alpha}
\leq 5\sqrt 2\, L + 211200L^3  + 440\sqrt 2\, L^2\,,\\
{\rm Lip} (\tilde\phi^{-1})
&\leq\frac {\tilde a}{\tilde a'} + \frac{2\tilde b}{\tilde b'\sin\tilde \alpha'} +\frac{\tilde a}{\tilde a'\sin\tilde \alpha'}
\leq 90 L^2 + \frac{3\sqrt 2 L}5  + 540\, L^4 \,.
\end{split}\]
Thus, we have checked the validity of~(iii).\par
Concerning~(iv), we have to show that
\begin{align}\label{nitsch}
\angle{P_N}{Q_M}{O}\geq \frac{1}{87L}\,, &&
\angle{Q_M}{P_N}{O}\geq \frac{1}{87L}\,.
\end{align}
In fact, applying~(\ref{ii+1j+1}) with $i=N-1$ and then $j=M-1$, we have that
\[
\angle{Q_M}{P_N}{O} = \pi - \angle{P_{N-1}}{P_N}{Q_M}  \geq \arctan\bigg(\frac{1}{86L}\bigg) \geq \frac{1}{87L}\,,
\]
and the same argument, exchanging the roles of $P_N$ and $Q_M$, ensures that also $\angle{P_N}{Q_M}{O}\geq 1/(87L)$, thus proving the validity of~(\ref{nitsch}). Property~(iv) is then established and the proof is concluded.
\end{proof}

\bigstep{VIII}{Definition of the piecewise affine extension $v$}

We finally come to the explicit definition of the piecewise affine map $v$. It is important to recall now Lemma~\ref{lemma:center} of Step~I. It provides us with a central ball $\widehat{\u\B}\subseteq\Delta$ which is such that the intersection of its boundary with $\partial\Delta$ consists of $N$ points $\u{A}_1,\, \u{A}_2,\, \dots \, ,\, \u{A}_N$, with $N\geq 2$. Moreover, for each $1\leq i\leq N$ one has that the path $\arc{A_iA_{i+1}}$ does not contain other points $A_j$ with $j\neq i, \, i+1$. Or, in other words, that for each $1\leq i \leq N$ the anticlockwise path connecting $A_i$ and $A_{i+1}$ on $\partial\D$ has length at most $2$ (keep in mind Remark~\ref{rem-check}). Notice that this implies, in the case $N=2$, that the points $A_1$ and $
A_2$ are opposite points of $\partial\D$. The set $\Delta$ is then subdivided in $N$ primary sectors $\DoubleS(\u{A}_i\u{A}_{i+1})$, plus the remaining polygon $\Pi$ (see e.g. Figure~\ref{Fig:23}, where $\Pi$ is a coloured quadrilateral).\par\smallskip

Moreover, thanks to Step~VII, we have $N$ disjoint polygonal subsets $\D_i$ as in the Figure, and $N$ extensions $u_i:\D_i\to \DoubleS(\u{A}_i\u{A}_{i+1})$. It is then easy to guess a possible definition of $v$, that is setting $v\equiv u_i$ on each $\D_i$ and then sending in the obvious piecewise affine way the set $\D\setminus \cup_i \D_i$ (dark in the figure) into the polygon $\Pi$, defining $u(O)$ as the center of $\widehat{\u \B}$. Unfortunately, this strategy does not always work. For instance, if $N=2$, then $\Pi$ is a degenerate empty polygon, thus it cannot be the bi-Lipschitz image of the non-empty region $\D\setminus \cup_i \D_i$. Also for $N\geq 3$, it may happen that the polygon $\Pi$ does not contain the center of $\widehat{\u\B}$, which is instead inside some sector $\DoubleS(\u{A}_i\u{A}_{i+1})$. In that case, obviously, the center of $\widehat{\u\B}$ can not be the point $u(O)$. Having these possibilities in mind, we are now ready to give the proof of the first part of Theorem~\mref{main}, that is, the existence of the piecewise affine extension $v$ of $u$.
\begin{figure}[htbp]
\begin{center}
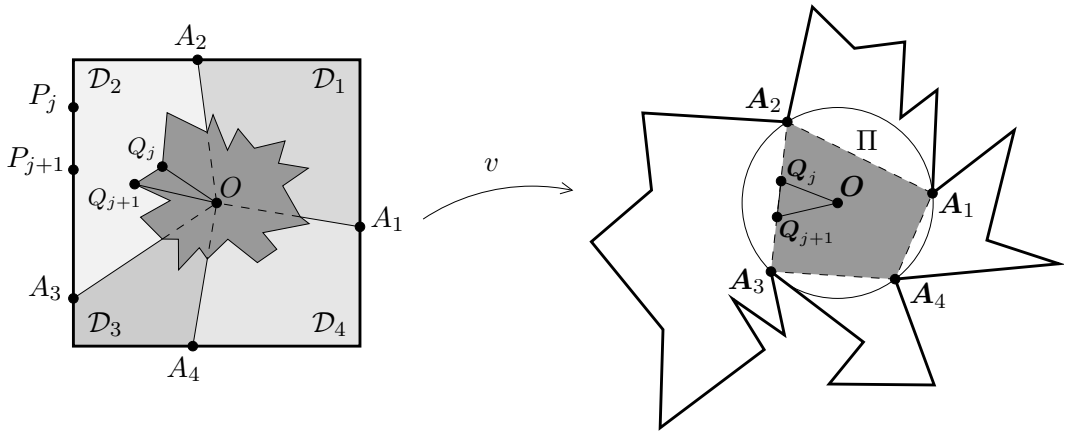\vspace{-10pt}
\caption{The sets $\D_i$ in $\D$ and the set $\Pi$ in $\Delta$.}\label{Fig:23}
\end{center}
\end{figure}
\proofof{Theorem~\mref{main} (piecewise affine extension)}
We need to consider three possible situations. To distinguish between them, let us start with a definition. For any $1\leq i \leq N$, we call $d_i$ the signed distance between the segment $\u{A}_i\u{A}_{i+1}$ and the center of $\widehat{\u{\B}}$, where the sign is positive if the center does not belong to $\DoubleS(\u{A}_i\u{A}_{i+1})$, and negative otherwise --for instance, in the situation of Figure~\ref{Fig:23} all the distances $d_i$ are positive. Let us also call $r$ the radius of $\widehat{\u{\B}}$, and observe that
\begin{equation}\label{estimr}
\frac{2}{3L} \leq r \leq \frac{2L}{\pi}\,.
\end{equation}
The first inequality has been already pointed out in Remark~\ref{rem-check}. Concerning the second one, it immediately follows by observing that the perimeter of $\Delta$ is at least $2\pi r$ by geometric reasons, and on the other hand it is less than $4L$ since it is the $L-$Lipschitz image of the square $\D$ which has perimeter $4$. We can then give our proof in the different cases.

\case{A}{For each $1\leq i \leq N$, one has $d_i \geq r/4$.}
This is the simplest of the three cases, and the situation is already shown in Figure~\ref{Fig:23}. We start by calling $\u{O}$ the center of $\widehat{\u\B}$. Then, for all $1\leq i \leq N$, let us define $v\equiv u_i$ on $\D_i$. We have now to send $\D\setminus \cup_i \D_i$ into $\Pi$. In order to do so, consider all the vertices $P_j$ of $\partial\D$. For each vertex $P_j$, which belongs to some set $\D_i$ for a suitable $i=i(j)$, there exists a point $Q_j$, which is the last point of the segment $P_jO$ which belongs to $\partial\D_i$. In fact, the segment $P_jO$ intersects $\partial \D_i$ only at $P_j$ and at $Q_j$, and the two points are the same if and only if $P_j\equiv A_i$ or $P_j\equiv A_{i+1}$. By the construction of Step~VII, we know that $v(Q_j)=(\u{P}_j)_{N(P_j)}$, and we will write for brevity $\u{Q}_j := (\u{P}_j)_{N(P_j)}$. Notice now that $\D\setminus\cup_i \D_i$ is the union of the triangles $Q_jQ_{j+1}O$, and on the other hand $\Pi$ is the union of the triangles $\u{Q}_j\u{Q}_{j+1}\u{O}$. We then conclude our definition of $v$ by imposing that $v$ sends in the affine way each triangle $Q_jQ_{j+1}O$ onto the triangle $\u{Q}_j\u{Q}_{j+1}\u{O}$. Hence, it is clear that $v$ is a piecewise affine homeomorphism between $\D$ and $\Delta$, which extends the original function $u$. Thus, to finish the proof we only have to check that $v$ is bi-Lipschitz with the right constant. Since this is already ensured by Lemma~\ref{lemmastep7} on each primary sector, it remains now only to consider a single triangle $Q_jQ_{j+1}O$. Using again Lemma~\ref{liptriangle} from Step~VII to estimate the bi-Lipschitz constant of the affine map on the triangle, we have to give upper and lower bounds for the quantities
\begin{align*}
a&=\ell\big(Q_jQ_{j+1}\big)\,, & 
b&=\ell\big(Q_j O\big) \,, & 
\alpha&=\angle O{Q_j}{Q_{j+1}}\,, \\
a'&=\ell\big(\u{Q}_j\u{Q}_{j+1}\big)\,,& 
b'&=\ell\big(\u{Q}_j\u{O}\big)\,, &
\alpha'&=\angle {\u{O}}{\u{Q}_j}{\u{Q}_{j+1}} \,.
\end{align*}
Let us then collect all the needed estimates: first of all, notice that the ratio $a/a'$ has already been evaluated in Lemma~\ref{lemmastep7}, either in Part~2 or in Part~3. Thus, recalling~(\ref{ratiosII}) and~(\ref{ratiosIII}), we already know that
\begin{equation}\label{8side11}
\frac{\sqrt{2}}{10L}\leq \frac a{a'} \leq 90 L^2\,.
\end{equation}
Concerning the ratio $b/b'$, notice that by geometric reasons and recalling~(\ref{under45}), we have
\begin{equation}\label{stimab}
\frac 1 {10} \leq b \leq \frac {\sqrt{2}}2\,,
\end{equation}
while by~(\ref{estimr}) and the assumption of this case
\begin{equation}\label{stimab'}
\frac 1{6L} \leq \frac r4 \leq b' \leq r \leq \frac{2L}\pi\,.
\end{equation}
Thus,
\begin{equation}\label{8side12}
\frac \pi{20L} \leq \frac b{b'} \leq 3\sqrt 2 L\,.
\end{equation}
Let us finally consider the angles $\alpha$ and $\alpha'$. Concerning $\alpha$, property~(iv) of Lemma~\ref{lemmastep7} tells us that
\begin{equation}\label{8angle10}
\frac 1{87L} \leq \alpha \leq \pi - \frac 1{87L}\,.
\end{equation}
On the other hand, by the assumption of this case we clearly have
\[
\arcsin  \frac 14 \leq \alpha' \leq \pi - \arcsin\frac 14\,,
\]
and then
\begin{align}\label{8angle1}
\frac{1}{\sin\alpha} \leq 88L\,, &&
\frac{1}{\sin\alpha'} \leq 4 \,.
\end{align}
We can then apply~(\ref{lipschtriangle}) making use of~(\ref{8side11}), (\ref{8side12}) and~(\ref{8angle1}) to get
\[\begin{split}
{\rm Lip} (\phi)
&\leq\frac {a'}a + \frac{2b'}{b\sin\alpha} + \frac{a'}{a\sin\alpha}
\leq 5\sqrt 2 L + \frac{3520}\pi\,L^2 + 440\sqrt 2 L^2\,,\\
{\rm Lip} (\phi^{-1})
&\leq\frac a{a'} + \frac{2b}{b'\sin\alpha'} +\frac{a}{a'\sin\alpha'}
\leq 90 L^2 + 24\sqrt 2 L+360L^2\,,
\end{split}\]
thus the claim of the theorem is obtained in this first case.

\case{B}{There exists some $1\leq i \leq N$ such that $-r/2 \leq d_i < r/4$.}
Also in this case, we set $u(O)=\u{O}$ to be the center of $\widehat{\u{\B}}$. Let us write now $\D= \cup_i \A_i$, where each $\A_i$ is the subset of $\D$ whose boundary is $A_iO \cup A_{i+1}O \cup \arc{A_iA_{i+1}}$. Notice that for each $i$, one has $\D_i\subseteq \A_i$, and in particular we set $\I_i = \A_i \setminus \D_i$, the ``internal part'' of $\A_i$. Our definition of $v$ will be done in such a way that, for each $1\leq i\leq N$, $v(\A_i)$ will be the union of the sector $\DoubleS(\u{A}_i\u{A}_{i+1})$ and the triangle $\u{A}_i\u{A}_{i+1}\u{O}$. Observe that, in the Case~A, we had defined $v$ so that for each $i$ one had $v(\D_i) = \DoubleS(\u{A}_i\u{A}_{i+1})$ and $v(\I_i)=\u{A}_i\u{A}_{i+1}\u{O}$.\par
Let us fix a given $1\leq i\leq N$, and notice that either $d_i \geq r/4$, or $-r/2 \leq d_i <r/4$. In fact, since we assume the existence of some $i$ for which $-r/2 \leq d_i < r/4$, then it is not possible that there exists some other $i$ with $d_i <-r/2$.\par
If $d_i\geq r/4$, then we define $v$ exactly as in Case~A, that is, we set $v\equiv u_i$ on $\D_i$, and for any two consecutive vertices $P_j,\, P_{j+1} \in \arc{A_iA_{i+1}}$ we let $v$ be the affine function transporting the triangle $Q_jQ_{j+1}O$ of $\D$ onto the triangle $\u{Q}_j\u{Q}_{j+1}\u{O}$ of $\Delta$, where $\u{Q}_k=\u{P}_{N(P_k)}$. In this case, $v$ is bi-Lipschitz on $\A_i$ with constant at most $5\sqrt 2 L + 3520 L^2/\pi  + 440\sqrt 2 L^2$, as we already showed in Case~A.\par
\begin{figure}[htbp]
\begin{center}
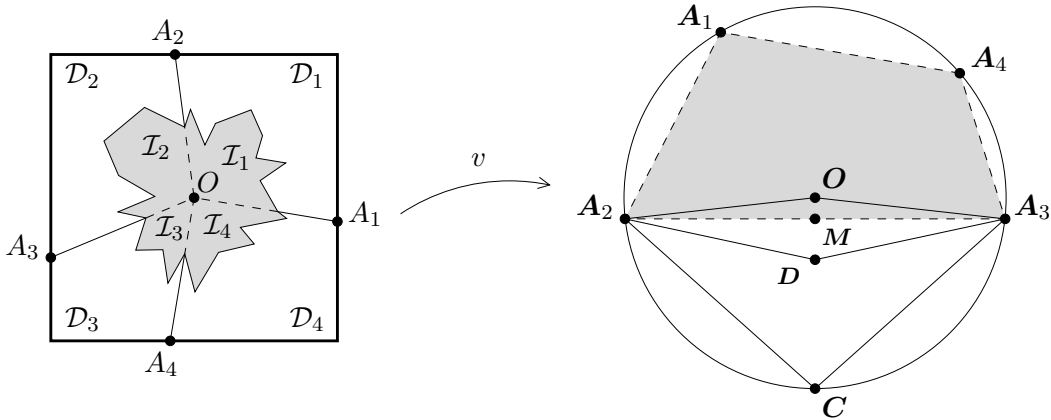\vspace{-10pt}
\caption{The situation for Case~B, with the sets $\A_i$ and the points $\u{M}$, $\u{D}$ and $\u{C}$.}\label{Fig:24}
\end{center}
\end{figure}
Consider then the case of an index $i$ such that $-r/4 \leq d_i\leq r/4$, as it happens for $i=2$ in Figure~\ref{Fig:24} (where $d_2$ is positive but smaller than $r/4$). As in the figure, let us call $\u{C}\in \partial \widehat{\u{\B}}$ the point belonging to the axis of the segment $\u{A}_i\u{A}_{i+1}$ and to the sector $\DoubleS(\u{A}_i\u{A}_{i+1})$, and let also $\u{D}\in \u{OC}$ be the point such that $\lsegm{OD}=r/4$. We now introduce a bi-Lipschitz and piecewise affine function $\Phi:\u{A}_i\u{A}_{i+1}\u{C} \to \u{A}_i\u{DA}_{i+1}\u{C}$. If we call $\u{M}$ the mid-point of $\u{A}_i\u{A}_{i+1}$, the function $\Phi$ is simply given by the affine map between the triangle $\u{A}_i\u{MC}$ and $\u{A}_i\u{DC}$, and by the affine map between $\u{A}_{i+1}\u{MC}$ and $\u{A}_{i+1}\u{DC}$. The fact that $\Phi$ is piecewise affine is clear, being $\Phi$ defined gluing two affine maps. Moreover, by the fact that $-r/2 \leq d_i< r/4$, $\Phi$ is $2-$Lipschitz and $\Phi^{-1}$ is $3-$Lipschitz. We will extend $\Phi: \DoubleS(\u{A}_i\u{A}_{i+1})\to \DoubleS(\u{A}_i\u{A}_{i+1})$, whitout need of changing the name, as the identity out of the triangle $\u{A}_i\u{A}_{i+1}\u{C}$. Of course also the extended $\Phi$ is $2-$Lipschitz and its inverse is $3-$Lipschitz.\par
We are now ready to define $v$ in $\A_i$. First of all, we set $v \equiv \Phi \circ u_i$ on $\D_i$. Thanks to Lemma~\ref{lemmastep7} and the properties of Lipschitz functions, we have that $v$ is piecewise affine and bi-Lipschitz with constant $3\cdot 212000L^4=636000L^4$ on its image, which is $\DoubleS(\u{A}_i\u{A}_{i+1})\setminus \u{A}_i\u{A}_{i+1}\u{D}$. To conclude, we need to send $\I_i$ onto the quadrilater $\u{A}_i\u{OA}_{i+1}\u{D}$. To do so, consider all the vertices $P_j\in\arc{A_iA_j}$, and define $Q_j\in \partial\D_i$ as in Case~A. This time, we will not set $\u{Q}_j=u_i(Q_j)$: instead, $\u{Q}_j$ will be defined as $\u{Q}_j := \Phi\big(u_i(Q_j)\big)$, so that $v(Q_j)=\u{Q}_j$ as usual. Notice that, again, $\I_i$ is the union of the triangles $Q_jQ_{j+1}O$, while the quadrilateral $\u{A}_i\u{OA}_{i+1}\u{D}$ is the union of the triangles $\u{Q}_j\u{Q}_{j+1}\u{O}$ (up to the possible addition of a new vertex corresponding to $\u{D}$). The map $v$ on $\I_i$ will be then the map which sends each triangle $Q_jQ_{j+1}O$ onto $\u{Q}_j\u{Q}_{j+1}\u{O}$ in the affine way. Clearly the map $v$ is then a piecewise affine homeomorphism, so that again we only have to check its bi-Lipschitz constant (Figure~\ref{Fig:25} may help the reader to follow the construction). As usual, we will apply~(\ref{lipschtriangle}) of Lemma~\ref{liptriangle}, so we set the quantities
\begin{align*}
a&=\ell\big(Q_jQ_{j+1}\big)\,, & 
b&=\ell\big(Q_j O\big) \,, & 
\alpha&=\angle O{Q_j}{Q_{j+1}}\,, \\
a'&=\ell\big(\u{Q}_j\u{Q}_{j+1}\big)\,,& 
b'&=\ell\big(\u{Q}_j\u{O}\big)\,, &
\alpha'&=\angle {\u{O}}{\u{Q}_j}{\u{Q}_{j+1}} \,.
\end{align*}
\begin{figure}[htbp]
\begin{center}
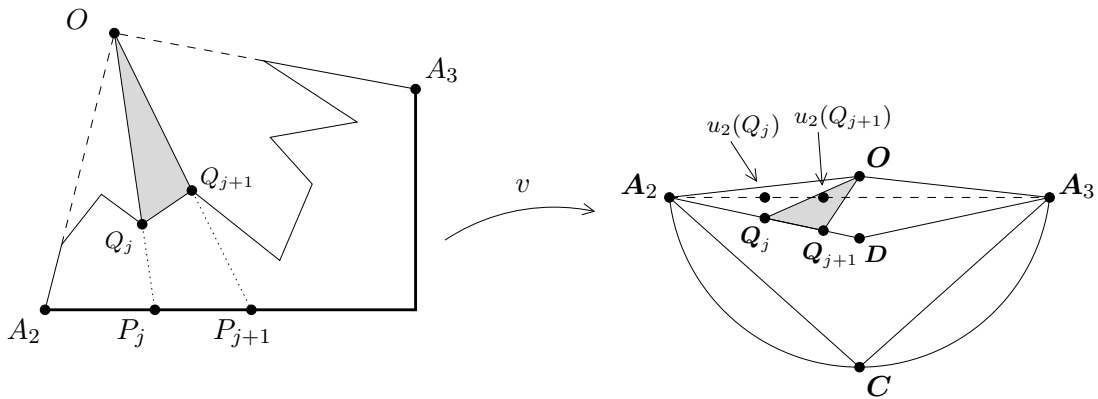\vspace{-10pt}
\caption{A zoom for Case~B, with $Q_j,\,Q_{j+1},\, u_2(Q_j),\, u_2(Q_{j+1}),\, \u{Q}_j$ and $\u{Q}_{j+1}$.}\label{Fig:25}
\end{center}
\end{figure}
Recall that, studying Case~A, we have already found in~(\ref{8side11}) that for each vertex $P_j\in\arc{A_iA_{i+1}}$ one has
\begin{equation}\label{8side11bis}
\frac{\sqrt{2}}{10L}\leq \frac{\ell\big(Q_jQ_{j+1}\big)}{\ell\Big(u_i(Q_j)u_i(Q_{j+1})\Big)} \leq 90 L^2\,.
\end{equation}
Notice also that now we have $\ell\big(Q_jQ_{j+1}\big)=a$, exactly as in Case~A, but it is no more true that $\ell\Big(u_i(Q_j)u_i(Q_{j+1})\Big)=a'$. However, since $\Phi$ is $2-$Lipschitz and $\Phi^{-1}$ is $3-$Lipschitz, we have
\[\begin{split}
a'&=\ell\big(\u{Q}_j\u{Q}_{j+1}\big)=\ell\Big(\Phi\big(u_i(Q_j)\big)\Phi\big(u_i(Q_{j+1})\big)\Big)\leq 2\,\ell\Big(u_i(Q_j)u_i(Q_{j+1})\Big)\,,\\
a' &=\ell\big(\u{Q}_j\u{Q}_{j+1}\big)=\ell\Big(\Phi\big(u_i(Q_j)\big)\Phi\big(u_i(Q_{j+1})\big)\Big) \geq\frac{\ell\Big(u_i(Q_j)u_i(Q_{j+1})\Big)}3\,,
\end{split}\]
which by~(\ref{8side11bis}) ensures
\begin{equation}\label{8side21}
\frac{\sqrt{2}}{20L}\leq \frac a{a'} \leq 270 L^2\,.
\end{equation}
To bound the ratio $b/b'$, we have to estimate both $b$ and $b'$. Concerning $b$, we already know by~(\ref{stimab}) that
\[
\frac 1 {10} \leq b \leq \frac {\sqrt{2}}2\,.
\]
On the other hand, let us study $b'$. The estimate from above, exactly as in~(\ref{stimab'}), is simply obtained by~(\ref{estimr}) as
\[
b' \leq r \leq \frac{2L}\pi\,.
\]
Instead, to get the estimate from below, it is enough to recall that $\u{Q}_j$ belongs to the segment $\u{A}_i\u{D}$ (or $\u{A}_{i+1}\u{D}$). Thus, being $d_i\leq r/4$, an immediate geometric argument and again~(\ref{estimr}) give
\[
b' \geq \frac{1} {2\sqrt{7}}\, r \geq \frac{1} {3\sqrt{7}L}\,.
\]
Collecting the inequalities that we just found, we get
\begin{equation}\label{8side22}
\frac \pi{20L} \leq \frac b{b'} \leq \frac 32\,\sqrt{14} L\,.
\end{equation}
Concerning the angles, (\ref{8angle10}) already tells us that
\[
\frac 1{87L} \leq \alpha \leq \pi - \frac 1{87L}\,.
\]
Moreover, an immediate geometric argument ensures that $\sin\alpha'$ is minimal if $\alpha' = \angle{\u{O}}{\u{A}_i}{\u{D}}$, and in turn this last angle depends only on $d_i$ and it is minimal when $d_i = -r/2$. A simple calculation ensures that, in this extremal case, one has
\[
\alpha' = \arctan \frac{1}{\sqrt 3 /2} - \arctan \frac{1/2}{\sqrt 3 / 2} > 15^\circ\,,
\]
thus we have
\begin{align}\label{8angle2}
\frac{1}{\sin\alpha} \leq 88L\,, &&
\frac{1}{\sin\alpha'} \leq 4 \,.
\end{align}
Therefore, by applying~(\ref{lipschtriangle}) having~(\ref{8side21}), (\ref{8side22}) and~(\ref{8angle2}) at hand, we get
\[\begin{split}
{\rm Lip} (\phi)
&\leq\frac {a'}a + \frac{2b'}{b\sin\alpha} + \frac{a'}{a\sin\alpha}
\leq 10\sqrt 2 L + \frac{3520}\pi\,L^2 + 880\sqrt 2 L^2\,,\\
{\rm Lip} (\phi^{-1})
&\leq\frac a{a'} + \frac{2b}{b'\sin\alpha'} +\frac{a}{a'\sin\alpha'}
\leq 270L^2 + 12\sqrt {14} L+1080L^2\,.
\end{split}\]

\case{C}{There exists some $1\leq i \leq N$ such that $d_i < -r/2$.}

In this last case, notice that the index $i$ such that $d_i < -r/2$ is necessarily unique, since if $d_i<-r/2$ then for all $j\neq i$ one has $d_j> r/2$. For simplicity of notation, let us assume that the index is $i=1$. In this case, differently from the preceding ones, we will \emph{not} set $\u{O}$ to be the center of $\widehat{\u{\B}}$. Instead, as in Figure~\ref{Fig:26}, let us call $\u{M}$ the midpoint of $\u{A}_1\u{A}_2$, $\u{C}\in\widehat{\u{\B}}$ the point such that the triangle $\u{A}_1\u{A}_2\u{C}$ is equilateral, and $\u{D}$ and $\u{O}$ the two points which divide the segment $\u{CM}$ in three equal parts. We will define the extension $v$ in such a way that $v(O)=\u{O}$.\par
Before starting, we need to underline a basic estimate, that is,
\begin{equation}\label{basicestimate}
\frac 4{3L} \leq \ell\big(\u{A}_1\u{A}_2\big) \leq \frac{2\sqrt 3}\pi \, L\,.
\end{equation}
The right estimate is an immediate consequence of the assumption $d_1< -r/2$ and of~(\ref{estimr}). Concerning the left estimate, recall that, as noticed in Remark~\ref{rem-check}, there must be two points $\u{A}_i\u{A}_j\in \partial\widehat{\u{\B}}$ such that $\ell\big(\u{A}_i\u{A}_j\big) \geq 4/3L$. Thus the left estimate follows simply by observing that the distance $\ell\big(\u{A}_i\u{A}_j\big)$ is maximal, under the assumption of this Case~C, for $i=1$ and $j=2$.\par
We can now start our construction. Exactly as in Case~B, call $\Phi:\DoubleS(\u{A}_1\u{A}_2)\to \DoubleS(\u{A}_1\u{A}_2)$ the piecewise affine function which equals the identity out of $\u{A}_1\u{A}_2\u{C}$, and which sends in the affine way the triangle $\u{A}_1\u{MC}$ (resp. $\u{A}_2\u{MC}$) onto the triangle $\u{A}_1\u{DC}$ (resp. $\u{A}_2\u{DC}$). Also in this case, one easily finds that $\Phi$ is $2-$Lipschitz, while $\Phi^{-1}$ is $3-$Lipschitz. We are now ready to define the function $v$. As in Case~B, for any $i$ our definition will be so that $v(\A_i) = \DoubleS(\u{A}_i\u{A}_{i+1})\cup \u{A}_i\u{A}_{i+1}\u{O}$.\par
\begin{figure}[htbp]
\begin{center}
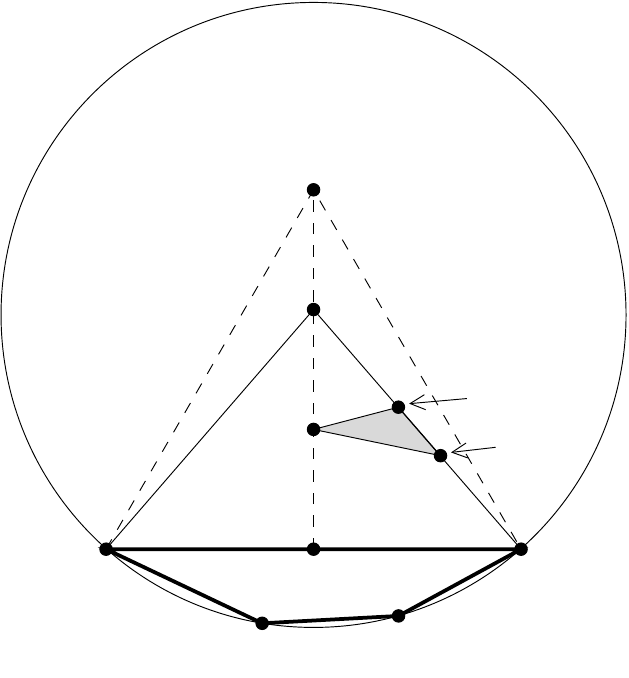\vspace{-10pt}
\caption{Situation in case~C, with $\u{A}_1,\,\u{A}_2,\, \u{C},\, \u{D},\, \u{M}$ and $\u{O}$.}\label{Fig:26}
\end{center}
\end{figure}
Let us start with $i=1$. First of all, we define $v:\D_1\to\Delta$ as $v= \Phi \circ u_1$, which is, exactly as in Case~B, a $636000L^4-$biLipschitz piecewise affine homeomorphism between $\D_1$ and $\DoubleS(\u{A}_1\u{A}_2)\setminus \u{A}_1\u{A}_2\u{D}$. Moreover, defining $Q_j$ and $\u{Q}_j$ as in Case~B, the internal part $\I_1$ is the union of the triangles $Q_jQ_{j+1}O$, while $\u{A}_1\u{OA}_2\u{D}$ is the union of the triangles $\u{Q}_j\u{Q}_{j+1}\u{O}$ (again, possibly adding a vertex corresponding to $\u{D}$). We will then define again $v:\I_1\to \D$ by sending in the affine way each triangle in its corresponding one, and since $v$ is again a piecewise affine homeomorphism by definition we have to check its bi-Lipschitz constant. To do so, we define as in Case~B the constants
\begin{align*}
a&=\ell\big(Q_jQ_{j+1}\big)\,, & 
b&=\ell\big(Q_j O\big) \,, & 
\alpha&=\angle O{Q_j}{Q_{j+1}}\,, \\
a'&=\ell\big(\u{Q}_j\u{Q}_{j+1}\big)\,,& 
b'&=\ell\big(\u{Q}_j\u{O}\big)\,, &
\alpha'&=\angle {\u{O}}{\u{Q}_j}{\u{Q}_{j+1}} \,.
\end{align*}
The very same arguments which lead to~(\ref{8side21}) and~(\ref{8angle10}) give again
\begin{align}\label{8side31}
\frac{\sqrt{2}}{20L}\leq \frac a{a'} \leq 270L^2\,, && \frac{1}{\sin\alpha} \leq 88L\,.
\end{align}
Since~(\ref{stimab}) is still true, to estimate $b/b'$ we again need to bound $b'$ from above and from below. By easy geometric arguments, since $\u{Q}_j$ belongs to $\u{A}_1\u{D}$ or to $\u{A}_2\u{D}$, we find
\[
\frac{\sqrt 7}{14} \,\ell\big(\u{A}_1\u{A}_2\big) \leq b' \leq \ell\big(\u{A}_1\u{O}\big) = \frac{\sqrt 3}3\, \ell\big(\u{A}_1\u{A}_2\big)\,.
\]
(recall that Figure~\ref{Fig:26} depicts the situation and the position of the points). Thanks to~(\ref{basicestimate}), then, we deduce
\[
\frac{2 \sqrt 7}{21L} \leq b' \leq \frac{2}\pi \, L \,,
\]
which by~(\ref{stimab}) yields
\begin{equation}\label{8side32}
\frac{\pi}{20 L}\leq \frac b{b'} \leq  \frac 34\,\sqrt{14}\,L\,.
\end{equation}
Finally, we have to estimate $\sin\alpha'$. As is clear from Figure~\ref{Fig:26}, $\sin\alpha'$ is minimal if $\u{Q}_j\equiv \u{A}_1$, thus if $\alpha' = \angle{\u{O}}{\u{A}_1}{\u{D}}$. Since in this extremal case one has
\[
\alpha' = \arctan \frac{2\sqrt 3}{3} - \arctan \frac{\sqrt 3}{3}> 15^\circ\,,
\]
we obtain
\begin{equation}\label{8angle3}
\sin\alpha' \geq \frac 14 \,.
\end{equation}
Applying then once more~(\ref{lipschtriangle}), thanks to~(\ref{8side31}), (\ref{8side32}) and~(\ref{8angle3}) we get
\[\begin{split}
{\rm Lip} (\phi)
&\leq\frac {a'}a + \frac{2b'}{b\sin\alpha} + \frac{a'}{a\sin\alpha}
\leq 10\sqrt 2 L + \frac{3520}\pi\,L^2 + 880\sqrt 2 L^2\,,\\
{\rm Lip} (\phi^{-1})
&\leq\frac a{a'} + \frac{2b}{b'\sin\alpha'} +\frac{a}{a'\sin\alpha'}\leq 270L^2 + 6\sqrt{14} L+1080L^2\,.
\end{split}\]

To conclude, we have now to consider that case $i\neq 1$. Notice that now we cannot simply rely on the calculations done in Case~A as we did in Case~B, because this time $\u{O}$ is not the center of $\widehat{\u{\B}}$. Nevertheless, we still define $v\equiv u_i$ on $\D_i$, which is $212000L^4$ bi-Lipschitz by Step~VII, and again, to conclude, we have to send $\I_i$ onto $\u{A}_i\u{A}_{i+1}\u{O}$. Since the first set is the union of the triangles $Q_jQ_{j+1}O$, while the latter is the union of the triangles $\u{Q}_j\u{Q}_{j+1}\u{O}$, we define $v$ on $\I_i$ as the piecewise affine map which sends each triangle onto its correspondent one, and we only have to check the bi-Lipschitz constant of $v$ on $\I_1$. As usual, we set
\begin{align*}
a&=\ell\big(Q_jQ_{j+1}\big)\,, & 
b&=\ell\big(Q_j O\big) \,, & 
\alpha&=\angle O{Q_j}{Q_{j+1}}\,, \\
a'&=\ell\big(\u{Q}_j\u{Q}_{j+1}\big)\,,& 
b'&=\ell\big(\u{Q}_j\u{O}\big)\,, &
\alpha'&=\angle {\u{O}}{\u{Q}_j}{\u{Q}_{j+1}} \,.
\end{align*}
Let us now make the following observation. Even though the situation is not the same as in Case~A, as we pointed out above, the only difference is in fact that now $\u{O}$ is not the center of $\widehat{\u{\B}}$. And this difference clearly affects only $b'$ and $\alpha'$, thus~(\ref{8side11}), (\ref{stimab}) and~(\ref{8angle10}) already tell us
\begin{align*}
\frac{\sqrt{2}}{10L}\leq \frac a{a'} \leq 90L^2\,, && \frac 1 {10} \leq b \leq \frac {\sqrt{2}}2\,, && \frac 1{87L} \leq \alpha \leq \pi - \frac 1{87L}\,.
\end{align*}
Concerning $b'$, since any point $\u{Q}_j$ is below $\u{A}_1\u{A}_2$ by construction (recall that we are considering the case $i\neq 1$, so that $\u{Q}_j$ belongs to the side $\u{A}_i\u{A}_{i+1}$), we immediately deduce that
\[
b' \geq \lsegm{MO} =  \frac {\sqrt{3}} 6\, \ell\big(\u{A}_1\u{A}_2\big) \geq
\frac {2\sqrt{3}}{9L} \,,
\]
also using~(\ref{basicestimate}). On the other hand, by the assumption $d_1< -r/2$ and by construction it immediately follows that $\u{O}$ is below the center of $\widehat{\u{\B}}$, then keeping in mind~(\ref{estimr}) we have
\[
b' \leq r \leq \frac{2L}{\pi}\,.
\]
Finally, concerning $\alpha'$, it is clear by construction that both $\alpha'$ and $\pi-\alpha'$ are strictly bigger than $\angle{\u{A}_1}{\u{A}_2}{\u{O}}$, thus
\[
\sin \alpha' \geq \sin \angle{\u{A}_1}{\u{A}_2}{\u{O}} = \sin \bigg( \arctan \frac{\sqrt 3} 3\bigg) = \frac 1 2\,.
\]
Summarizing, we have
\begin{align*}
\frac{\sqrt{2}}{10L}\leq \frac a{a'} \leq 90 L^2\,, && \frac \pi{20 L}\leq \frac b{b'} \leq \frac {3\sqrt{6} L}4\,, && \sin\alpha \geq \frac{1}{88L}\,, && \sin\alpha' \geq \frac 12\,.
\end{align*}
Now, it is enough to use~(\ref{lipschtriangle}) for a last time to obtain
\[\begin{split}
{\rm Lip} (\phi)
&\leq\frac {a'}a + \frac{2b'}{b\sin\alpha} + \frac{a'}{a\sin\alpha}
\leq 5 \sqrt{2} L + \frac{3520}{\pi}\,L^2+440 \sqrt{2} L^2
\,,\\
{\rm Lip} (\phi^{-1})
&\leq\frac a{a'} + \frac{2b}{b'\sin\alpha'} +\frac{a}{a'\sin\alpha'}
\leq 90L^2 + 3\sqrt 6 L + 180 L^2\,
\end{split}\]
and then the proof of the first part of Theorem~\mref{main} is finally concluded.
\end{proof}

\bigstep{IX}{Definition of the smooth extension $v$.}

In this last step, we show the existence of the smooth extension $v$ of $u$, thus concluding the proof of Theorem~\mref{main}. The proof is an immediate corollary of the following recent result by Mora-Corral and the second author (see~\cite[Theorem~A]{MP}; in fact, we prefer to claim here only the part of that result that we need in this paper).

\begin{theorem}\label{carlos}
Let $v:\Omega\to\R^2$ be a (countably) piecewise affine homeomorphism, bi-Lipschitz with constant $L$. Then there exists a smooth diffeomorfism $\hat v:\Omega\to v(\Omega)$ such that $\hat v \equiv v$ on $\partial \Omega$, $\hat v$ is bi-Lipschitz with constant at most $70L^{7/3}$, and
\[
\|\hat v - v\|_{L^{\infty}(\Omega)} + \| D\hat v - Dv\|_{L^p(\Omega)} + \|\hat v^{-1} - v^{-1}\|_{L^{\infty}(v(\Omega))} + \| D\hat v^{-1} - Dv^{-1}\|_{L^p(v(\Omega))} \leq \eps\,.
\]
\end{theorem}

Having this result at hand, the conclusion of the proof of Theorem~\mref{main} is immediate.

\proofof{Theorem~\mref{main} (smooth extension)}
Let $v$ be an affine extension of $u$ having bi-Lipschitz constant at most $CL^4$, which exists thanks to the proof of the first part of the Theorem, Step~VIII. By Theorem~\ref{carlos}, there exists a map $\tilde v$ which is smooth, concides with $v$ on $\partial\D$, and has bi-Lipschitz constant at most $70 C^{7/3} L^{28/3}$. This map $\tilde v$ is a smooth extension of $u$ as required.
\end{proof}

\section{Proof of Theorem~\mref{maingen}\label{secmaingen}}

We now give the proof of Theorem~\mref{maingen}, which will be obtained from Theorem~\mref{main} by a quick extension argument. We will use the following simple geometric result.

\begin{lemma}\label{altrolavoro}
Let $\varphi: \partial\D \to \R^2$ be an $L$ bi-Lipschitz map. Then, for any $\eps>0$, there exists a piecewise affine map $\varphi_\eps:\partial\D\to\R^2$ which is $3L$ bi-Lipschitz and such that
\[
|\varphi(P) - \varphi_\eps(P) |\leq \eps \qquad \forall P\in\partial\D\,.
\]
\end{lemma}
The proof of this result can be found in the very recent paper~\cite[Lemma~2.3]{DP}. It is interesting to underline here that the main result of that paper, Theorem~\ref{al} below, uses our Theorem~\mref{main} in a crucial way.
\begin{theorem}[{\cite[Theorem~A]{DP}}]\label{al}
If $\Omega\subseteq\R^2$ is a bounded open set and $v:\Omega\to\Delta\subseteq\R^2$ is an $L$ bi-Lipschitz homeomorphism, then for all $\eps>0$ and $1\leq p<+\infty$ there exists a bi-Lipschitz homeomorphism $\omega:\Omega\to\Delta$, such that $\omega=v$ on $\partial\Omega$,
\[
\|\omega-v\|_{L^\infty(\Omega)}+\|\omega^{-1}-v^{-1}\|_{L^\infty(\Delta)}+\|D\omega-Dv\|_{L^p(\Omega)}+\|D\omega^{-1}-Dv^{-1} \|_{L^\infty(\Delta)}\leq \eps\,,
\]
and $\omega$ is either countably piecewise affine or smooth. In particular, the piecewise affine map can be taken $K_1 L^4$ bi-Lipschitz, and the smooth one $K_2 L^{28/3}$ bi-Lipschitz, being $K_1$ and $K_2$ purely geometric constants.
\end{theorem}

We can now show our Theorem~\mref{maingen}.
\proofof{Theorem~\mref{maingen}}
Let $u:\partial\D\to\R^2$ be an $L$ bi-Lipschitz map. Fix $\eps>0$ and apply Lemma~\ref{altrolavoro}, obtaining a $3L$ bi-Lipschitz and piecewise affine map $u_\eps:\partial\D\to\R^2$, with $\|u_\eps-u\|_{L^\infty(\partial\D)}\leq \eps$. Theorem~\mref{main}, applied to $u_\eps$, gives then an extension $v_\eps:\D\to\R^2$ which is $81CL^4$ bi-Lipschitz and satisfies $v_\eps=u_\eps$ on $\partial\D$. By a trivial compactness argument, there is a sequence $v_{\eps_j}$ which uniformly converges to an $81CL^4$ bi-Lipschitz function $v$. By construction, one clearly has that $v\equiv u$ on $\partial\D$, thus the thesis is obtained.
\end{proof}

\begin{corollary}\label{allast}
Under the assumptions of Theorem~\mref{maingen}, there exists an extension $\omega:\D\to\R^2$ of $u$ which is countably piecewise affine (resp. smooth), and which is $K_1 C''^4 L^{16}$ bi-Lipschitz (resp. $K_2 C''^4 L^{112/3}$ bi-Lipschitz).
\end{corollary}
\begin{proof}
This immediately follows from Theorem~\mref{maingen} and Theorem~\ref{al}. In fact, if $v$ is a $C''L^4$ bi-Lipschitz function given by Theorem~\mref{maingen}, then Theorem~\ref{al} provides us with a countable piecewise affine function $\omega$ which is very close to $v$, coincides with $v$ on $\partial\D$, and is $K_1\big(C''L^4\big)^4$ bi-Lipschitz, and with a smooth function $\widetilde\omega$, again very close to $v$, coinciding with $v$ on $\partial \D$ and $K_2 (C''L^4)^{28/3}$ bi-Lipschitz. These two function $\omega$ and $\widetilde\omega$ are the searched extensions of $u$.
\end{proof}

We conclude the paper with a last observation.

\begin{remark}\label{lastsec}
One could be not satisfied to pass from Theorem~\mref{maingen} to Corollary~\ref{allast} passing from $L^4$ to $L^{16}$ (resp. $L^{112/3}$). In fact, it is possible to modify the construction of Theorem~\mref{main} so as to directly obtain, in the case of a general $L$ bi-Lipschitz function $u:\partial\D\to\R^2$, a countably piecewise affine extension $v$ of $u$ which is $\widetilde CL^4$ bi-Lipschitz. And then, thanks to Theorem~\ref{carlos}, one would also get a smooth extension $v$ which is $70 \widetilde C^{7/3} L^{28/3}$ bi-Lipschitz.
\end{remark}

\section*{Acknowledgments}
We warmly thank Tapio Rajala for some fruitful discussions and for having pointed out to us the paper~\cite{Tukia}. Both the authors have been supported by the ERC Starting Grant n. 258685.

\end{document}